%% file: safem.tex
\documentclass[11pt,reqno]{NumPDEsArticle}

\usepackage[utf8]{inputenc}
\usepackage[T1]{fontenc}
\usepackage[english]{babel}

\RequirePackage[capitalize,nameinlink]{cleveref}[0.19]

\crefname{section}{section}{sections}
\crefname{subsection}{subsection}{subsections}
\Crefname{section}{Section}{Sections}
\Crefname{subsection}{Subsection}{Subsections}
\crefname{appendix}{Appendix}{Appendices}
\Crefname{figure}{Figure}{Figures}

\crefformat{equation}{\textup{#2(#1)#3}}
\crefrangeformat{equation}{\textup{#3(#1)#4--#5(#2)#6}}
\crefmultiformat{equation}{\textup{#2(#1)#3}}{ and \textup{#2(#1)#3}}
{, \textup{#2(#1)#3}}{, and \textup{#2(#1)#3}}
\crefrangemultiformat{equation}{\textup{#3(#1)#4--#5(#2)#6}}
{ and \textup{#3(#1)#4--#5(#2)#6}}{, \textup{#3(#1)#4--#5(#2)#6}}{, and \textup{#3(#1)#4--#5(#2)#6}}
\Crefformat{equation}{#2Equation~\textup{(#1)}#3}
\Crefrangeformat{equation}{Equations~\textup{#3(#1)#4--#5(#2)#6}}
\Crefmultiformat{equation}{Equations~\textup{#2(#1)#3}}{ and \textup{#2(#1)#3}}
{, \textup{#2(#1)#3}}{, and \textup{#2(#1)#3}}
\Crefrangemultiformat{equation}{Equations~\textup{#3(#1)#4--#5(#2)#6}}
{ and \textup{#3(#1)#4--#5(#2)#6}}{, \textup{#3(#1)#4--#5(#2)#6}}{, and \textup{#3(#1)#4--#5(#2)#6}}

\crefformat{enumi}{\textup{#2#1#3}}
\crefrangeformat{enumi}{\textup{#3#1#4--#5#2#6}}
\crefmultiformat{enumi}{\textup{#2#1#3}}{ and \textup{#2#1#3}}
{, \textup{#2#1#3}}{, and \textup{#2#1#3}}
\crefrangemultiformat{enumi}{\textup{#3#1#4--#5#2#6}}%
{ and \textup{#3#1#4--#5#2#6}}{, \textup{#3#1#4--#5#2#6}}{, and \textup{#3#1#4--#5#2#6}}

\crefdefaultlabelformat{#2\textup{#1}#3}

\AddToHook{env/lemma/begin}{\crefalias{theorem}{lemma}}
\AddToHook{env/corollary/begin}{\crefalias{theorem}{corollary}}
\AddToHook{env/proposition/begin}{\crefalias{theorem}{proposition}}
\AddToHook{env/definition/begin}{\crefalias{theorem}{definition}}
\AddToHook{env/remark/begin}{\crefalias{theorem}{remark}}

\usepackage{csquotes}
\usepackage{enumitem}
\usepackage{amsmath,amssymb}
\usepackage{orcidlink}
\usepackage[dvipsnames]{xcolor}

\usepackage[basicdelimiters]{NumPDEsMacros}
\usepackage{LietzMacros}

\makeatletter
\renewcommand{\eqreff}{\@ifstar\@eqreffstar\@eqreff}
\def\@eqreff#1#2{\stackrel{\cref{#1}}{#2}}
\def\@eqreffstar#1#2{\stackrel{\mathclap{\cref{#1}}}{#2}}
\makeatother

\usepackage{tabularray}
\UseTblrLibrary{booktabs}
\SetTblrInner[booktabs]{
  colsep=5pt,
  leftsep=5pt,
  rightsep=5pt,
  abovesep = 0ex,
  belowsep = 0ex,
  rowsep = 0pt
}

\usepackage{tikz}
\usetikzlibrary{arrows, positioning, shapes, matrix, calc, fit, backgrounds, babel, external}
\tikzset{reference/.style={thick,dashed}}

\usepackage{pgfplots}
\usepackage{pgfplotstable}
\pgfplotsset{
  compat=newest,
  every axis/.style={scale only axis},
  grid style={densely dotted, semithick},
}

\newlength{\indicatorWidth}
\newlength{\meshWidth}
\newlength{\convergenceWidth}
\title{Unconditional full linear convergence and quasi-optimal complexity of smoothed adaptive finite element methods}

\author{Philipp Bringmann~\orcidlink{0000-0002-4546-5165}}
\author{Christoph Lietz~\orcidlink{0009-0005-1857-0954}}
\author{Dirk Praetorius~\orcidlink{0000-0002-1977-9830}}
\address{TU Wien, Institute of Analysis and Scientific Computing, Wiedner Hauptstrasse 8--10, 1040, Austria}
\email{philipp.bringmann@asc.tuwien.ac.at}
\email{christoph.lietz@asc.tuwien.ac.at \quad {\normalfont{(corresponding author)}}}
\email{dirk.praetorius@asc.tuwien.ac.at}



\thanks{
  This research was funded in whole or in part by the Austrian Science Fund (FWF)
  [\href{https://www.fwf.ac.at/en/research-radar/10.55776/F65}{10.55776/F65},
      \href{https://www.fwf.ac.at/en/research-radar/10.55776/PAT3699424}{10.55776/PAT3699424}, and
      \href{https://www.fwf.ac.at/en/research-radar/10.55776/PAT3446525}{10.55776/PAT3446525}].
  Additionally, Christoph Lietz is supported by the Vienna School of Mathematics.
}
\begin{document}

\maketitle

\begin{abstract}
  We present the first rigorous convergence analysis of the smoothed adaptive finite element method (S-AFEM) proposed in~[Mulita, Giani, Heltai: SIAM J.
  \ Sci.\ Comput.\ 43, 2021].
  S-AFEM modifies the classical adaptive finite element method (AFEM) by performing accurate discrete solves only on periodically determined mesh levels, while the intermediate levels employ a fixed number of cheap smoothing iterations.
  Numerical experiments in that work showed that this strategy generates adapted meshes comparable to those of AFEM at substantially lower computational cost.
  In this paper, we prove unconditional full R-linear convergence of a suitable quasi-error quantity and, for sufficiently small adaptivity parameters, optimal convergence rates with respect to the overall computational cost.
  The analysis requires only a mild uniform stability assumption on the employed smoother, satisfied by standard methods such as Richardson, Gauss--Seidel, conjugate gradient, and multigrid schemes.
  Our results apply to general second-order linear elliptic PDEs and show that S-AFEM retains all desired abstract convergence guarantees of AFEM while reducing the cumulative computational time.
  Numerical experiments validate the theory, analyze runtime performance, and underline the potential of S-AFEM for speed-up in AFEM computations.
\end{abstract}

\input{safemBody.tex}

{
  \sloppy
  \printbibliography
}

\appendix
\crefalias{section}{appendix}
\crefalias{subsection}{appendix}
\input{safemAppendix.tex}

\end{document}

%% file: safemBody.tex

\section{Introduction}\label{sec:intro}

\emph{Motivation and overview.}\quad
The ultimate goal of any numerical method for partial differential equations (PDEs) is to approximate the exact solution with quasi-minimal computational cost.
Adaptive finite element methods (AFEMs) achieve this by iteratively steering the mesh refinement based on \textsl{a~posteriori} refinement indicators and a marking strategy.
The smoothed adaptive finite element method (S-AFEM), introduced in~\cite{HeltMul}, is a practice-oriented variant of AFEM designed to significantly reduce the number of expensive discrete solves within the adaptive loop.
It is motivated by the key observation that a few iterations of a cheap smoother already yield refinement indicators close to those obtained with the exact discrete solution; see \cref{fig:indicators}.
\begin{figure}
  \centering
  \input{plots/indicators/indicatorsInitial.tex}
  \hfill
  \input{plots/indicators/indicatorsSmoothed.tex}
  \hfill
  \input{plots/indicators/indicatorsExact.tex}
  \caption{Refinement indicators for a Poisson problem for an initial approximation (left), three Gauss--Seidel iterations (center), and the exact discrete solution (right).}
  \label{fig:indicators}
\end{figure}
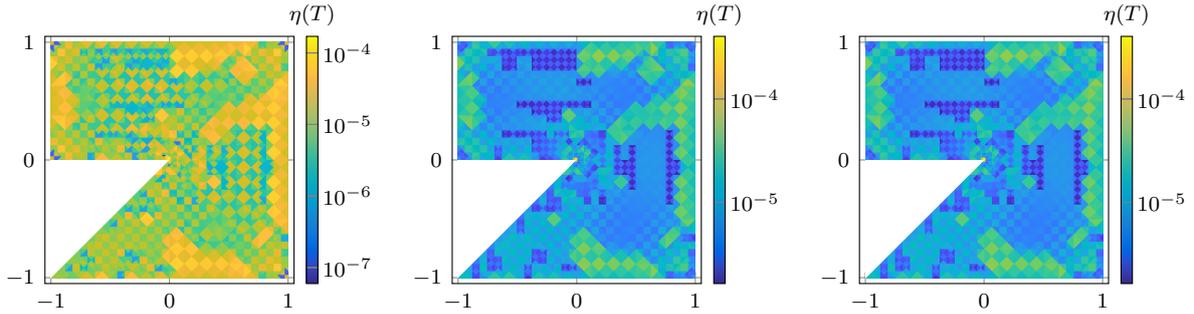
S-AFEM modifies the classical AFEM loop in that accurate discrete solutions are computed only on periodically determined levels of the mesh hierarchy.
The intermediate levels employ a fixed number of inexpensive smoothing steps.
Extensive numerical experiments in~\cite{HeltMul} demonstrate that this strategy produces mesh sequences that closely resemble those of classical AFEM, but at significantly reduced computational cost.
A rigorous convergence analysis of S-AFEM, however, has been missing so far.

Algorithmically, the method inserts a cheap \texttt{SMOOTH} module into the classical AFEM loop, replacing \texttt{SOLVE} on intermediate levels; see \cref{fig:loops}.
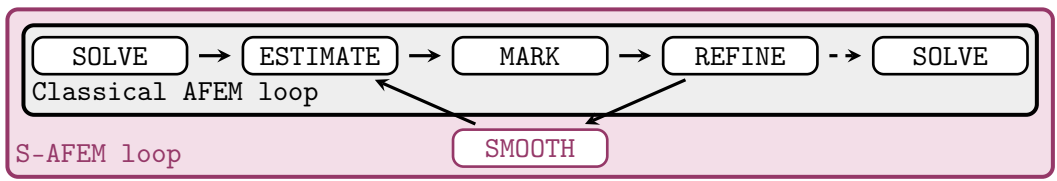
\begin{figure}
  \centering
  \input{figures/loops.tex}
  \caption{Illustration of adaptive refinement loops.}
  \label{fig:loops}
\end{figure}
The classical AFEM loop in \cref{fig:loops} generates a sequence of meshes $(\TT_\ell)_{\ell\in\N_0}$ and corresponding finite element approximations $(u_\ell)_{\ell\in\N_0}$ to the exact continuous solution $u^\star$.
On each mesh~$\TT_\ell$, the \texttt{SOLVE} module computes an \emph{accurate} approximation $u_\ell$ of the exact discrete solution $u_\ell^\star$.
The \texttt{ESTIMATE} module then evaluates computable refinement indicators $\eta_\ell(T, u_\ell)\geq 0$ for each element $T\in\TT_\ell$ to estimate the local contribution to the overall \textsl{a~posteriori} error estimator~$\eta_\ell(u_\ell)$.
Based on these indicators, the $\texttt{MARK}$ module selects a subset of elements with comparatively large indicators for refinement, and $\texttt{REFINE}$ generates a new mesh $\TT_{\ell+1}$ by refining (at least) these marked elements.

S-AFEM alters this algorithm by fixing a period $L\in\N$ and performing \texttt{SOLVE} only on every $L$-th mesh, resulting in the S-AFEM loop illustrated in \cref{fig:loops}.
On the intermediate mesh levels, \texttt{SOLVE} is replaced by the \texttt{SMOOTH} module consisting of a small number of iterations of a computationally very cheap smoother.
Although such smoothers efficiently damp only selected error components and do not aim to fully solve the discrete problem, the resulting refinement indicators remain effective in steering the mesh refinement (cf.~\cref{fig:indicators}).

\emph{State of the art.}\quad
The mathematical analysis of the classical AFEM loop in \cref{fig:loops} has developed substantially over the past three decades; see, for instance,~\cite{Doer96, MNS00, BDD04, Ste07, Ste08, CKNS08, CN12, FFP14} for linear elliptic PDEs,~\cite{Vee02, DK08, BDK12, GMZ12} for certain nonlinear PDEs, and~\cite{CFPP14} for an abstract axiomatic framework.
These contributions establish plain convergence~\cite{Doer96, MNS00, Vee02, DK08, GMZ12} as well as optimal convergence rates with respect to the number of degrees of freedom~\cite{BDD04, CKNS08, CN12, BDK12, GMZ12, FFP14}.
Since AFEM proceeds iteratively, however, optimality should be measured in terms of the overall computational cost (resp.~cumulative computation time) rather than the number of degrees of freedom alone.
This notion, referred to as \emph{quasi-optimal complexity}, was originally developed for adaptive wavelet methods~\cite{CDD01, CDD03} and later extended to AFEM in~\cite{Ste07, CG12} under the assumption that the inexact solver is sufficiently accurate.

The works~\cite{GHPS21,BFMPS24} employ a uniformly contractive iterative solver computing inexact approximations $u_\ell^k\approx u_\ell^\star$ on each mesh $\TT_\ell$ for solver iterations $k=1,\dots,{\kk}[\ell]$.
The solver is terminated using a stopping criterion that controls the algebraic error relative to the discretization error via a parameter $\lambda>0$.
For general second-order linear elliptic PDEs,~\cite{BFMPS24} establishes \emph{unconditional full R-linear convergence} of the quasi-error $\Hr_\ell^k\coloneq\enorm{u_\ell^\star-u_\ell^k}+\eta_\ell(u_\ell^\star)$.
That is, for arbitrary $\lambda>0$, the quasi-error $\Hr_\ell^k$ converges R-linearly with respect to the lexicographic order $\abs{\ell, k}\in\N_0$ for index pairs $(\ell,k)\in\N_0^2$ appearing in the algorithm (collected in the index set $\QQ$), i.e.,
\begin{equation}\label{eq:fullrlinear}
  \Hr_\ell^k\leq\Clin\qlin^{\abs{\ell,k}-\abs{\ell^\prime,k^\prime}}\,\Hr_{\ell^\prime}^{k^\prime}\quad
  \text{for all }(\ell^\prime,k^\prime),(\ell,k)\in\QQ\text{ with }\abs{\ell^\prime,k^\prime}\leq\abs{\ell,k}.
\end{equation}
Moreover,~\cite{GHPS21} identifies~\cref{eq:fullrlinear} as a key ingredient for proving quasi-optimal complexity for sufficiently small $\lambda>0$.
Formally, quasi-optimal complexity asserts decay of the quasi-error at the largest possible convergence rate $s>0$ with
\begin{equation*}
  \sup_{(\ell,k)\in\QQ}\Bigg(\sum_{\substack{(\ell^\prime, k^\prime)\in\QQ \\ \abs{\ell^\prime, k^\prime}\leq \abs{\ell, k}}}\#\TT_{\ell^\prime}\Bigg)^s\,\Hr_\ell^k<+\infty.
\end{equation*}
Since all modules of the adaptive loop in \cref{fig:loops} can usually be realized with linear complexity $\OO(\#\TT_\ell)$, this notion indeed guarantees optimal convergence rates with respect to the overall computational cost, i.e., optimal complexity.

\emph{Main results and contributions.}\quad
In the present work, we establish the first convergence analysis of the S-AFEM loop shown in \cref{fig:loops} with fixed period~$L \in \N$, which we call the S-AFEM algorithm.
We prove that, despite including smoothing on intermediate mesh levels, S-AFEM still satisfies all abstract convergence guarantees known for AFEM with inexact solvers.
In particular, we show that S-AFEM ensures \emph{unconditional full R-linear convergence}~\cref{eq:fullrlinear}.
Moreover, by introducing a novel cardinality-control step that limits the number of marked elements on intermediate mesh levels, we prove \emph{quasi-optimal complexity} of S-AFEM for sufficiently small adaptivity parameters.
Our analysis requires only a mild \emph{uniform stability} assumption on the smoother used in \texttt{SMOOTH}.
This assumption is satisfied by standard smoothers such as Richardson, Gauss--Seidel, and conjugate gradient (CG) iterations, as well as by more advanced methods like preconditioned CG and multigrid methods.
On the periodically occurring \texttt{SOLVE} levels, we employ a uniformly contractive iterative solver with the stopping criterion from~\cite{GHPS21}.
Crucially, the analysis imposes no restrictions on the choice of the period~$L$ or on the number of smoothing iterations performed on intermediate levels.
While~\cite{LZ21, LS25} construct \textsl{a~posteriori} error estimators based on smoothers, we steer the mesh refinement with the standard residual-based estimator.

\emph{Outline.}\quad
The remainder of this work is structured as follows. \Cref{sec:prelim} presents the model problem of a general second-order linear elliptic PDE, the mesh-refinement framework, and the discrete formulation.
It also recalls the residual-based error estimator~\cite{ao2000,v2013}, the axioms of adaptivity~\cite{CFPP14}, and the quasi-orthogonality from~\cite{Fei22}.
\Cref{sec:smoothers} introduces a broad class of examples for uniformly stable smoothers and shows how Zarantonello symmetrization enables the construction of such smoothers also for non-symmetric problems.
\Cref{sec:algo} formalizes the S-AFEM loop shown in \cref{fig:loops} in \cref{algorithm:SAFEM}.
The main contributions follow thereafter: \Cref{sec:rlinear} establishes full R-linear convergence of \cref{algorithm:SAFEM}, first unconditionally in \cref{th:fullrlinear} and later, in \cref{th:Condfullrlinear}, under an even weaker assumption on the smoother but for a sufficiently small solver stopping parameter $\lambda>0$.
\Cref{sec:optrates} shows in \cref{th:optrates} that the algorithm attains quasi-optimal complexity.
Finally, \cref{sec:numExp} presents numerical experiments that confirm the theoretical results and analyze the runtime performance of S-AFEM.

\section{Preliminaries}\label{sec:prelim}

\subsection{Model problem}

Let $\Omega\subset\R^d, d\geq 1$ be a bounded polyhedral Lipschitz domain.
Let $\vec{A}\in W^{1,\infty}(T; \R_\text{sym}^{d\times d})$ and $\vec{f}\in W^{1,\infty}(T; \R^d)$ be piecewise Lipschitz-continuous with respect to each element $T\in\TT_0$ of an initial triangulation $\TT_0$ of $\Omega$ (cf.~\cref{subsec:mesh}).
Moreover, let $\vec{b}\in L^\infty(\Omega; \R^d)$, $c\in L^\infty(\Omega)$, and $f\in L^2(\Omega)$.
The second-order linear elliptic model problem seeks $u^\star\in H_0^1(\Omega)$ such that
\begin{equation}\label{eq:PDE}
  -\div(\vec{A}\nabla u^\star)+ \vec{b}\cdot\nabla u^\star+cu^\star = f-\div\vec{f}\;\:
  \text{in }\Omega
  \quad\text{subject to}\quad
  u^\star = 0\;\:\text{on }\partial\Omega.
\end{equation}
The bilinear form
\begin{equation}\label{eq:bilinearform}
  b(v,w)\coloneq \scalarproduct{\vec{A}\nabla v}{\nabla w}_{L^2(\Omega)}
  +\scalarproduct{\vec{b}\cdot\nabla v+cv}{w}_{L^2(\Omega)}
\end{equation}
is assumed to fit into the Lax--Milgram setting, i.e., there exist constants $\Cell, \Cbnd>0$ that may depend on $\Omega$ and the coefficients $\vec{A}, \vec{b}, c$ such that, for all $v,w\in H^1_0(\Omega)$,
\begin{equation}\label{eq:bilinear_ellbnd}
  \Cell\,\|v\|_{H^1(\Omega)}^2\leq b(v,v)\quad\text{and}\quad b(v,w)\leq\Cbnd\,\|v\|_{H^1(\Omega)}
  \,\|w\|_{H^1(\Omega)}.
\end{equation}
Due to \cref{eq:bilinear_ellbnd}, the weak formulation of~\cref{eq:PDE} has a unique solution $u^\star\in H_0^1(\Omega)$ satisfying%
\begin{equation}\label{eq:weakform}
  b(u^\star, v)=\scalarproduct{f}{v}_{L^2(\Omega)}+\scalarproduct{\vec{f}}{\nabla v}_{L^2(\Omega)}
  \eqcolon F(v)\quad\text{for all }v\in H_0^1(\Omega).
\end{equation}
The symmetric part $a(\cdot, \cdot)$ of $b(\cdot, \cdot)$ is defined via $a(v,w)\coloneq[b(v,w)+b(w,v)]/2$.
It satisfies the same ellipticity and continuity bounds as in~\cref{eq:bilinear_ellbnd} with $b(\cdot, \cdot)$ replaced by $a(\cdot, \cdot)$ and identical constants $\Cell$ and $\Cbnd$.
Hence, $a(\cdot, \cdot)$ defines an equivalent scalar product on $H_0^1(\Omega)$ and the induced energy norm $\enorm{v}\coloneq a(v,v)^{1/2}=b(v,v)^{1/2}$ is equivalent to $\|\cdot\|_{H^1(\Omega)}$ on $H_0^1(\Omega)$.

\subsection{Triangulations and refinement}\label{subsec:mesh}

Let $\TT_0$ be an initial conforming triangulation of $\Omega$ into compact simplices.
Newest-vertex bisection (NVB) is employed as a refinement strategy~\cite{Ste08, KPP13, AFFKP13, DGS25}, yielding uniformly shape-regular meshes.
For any conforming triangulation $\TT_H$ and any $\MM_H\subseteq\TT_H$, let $\refine(\TT_H,\MM_H)\eqcolon\TT_h$ denote the coarsest NVB refinement of $\TT_H$ such that at least every marked element is refined, i.e., $\MM_H \subseteq \TT_H\setminus \TT_h$.
Define $\refine(\TT_H)$ to be the set of all conforming triangulations obtainable from $\TT_H$ by finitely many NVB refinements, and $\T\coloneq\refine(\TT_0)$.

\subsection{Discretization}

Let $p\in\N$ be a fixed polynomial degree.
For each $\TT_H\in\T$, define the conforming
Lagrange finite element space $\XX_H\subset H_0^1(\Omega)$ of degree $p$ by
\begin{equation}\label{eq:femspace}
  \XX_H\coloneq\set{v_H\in H_0^1(\Omega) \given \forall T\in\TT_H\colon v_H|_T
    \text{ is a polynomial of total degree $\leq p$}}.
\end{equation}
This ensures nestedness of the discrete spaces, i.e., $\TT_h\in\refine(\TT_H)$ implies $\XX_H\subseteq\XX_h$.
For every $\TT_H\in\T$, there exists a unique discrete Galerkin
solution $u_H^\star\in\XX_H$ solving
\begin{equation}\label{eq:galeq}
  b(u_H^\star,v_H)=F(v_H)\quad\text{for all }v_H\in\XX_H.
\end{equation}

\subsection{Error estimator}

For each mesh $\TT_H\in\T$ and for all $T\in\TT_H$ and $v_H\in\XX_H$, we employ the
residual-based error estimator~\cite{ao2000,v2013} with local contributions
\begin{equation}\label{eq:resestim}
  \begin{aligned}
    \eta_H(T,v_H)^2
     & \coloneq\abs{T}^{2/d}\norm{-\div(\vec{A}\nabla v_H-\vec{f})
      +\vec{b}\cdot\nabla v_H+cv_H-f}^2_{L^2(T)}
    \\
     & \qquad +\abs{T}^{1/d}\norm{\jump{(\vec{A}\nabla v_H-\vec{f})\cdot\vec{n}}}^2_{L^2(\partial T\cap\Omega)},
  \end{aligned}
\end{equation}
where $\vec{n}$ denotes the outer unit normal vector on $\partial T$ and $\jump{\cdot}$ denotes the jump across ($d\!-\!1$)-dimensional hyperfaces. For any $\UU_H\subseteq\TT_H$, set
\begin{equation}\label{eq:resestimSum}
  \eta_H(\UU_H,v_H)\coloneq\biggl(\sum_{T\in\UU_H}\eta_H(T,v_H)^2\biggr)^{1/2}
  \quad\text{and}\quad\eta_H(v_H)\coloneq\eta_H(\TT_H,v_H).
\end{equation}

Following~\cite{CFPP14}, the estimator~\cref{eq:resestim} satisfies the \emph{axioms of adaptivity} stated below.

\begin{proposition}[axioms of adaptivity]
  There exist $\Cstab, \Crel, \Cdrel>0$, $0<\qred<1$, and $\Cmon\geq 1$ such that,
  for any $\TT_H\in\T$, any $\TT_h\in\refine(\TT_H)$, any $\UU_H\subseteq\TT_H\cap\TT_h$, and arbitrary $v_H\in\XX_H$, $v_h\in\XX_h$, the following properties hold:
  \begin{enumerate}[font=\upshape]
    \renewcommand{\theenumi}{(A\arabic{enumi})}
    \item[(A1)]\refstepcounter{enumi}\label{A1} \emph{Stability:}
      \quad$\abs{\eta_h(\UU_H, v_h)-\eta_H(\UU_H,v_H)}\leq\Cstab\,\enorm{v_h-v_H}$.
    \item[(A2)]\refstepcounter{enumi}\label{A2} \emph{Reduction:}
      \quad$\eta_h(\TT_h\setminus\TT_H, v_H)\leq\qred\,\eta_H(\TT_H\setminus\TT_h,v_H)$.
    \item[(A3)]\refstepcounter{enumi}\label{A3}\emph{Reliability:}
      \quad$\enorm{u^\star-u_H^\star}\leq\Crel\,\eta_H(u_H^\star)$.
      \renewcommand{\theenumi}{(A3$^{+}$)}
    \item[(A3$^+$)]\refstepcounter{enumi}\label{A3p}\emph{Discrete reliability:}
      \quad$\enorm{u_h^\star-u_H^\star}\leq\Cdrel\,\eta_H(\TT_H\setminus\TT_h, u_H^\star)$.
      \renewcommand{\theenumi}{(QM)}
    \item[(QM)]\refstepcounter{enumi}\label{QM}\emph{Quasi-monotonicity:}
      \quad$\eta_h(u_h^\star)\leq\Cmon\,\eta_H(u_H^\star)$.
  \end{enumerate}
  The constant $\Crel$ depends only on the dimension~$d$, uniform shape regularity, and the coefficients $\vec{A}, \vec{b}, c$, while $\Cstab$ and $\Cdrel$ additionally depend on the polynomial degree~$p$.
  Moreover, there hold $\qred\leq2^{-1/(2d)}$ and $\Cmon\leq1+\Cstab\Cdrel$.
\end{proposition}

To compensate for the absence of Pythagoras-type identities in non-symmetric problems, various quasi-orthogonality results have been developed; see~\cite{MN05, CN12, FFP14, BHP17}.
We recall below the relaxed form~\cref{A4} from~\cite{Fei22}, shown to hold \textsl{a~priori} for uniformly inf-sup stable problems and, in particular, for the present Lax--Milgram setting.

\begin{proposition}
  There exist constants $\Corth>0$ and $0<\delta\leq 1$ such that, for any sequence
  $\XX_\ell\subseteq\XX_{\ell+1}\subseteq H_0^1(\Omega)$ of nested finite-dimensional subspaces, the exact discrete solutions $u_\ell^\star\in\XX_\ell$ to problem~\cref{eq:galeq} satisfy, for all $\ell,N\in\N_0$,
  \begin{enumerate}[font=\upshape]
    \renewcommand{\theenumi}{(A\arabic{enumi})}
    \setcounter{enumi}{3}
    \item[(A4)]\refstepcounter{enumi}\label{A4} \emph{Quasi-orthogonality:}
      \quad$\sum_{\ell^\prime=\ell}^{\ell+N}\enorm{u_{\ell^\prime+1}^\star-u_{\ell^\prime}^\star}^2
        \leq\Corth(N+1)^{1-\delta}\enorm{u^\star-u_\ell^\star}^2$.
  \end{enumerate}
  The constants $\Corth$ and $\delta$ depend only on $\Cbnd$ and $\Cell$ from~\cref{eq:bilinear_ellbnd}.
\end{proposition}

\section{Discrete iterative smoothers}\label{sec:smoothers}

\subsection{Uniform stability and contraction}

A smoother for the discrete variational problem~\cref{eq:galeq} is represented by an iteration operator $\Psi_H\colon\XX_H\to\XX_H$.
We consider the following two mesh-size-independent properties of such operators.

\begin{enumerate}[font = \upshape]
  \renewcommand{\theenumi}{(US)}
  \item[(US)]\refstepcounter{enumi}\label{eq:stablesolver} {Uniform stability:}
    There exists a constant $\Calg>0$, independent of $\TT_H\in\T$, such that, for all $v_H\in\XX_H$, it holds that $\enorm{u^\star_H-\Psi_H(v_H)}\leq \Calg\,\enorm{u^\star_H-v_H}$.
    \renewcommand{\theenumi}{(UC)}
  \item[(UC)]\refstepcounter{enumi}\label{eq:contractivesolver} {Uniform contraction:}
    There exists a constant $0<\qalg<1$, independent of $\TT_H\in\T$, such that~\cref{eq:stablesolver} holds with $\Calg\coloneq\qalg<1$.
\end{enumerate}

\subsection{Uniformly stable smoothers for symmetric PDEs}\label{sec:stableSym}\label{subsec:smoothers:sym}

If the PDE~\cref{eq:PDE} is symmetric,
classic iterative methods for solving~\cref{eq:galeq} contract the error in the energy norm,
i.e., letting $\Psi_H\colon\XX_H\to\XX_H$ denote the iteration operator, there exists a contraction factor $0<q_H<1$, possibly depending on $\TT_H\in\T$, such that
\begin{equation}\label{eq:energyContraction}
  \enorm{u^\star_H-\Psi_H(v_H)}\leq q_H\,\enorm{u^\star_H-v_H}\quad\text{for all }v_H\in\XX_H.
\end{equation}
Although such $\Psi_H$ are not uniformly contractive~\cref{eq:contractivesolver} in general, contraction~\cref{eq:energyContraction} immediately implies stability~\cref{eq:stablesolver} with $\Calg=1$.
The mesh-dependent contraction~\cref{eq:energyContraction} holds for essentially all classical smoothers; see the examples given below:

\begin{itemize}
  \item \emph{Richardson iteration} for sufficiently small damping~\cite[Section~3.5.1]{H16}.
  \item \emph{Damped Jacobi iteration} for sufficiently small damping~\cite[Theorem~6.11]{H16}.
  \item \emph{Gauss--Seidel iteration}~\cite[Theorem~3.39]{H16}.
  \item \emph{Conjugate gradient method (CG)}~\cite[Theorem~11.3.3]{GV13}.
  \item \emph{Preconditioned CG (PCG)} with an SPD preconditioner~\cite[Section~11.5.2]{GV13}, such as incomplete Cholesky preconditioners, provided the factorization terminates successfully, or block Jacobi preconditioners.
\end{itemize}

\subsection{Uniformly stable smoothers for non-symmetric PDEs}\label{subsec:smoothers:nonsym}

To obtain uniformly stable~\cref{eq:stablesolver} smoothers for non-symmetric problems, we symmetrize~\cref{eq:galeq} and apply a stable smoother to the resulting symmetric system.
Following~\cite{BIM24}, we use a single Zarantonello step~\cite{Zar60} for symmetrization.
For $\delta>0$, the Zarantonello operator $\ZZ_H^\delta\colon\XX_H\to\XX_H$ maps $v_H\in\XX_H$ to the unique solution $\ZZ_H^\delta v_H\in\XX_H$ of
\begin{equation}\label{eq:zarantonello}
  a(\ZZ_H^\delta v_H, w_H)=a(v_H, w_H)+\delta\big[F(w_H)-b(v_H, w_H)]\quad\text{for all $w_H\in\XX_H$.}
\end{equation}
As $a(\cdot, \cdot)$ is an equivalent scalar product on $H_0^1(\Omega)$, the Riesz theorem ensures that $\ZZ_H^\delta$ is well-defined.
The proof of~\cite[Theorem 25.B]{Zei90}, applied to $\XX_H$ endowed with $a(\cdot, \cdot)$, shows that $\ZZ_H^\delta$ is Lipschitz continuous in the energy norm, i.e., for all $v_H, w_H\in\XX_H$,
\begin{equation}\label{eq:ZarLipschitz}
  \enorm{\ZZ_H^\delta v_H - \ZZ_H^\delta w_H}\leq C[\delta]\,\enorm{v_H-w_H}\quad\text{with}\quad C[\delta]=\big[1-\delta(2-\delta\Cbnd^2/\Cell^2)\big]^{1/2},
\end{equation}
where $\Cell, \Cbnd$ are the constants from~\cref{eq:bilinear_ellbnd}.
While $C[\delta]\geq 0$ for any choice of $\delta$, we note that even $0\leq C[\delta]<1$ for $0<\delta<2\Cell^{2}/\Cbnd^2$.
The reasoning from~\cite[Section~2]{BIM24} shows that applying a uniformly stable smoother to~\cref{eq:zarantonello} yields a uniformly stable smoother for the original non-symmetric system~\cref{eq:galeq}.
The proof is provided in \cref{sec:appendixA}.

\begin{proposition}[uniform stability of inexact Zarantonello symmetrization]\label{prop:ZarStable}
  Let $\delta>0$ and let $J\in\N_0$ be a number of smoother iterations.
  For every $\TT_H\in\T$ and $v_H\in\XX_H$, let $\Theta_H(v_H)\colon\XX_H\to\XX_H$ be the iteration operator of a smoother for the symmetric problem~\cref{eq:zarantonello}, assumed to be uniformly stable in the sense that there exists a constant $\const{C}{alg}>0$, independent of the mesh size of $\TT_H$ and $v_H$, such that
  \begin{equation}\label{prop:ZarStable:eq1}
    \enorm{\ZZ_H^\delta v_H-\Theta_H(v_H)(w_H)}\leq \const{C}{alg}\,\enorm{\ZZ_H^\delta v_H-w_H}\quad\text{for all $w_H\in\XX_H$}.
  \end{equation}
  Then, the operator $\Psi_H\colon\XX_H\to\XX_H$ defined by $\Psi_H(v_H)\coloneq(\Theta_H(v_H))^{J}(v_H)$ is a uniformly stable~\cref{eq:stablesolver} smoother for the non-symmetric system~\cref{eq:galeq}, i.e.,
  \begin{equation}\label{prop:ZarStable:eq2}
    \enorm{u^\star_H-\Psi_H(v_H)}\leq \big[C[\delta]+\const{C}{alg}^{J}(C[\delta]+1)\big]\,\enorm{u^\star_H-v_H}\quad\text{for all $v_H\in\XX_H$,}
  \end{equation}
  where $C[\delta]$ is the constant from~\cref{eq:ZarLipschitz}.
  Moreover, provided that $\const{C}{alg}<1$ and $\delta<2\Cell^{2}/\Cbnd^2$, there exists $J_0\in\N$ so that $\Psi_H$ is uniformly contractive~\cref{eq:contractivesolver} for $J\geq J_0$.
\end{proposition}

\section{Smoothed AFEM algorithm}\label{sec:algo}

The S-AFEM loop in \cref{fig:loops} distinguishes two types of refinement levels in the mesh hierarchy: \emph{solve levels} and \emph{intermediate levels}.
Solve levels correspond to \texttt{SOLVE}, while intermediate levels correspond to \texttt{SMOOTH}.
The spacing between two consecutive solve levels is determined by a fixed period $L\in\N$, so that the set of solve levels is given by $L\cdotS\N_0\coloneq\set{\ell\in\N_0\given\ell=jL\text{ for some $j\in\N_0$}}$.
Two (generally different) iterative methods are employed.
On solve levels, a uniformly contractive~\cref{eq:contractivesolver} solver $\Phi_H\colon\XX_H\to\XX_H$ is applied and stopped according to a criterion balancing algebraic and discretization errors; see~\cite{BJR95, CS07, GHPS21}.
For symmetric problems, examples include the $hp$-robust geometric multigrid method from~\cite{IMPS24}.
For non-symmetric problems, one may apply a uniformly contractive solver to the Zarantonello symmetrized problem~\cref{eq:zarantonello}, with \cref{prop:ZarStable} guaranteeing~\cref{eq:contractivesolver} for suitable parameter choices.
On intermediate levels, a uniformly stable~\cref{eq:stablesolver} smoother $\Psi_H\colon\XX_H\to\XX_H$ performs exactly $K\in\N$ iterations.
In addition to standard Dörfler marking, a cardinality control step limits the number of marked elements on the intermediate levels.
\cref{algorithm:SAFEM} below formalizes this procedure.

\begin{algorithm}[S-AFEM]\label{algorithm:SAFEM}
  Given an initial mesh $\TT_0$, a number of smoothing steps $K\in\N$, a period $L\in\N$, adaptivity
  parameters $0<\theta\leq 1,\Cmark,\Ccard\geq 1$, a stopping~parameter $\lambda>0$,
  and an initial guess $u_0^0\in\XX_0$, iterate the following steps for $\ell=0, 1, 2,\dots$:
  \begin{enumerate}[label=(\Roman*), font = \upshape]
    \item\label{algorithm:SAFEM:1:solve} \textbf{Case 1:} If $\ell\in L\cdotS\N_0$, proceed with \texttt{\upshape SOLVE \& ESTIMATE:}
      \begin{enumerate}[label={}]
        \item For all $k=1,2,3,\dots$, repeat the following
          steps~\crefrange{algorithm:SAFEM:item_a}{algorithm:SAFEM:item_b} until
          \begin{equation}\label{algorithm:SAFEM:solverstop}
            \enorm{u_\ell^k-u_\ell^{k-1}}\leq\lambda\,\eta_\ell(u_\ell^k).
          \end{equation}
          \begin{enumerate}[label=(\alph*), ref = {\rm (\alph*)}, font = \upshape]
            \item\label{algorithm:SAFEM:item_a}
              Compute $u_\ell^k\coloneq\Phi_\ell(u_\ell^{k-1})$ with one step of the iterative solver.
            \item\label{algorithm:SAFEM:item_b}
              Compute refinement indicator $\eta_\ell(T, u_\ell^k)$ for all $T\in\TT_\ell$.
          \end{enumerate}
        \item Upon termination of the preceding $k$-loop, define ${\kk}[\ell]\coloneq k$.
      \end{enumerate}
    \item\label{algorithm:SAFEM:1:smooth} \textbf{Case 2:} If $\ell\notin L\cdotS\N_0$, proceed with:
      \begin{enumerate}[label={}]
        \item \texttt{\upshape SMOOTH:}
          For all $k=1,\dots, K$, compute $u_\ell^k\coloneq\Psi_\ell(u_\ell^{k-1})$.
          Set ${\kk}[\ell]\coloneq K$.
        \item \texttt{\upshape ESTIMATE:}
          Compute refinement indicator $\eta_\ell(T, u_\ell^{{\kk}[\ell]})$ for all $T\in\TT_\ell$.
      \end{enumerate}
    \item\label{algorithm:SAFEM:2} \texttt{\upshape MARK:}
      Determine a set $\widetilde{\MM}_\ell\subseteq\TT_\ell$ with up to the factor $\Cmark$ minimal cardinality satisfying the Dörfler marking criterion $\theta\eta_\ell(u_\ell^{{\kk}[\ell]})^2\leq\eta_\ell(\widetilde{\MM}_\ell,u_\ell^{{\kk}[\ell]})^2$.
    \item\label{algorithm:SAFEM:3} \texttt{\upshape MARK (CARDINALITY CONTROL):}
      Select a subset $\MM_\ell\subseteq\widetilde{\MM}_\ell$ such that
      \begin{equation}\label{algorithm:SAFEM:cardcontrol}
        \MM_\ell=\widetilde{\MM}_\ell\;\text{if }\ell\in L\cdotS\N_0,\quad\text{and}\quad
        \#\MM_\ell\leq\Ccard\,\#\MM_{\ell-1}\;\text{otherwise.}
      \end{equation}
    \item\label{algorithm:SAFEM:4} \texttt{\upshape REFINE:}
      Generate $\TT_{\ell+1}\coloneq\refine(\TT_\ell,\MM_\ell)$ and set $u_{\ell+1}^0\coloneq u_\ell^{{\kk}[\ell]}$.
  \end{enumerate}
\end{algorithm}

\begin{remark}[asymptotic S-AFEM formulation]
  In~\cite{HeltMul}, S-AFEM is formulated for a prescribed number of adaptive mesh refinements.
  The initial and final levels use accurate algebraic solves, while the intermediate levels use a fixed number of smoothing steps.

  The number of intermediate smoothing levels needed to reach a prescribed accuracy is, in general, not known \textsl{a~priori} and cannot be reliably determined \textsl{a~posteriori} without robust control of the algebraic error of the smoother.
  The successive repetition of this finite S-AFEM procedure from~\cite{HeltMul} led to \cref{algorithm:SAFEM} and enables the asymptotic convergence analysis in the present paper.
\end{remark}

\begin{remark}[cardinality control]
  \Cref{algorithm:SAFEM} introduces a novel cardinality-control step \cref{algorithm:SAFEM}\cref{algorithm:SAFEM:3} that modifies the preliminary Dörfler set $\widetilde{\MM}_\ell$ from \cref{algorithm:SAFEM}\cref{algorithm:SAFEM:2} to a final marking set $\MM_\ell$ enforcing two key properties: First, on solve levels $\ell\in L\cdotS\N_0$, no modification occurs, i.e., $\MM_\ell = \widetilde{\MM}_\ell$, so the Dörfler criterion remains valid.
  Second, on intermediate levels $\ell\in\N_0\setminus L\cdotS\N_0$, the cardinality $\#\MM_\ell$ is bounded (up to $\Ccard$) by $\#\MM_{\ell-1}$.
  Any subset $\MM_\ell \subseteq \widetilde{\MM}_\ell$ within this bound is admissible, including the extreme choice $\MM_\ell = \emptyset$ on all intermediate levels.
  Evidently, a larger value of $\Ccard$ allows $\MM_\ell = \widetilde{\MM}_\ell$ more often.
  The purpose of cardinality control is to prevent over-refinement on intermediate levels with inaccurate algebraic solutions.
  Its role in the analysis is discussed in \cref{rem:CC-linear,rem:CC-rates} below.
  Its practical effect is addressed in \cref{subsec:kellogg}.
\end{remark}

\begin{remark}[relation to inexact AFEM theory]
  For $L=1$,~\cref{algorithm:SAFEM:1:smooth,algorithm:SAFEM:3} are void, so \emph{every} mesh level employs a uniformly contractive solver~\cref{eq:contractivesolver} with stopping criterion~\cref{algorithm:SAFEM:solverstop}.
  Hence, \cref{algorithm:SAFEM} reduces to the scheme of~\cite{GHPS21, BFMPS24}, where full R-linear convergence and quasi-optimal complexity were established.
  For $L\geq 2$, uniformly contractive solvers are required only on the solve levels, while intermediate levels apply a fixed number of iterations of a uniformly stable smoother~\cref{eq:stablesolver}.
  Thus, \cref{algorithm:SAFEM} significantly relaxes the algebraic requirements of~\cite{GHPS21,BFMPS24} while preserving the same convergence guarantees.
\end{remark}

As is standard in adaptive algorithms with inexact solutions (see, e.g.,~\cite{BFMPS24}), we introduce the index set $\QQ\coloneq\set{(\ell,k)\in\N_0^2 \given u_\ell^k\in\XX_\ell\text{ is defined in \cref{algorithm:SAFEM}}}$ equipped with the partial order $(\ell^\prime,k^\prime)\leq(\ell,k)$ if $u_{\ell^\prime}^{k^\prime}$ is defined not later than $u_\ell^k$ in \cref{algorithm:SAFEM}.
For $(\ell,k)\in\QQ$, define the total step counter by $\abs{\ell,k}\coloneq\#\set{(\ell^\prime,k^\prime)\in\QQ \given (\ell^\prime,k^\prime) < (\ell,k)}$, which yields an order-preserving bijection $\abs{\cdot}\colon\QQ\to\N_0$.
We write ${\kk}={\kk}[\ell]$ when the index $\ell$ is clear from the context, e.g., $u_\ell^{\kk}$ instead of $u_\ell^{{\kk}[\ell]}$.
Define the stopping index of the mesh-refinement loop by $\ellu\coloneq\sup\set{\ell\in\N_0\given (\ell,0)\in\QQ}\in\N_0\cup\{+\infty\}$.
Typically, $\ellu=+\infty$, i.e., the stopping criterion~\cref{algorithm:SAFEM:solverstop} is satisfied after finitely many solver steps ${\kk}[\ell]\in\N$ on each solve level $\ell\in L\cdotS\N_0$.
In the case $\ellu<+\infty$, the solver $\Phi_{\ellu}$ fails to terminate on $\TT_\ellu$ and we set ${\kk}[\ellu]\coloneq +\infty$, which can only occur for $\ellu\in L\cdotS\N_0$.
To cover this case, define the set of reached solve levels by $\underline{\SS}\coloneq\set{\ell\in L\cdotS\N_0\given (\ell, 0)\in\QQ}$ and the subset $\SS\coloneq\underline{\SS}\setminus\{\ellu\}$ where the solver terminates.
If $\ellu=+\infty$, then $\underline{\SS}=\SS=L\cdotS\N_0$.
In either case, $u_\ell^{{\kk}[\ell]}\in\XX_\ell$ is well-defined for all $\ell\in\SS$.
Similarly, define the set of reached intermediate levels by $\II\coloneq\set{\ell\in \N_0\setminus L\cdotS\N_0\given (\ell, 0)\in\QQ}$.

\section{Full R-linear convergence}\label{sec:rlinear}

This section establishes two results asserting full R-linear convergence of \cref{algorithm:SAFEM}.
The first result, \cref{th:fullrlinear}, establishes \emph{unconditional} full R-linear convergence for any choice of input parameters, provided the smoother is stable~\cref{eq:stablesolver} with $\Calg\leq 1$.
This mild assumption is satisfied by classical smoothers (cf.~\cref{sec:stableSym}) and even by the identity operator.
The second result, \cref{th:Condfullrlinear}, uses a perturbation argument to show that the restriction $\Calg\leq 1$ can even be removed when the stopping parameter $\lambda>0$ is chosen sufficiently small.

\begin{theorem}[unconditional full R-linear convergence]\label{th:fullrlinear}
  Assume that the solver $\Phi_\ell$ is uniformly contractive~\cref{eq:contractivesolver} and that the smoother $\Psi_\ell$ is uniformly stable~\cref{eq:stablesolver} with $\Calg\leq 1$.
  Let $K,L\in\N, \Cmark, \Ccard\geq 1, \lambda>0$, $u_0^0\in\XX_0$, and $0<\theta\leq 1$ be arbitrary.
  Then,
  \cref{algorithm:SAFEM} guarantees full R-linear convergence of the quasi-error%
  \begin{equation}\label{eq:defQuasiErr}
    \Hr_\ell^k\coloneq\enorm{u_\ell^\star-u_\ell^k}+\eta_\ell(u_\ell^\star),
  \end{equation}
  i.e., there exist constants $\Clin>0$ and $0<\qlin<1$ such that
  \begin{equation}\label{th:fullrlinear:eq0}
    \Hr_\ell^k\leq\Clin\qlin^{\abs{\ell,k}-\abs{\ell^\prime,k^\prime}}\Hr_{\ell^\prime}^{k^\prime}\quad
    \text{for all }(\ell^\prime,k^\prime),(\ell,k)\in\QQ\text{ with }\abs{\ell^\prime,k^\prime}\leq\abs{\ell,k}.
  \end{equation}
  The constants $\Clin, \qlin$ depend only on $K, L, \lambda, \theta, \qalg$, and the constants in~\crefrange{A1}{A4} and~\cref{QM}.
  Specifically, the constants $\Clin, \qlin$ grow at most linearly in $K$ and at most exponentially in $L$ (the precise bounds are detailed in Step~5 of the proof).
\end{theorem}

\begin{remark}\label{rem:convergence}
  R-linear convergence~\cref{th:fullrlinear:eq0} implies convergence of \cref{algorithm:SAFEM} via
  \begin{equation*}
    \enorm{u^\star-u_\ell^k}\leq\enorm{u^\star-u_\ell^\star}+\enorm{u_\ell^\star-u_\ell^{k}}\eqreff{A3}{\lesssim}\Hr_\ell^k\eqreff{th:fullrlinear:eq0}{\leq}\Clin\,\qlin^{\abs{\ell,k}}\,\Hr_0^0\:\xrightarrow{\abs{\ell,k}\to\infty}\:0.
  \end{equation*}
  In particular, if $\ellu<+\infty$, then $\eta_\ellu(u_\ellu^\star)\leq\Hr_\ellu^k\to 0$ as $k\to\infty$ implies $\eta_\ellu(u_\ellu^\star)=0$ and hence $u^\star=u_\ellu^\star$ by reliability~\cref{A3}, i.e., the case $\ellu<+\infty$ can only occur if $u^\star$ is discrete.
\end{remark}

\begin{remark}\label{rem:CC-linear}
  \Cref{th:fullrlinear} remains valid in the degenerate case $\const{C}{card}=+\infty$, in which cardinality control~\cref{algorithm:SAFEM:cardcontrol} is void and $\MM_\ell=\widetilde{\MM}_\ell$ is admissible on all mesh levels.
\end{remark}

The proof of \cref{th:fullrlinear} is based on the characterization of full R-linear convergence~\cref{th:fullrlinear:eq0} from~\cite[Lemma~2]{BFMPS24}, showing that~\cref{th:fullrlinear:eq0} is equivalent to \emph{tail summability}
\begin{equation}\label{eq:tailsum}
  \sum_{\substack{(\ell^\prime,k^\prime)\in\QQ \\ \abs{\ell^\prime,k^\prime}>\abs{\ell,k}}}\Hr_{\ell^\prime}^{k^\prime}\leq\Clin^\prime\,\Hr_\ell^k\quad\text{for all $(\ell,k)\in\QQ$}
\end{equation}
for a certain constant $\Clin^\prime>0$.
The proof of~\cref{eq:tailsum} relies on the following two lemmas.
The first refines~\cite[Lemma 1]{BFMPS24} by omitting one of its assumptions.

\begin{lemma}[tail summability criterion]\label{lem:tailsumCrit}
  Let $(a_j)_{j\in\N_0}$ and $(b_j)_{j\in\N_0}$ be sequences in $\R_{\geq 0}$.
  Suppose there exist constants $0<q<1$, $0<\delta\leq1$, and $C>0$ such that
  \begin{equation}\label{lem:tailsumCrit:eq1}
    a_{j+1}\leq qa_j+b_j\quad\text{and}\quad\sum_{j^\prime=j}^{j+N}b_{j^\prime}^2\leq C(N+1)^{1-\delta}a_j^2\quad\text{for all $j,N\in\N_0$.}
  \end{equation}
  Then, the sequence $(a_j)_{j\in\N_0}$ is tail summable, i.e., there exists $\widetilde{C}>0$ such that
  \begin{equation}\label{lem:tailsumCrit:eq2}
    \sum_{j^\prime=j+1}^\infty a_{j^\prime}\leq \widetilde{C}\,a_j\quad\text{for all $j\in\N_0$}.
  \end{equation}
  The optimal constant $\widetilde{C}$ in~\cref{lem:tailsumCrit:eq2} is bounded in the following sense.
  Given
  $0<\varepsilon<q^{-2}-1$, define, for all $N\in\N_0$,
  \begin{equation}\label{lem:tailsumCrit:eq2.8}
    M(N) \coloneq \frac{1+(1+\varepsilon^{-1})C N^{1-\delta}} {1-(1+\varepsilon)q^{2}} \prod_{R=1}^{N} \frac{(1+\varepsilon)q^{2}+(1+\varepsilon^{-1})C R^{1-\delta}} {1+(1+\varepsilon^{-1})C R^{1-\delta}}>0.
  \end{equation}
  Then, there exists $N_0\in\N$ such that $M(N_0)<1$.
  For any such choice of $N_0$ and $\varepsilon$,
  \begin{equation}\label{lem:tailsumCrit:eq2.9}
    \widetilde{C}\leq
    \frac{ M(N_0)^{(1-N_0)/(2N_0)} }
    {1- M(N_0)^{1/(2N_0)}}
    \,\max_{R=0, \ldotp\ldotp, N_0-1}
    M(R)^{1/2}.
  \end{equation}
\end{lemma}
\begin{proof}
  Let $0<\varepsilon<q^{-2}-1$.
  Arguing as in the proof of~\cite[Lemma 1]{BFMPS24}, the perturbed contraction and the summability from~\cref{lem:tailsumCrit:eq1} imply that, for $M(N)$ from~\cref{lem:tailsumCrit:eq2.8},%
  \begin{equation}\label{lem:tailsumCrit:eq2.7}
    a_{j+N}^2\leq M(N)\,a_j^2\quad\text{for all $j,N\in\N_0$\quad and\quad$M(N)\to 0$ as $N\to+\infty$.}
  \end{equation}
  Choose $N_0\in\N$ such that $M(N_0)<1$.
  For arbitrary $N\in\N_0$, write $N=mN_0+r$ with $m,r\in\N_0$ and $r<N_0$.
  Repeated application of~\cref{lem:tailsumCrit:eq2.7} gives
  \begin{equation}\label{lem:tailsumCrit:eq2.6}
    \begin{aligned}
      a_{j+N}^2=a_{j+mN_0+r}^2\eqreff{lem:tailsumCrit:eq2.7}{\leq}
      M(N_0)^m\,a_{j+r}^2 & \eqreff*{lem:tailsumCrit:eq2.7}{\:\leq\:}\Bigl(\max_{R=0, \ldotp\ldotp, N_0-1} M(R)\Bigr)M(N_0)^ma_{j}^2 \\ &\:\leq\: \Bigl(\max_{R=0, \ldotp\ldotp, N_0-1} M(R)\Bigr)M(N_0)^{N/N_0-1}a_{j}^2.
    \end{aligned}
  \end{equation}
  Taking square roots in~\cref{lem:tailsumCrit:eq2.6} establishes R-linear convergence of $(a_j)_{j\in\N_0}$ with constants $\const{\widetilde{C}}{lin}\coloneq\max_{R=0, \ldotp\ldotp, N_0-1} M(R)^{1/2}M(N_0)^{-1/2}>0$ and $\const{\widetilde{q}}{lin}\coloneq M(N_0)^{1/(2N_0)}<1$.
  Hence, by~\cite[Lemma 2]{BFMPS24}, tail summability~\cref{lem:tailsumCrit:eq2} holds with $\widetilde{C}\leq\const{\widetilde{C}}{lin}\const{\widetilde{q}}{lin}/(1-\const{\widetilde{q}}{lin})$.
  Inserting the definition of $M(N)$ from~\cref{lem:tailsumCrit:eq2.8} thus yields the explicit bound in~\cref{lem:tailsumCrit:eq2.9}.
\end{proof}

The second lemma delivers the novel key idea by establishing tail summability of $\Hr_\ell^{\kk}$ along the solve levels via Dörfler marking, solver contraction~\cref{eq:contractivesolver}, and \cref{lem:tailsumCrit}.

\begin{lemma}[tail summability along solve levels]\label{lem:sumsolve}
  Under the assumptions of \cref{th:fullrlinear},
  the quasi-error $\Hr_\ell^{\kk}$ of the final iterates $u_\ell^{\kk}$ generated by
  \cref{algorithm:SAFEM} is tail summable on the solve levels, i.e., there exists a constant $\Ctail>0$ such that
  \begin{equation}\label{lem:sumsolve:eq0}
    \sum_{\ell^\prime\in\SS,\,\ell^\prime>\ell}\Hr_{\ell^\prime}^{\kk}\leq \Ctail\,\Hr_\ell^{\kk}
    \quad\text{for all }\ell\in\SS.
  \end{equation}
  The constant $\Ctail$ depends only on $L, \theta, \qalg$, and the constants in~\crefrange{A1}{A4}.
\end{lemma}
\begin{proof}
  The proof is split into four steps.

  \step{1}
  Let $\ell\in\SS$ be a solve level with $\ell+L\in\SS$.
  Due to~\crefrange{A1}{A2}, we have
  \begin{equation*}
    \begin{aligned}
      \eta_{\ell+L}(u_\ell^{\kk})^2
       & \leq \eta_\ell(\TT_{\ell+L}\cap\TT_\ell, u_\ell^{\kk})^{2}+\qred^2\,\eta_\ell(\TT_\ell\setminus\TT_{\ell+L}, u_\ell^{\kk})^{2} \\
       & =\eta_\ell(u_\ell^{\kk})^2-(1-\qred^2)\eta_\ell(\TT_\ell\setminus\TT_{\ell+L}, u_\ell^{\kk})^{2}.
    \end{aligned}
  \end{equation*}
  By \crefrange{algorithm:SAFEM:2}{algorithm:SAFEM:3}, $\MM_\ell$ satisfies the Dörfler criterion.
  Since $\MM_\ell\subseteq\TT_\ell\setminus\TT_{\ell+L}$, it follows that
  \begin{equation*}
    \theta\,\eta_\ell(u_\ell^{\kk})^2\leq\eta_\ell(\MM_\ell, u_\ell^{\kk})^2\leq\eta_\ell(\TT_\ell\setminus\TT_{\ell+L}, u_\ell^{\kk})^2.
  \end{equation*}
  The combination of the previous two displayed formulas leads to
  \begin{equation}\label{lem:sumsolve:eq0.4}
    \eta_{\ell+L}(u_\ell^{\kk})\leq\qest\,\eta_\ell(u_\ell^{\kk})\quad\text{with}\quad 0<\qest\coloneq\big[1-(1-\qred^2)\theta\big]^{1/2}<1.
  \end{equation}
  From~\cref{lem:sumsolve:eq0.4} together with stability~\cref{A1}, we obtain
  \begin{equation}\label{lem:sumsolve:eq0.5}
    \begin{aligned}
      \eta_{\ell+L}(u_{\ell+L}^{\kk})
      \leq\qest\,\eta_\ell(u_\ell^{\kk})+\Cstab\,\enorm{u_{\ell+L}^{\kk}-u_\ell^{\kk}}.
    \end{aligned}
  \end{equation}
  For all $\ell\in\SS$, define
  \begin{equation}\label{lem:sumsolve:eq0.6}
    \sigma_\ell\coloneq\sum_{\ell^\prime=\ell}^{\ell+L-1}\enorm{u_{\ell^\prime+1}^\star-u_{\ell^\prime}^\star}.
  \end{equation}
  Since $(\ell+L-1, 0)\in\QQ$ for all $\ell\in\SS$, $\sigma_\ell$ is well-defined.
  There holds
  \begin{equation}\label{lem:sumsolve:eq0.7}
    \enorm{u_{\ell+L}^{\kk}-u_\ell^{\kk}}\leq\enorm{u_{\ell+L}^\star-u_{\ell+L}^{\kk}}+\enorm{u_{\ell+L}^\star-u_\ell^{\kk}}\leq\enorm{u_{\ell+L}^\star-u_{\ell+L}^{\kk}}+\enorm{u_\ell^\star-u_\ell^{\kk}}+\sigma_\ell.
  \end{equation}
  The combination of~\cref{lem:sumsolve:eq0.5} and~\cref{lem:sumsolve:eq0.7} shows that, for all $\ell\in\SS$ with $\ell+L\in\SS$,
  \begin{equation}\label{lem:sumsolve:eq1}
    \eta_{\ell+L}(u_{\ell+L}^{\kk})\leq\qest\,\eta_\ell(u_\ell^{\kk})+\Cstab\big(\enorm{u_{\ell+L}^\star-u_{\ell+L}^{\kk}}+\enorm{u_\ell^\star-u_\ell^{\kk}}+\sigma_\ell\big).
  \end{equation}

  \step{2}
  Let $\ell\in\SS$ and $0 < r <L$.
  From ${\kk}[\ell+r]=K$, nested iteration $u_{\ell+r}^0=u_{\ell+r-1}^{\kk}$, and the stability~\cref{eq:stablesolver} with $\const{C}{alg}\leq 1$, it follows that
  \begin{equation*}
    \enorm{u_{\ell+r}^\star-u_{\ell+r}^{\kk}}\eqreff{eq:stablesolver}{\leq}\const{C}{alg}^K\,\enorm{u_{\ell+r}^\star-u_{\ell+r}^0}\leq\enorm{u_{\ell+r-1}^\star-u_{\ell+r-1}^{\kk}}+\enorm{u_{\ell+r}^\star-u_{\ell+r-1}^\star}.
  \end{equation*}
  The successive application of this argument, including the trivial case $r=0$, verifies
  \begin{equation}\label{lem:sumsolve:eq1.5}
    \enorm{u_{\ell+r}^\star-u_{\ell+r}^{\kk}}\leq\enorm{u_{\ell}^\star-u_\ell^{\kk}}+\sum_{\ell^\prime=\ell}^{\ell+r-1}\enorm{u_{\ell^\prime+1}^\star-u_{\ell^\prime}^\star}\quad\text{for all $\ell\in\SS$, $0\leq r<L$}.
  \end{equation}

  \step{3}
  Let $\ell\in\SS$ with $\ell+L\in\SS$.
  The nested iteration $u_{\ell+L}^0=u_{\ell+L-1}^{\kk}$, the solver contraction~\cref{eq:contractivesolver}, the bound ${\kk}[\ell+L]\geq 1$, and~\cref{lem:sumsolve:eq1.5} with $r=L-1$ show that%
  \begin{equation}\label{lem:sumsolve:eq2}
    \begin{aligned}
       & \enorm{u_{\ell+L}^\star-u_{\ell+L}^{{\kk}}}\eqreff{eq:contractivesolver}{\leq}\qalg^{{\kk}[\ell+L]}\,\enorm{u_{\ell+L}^\star-u_{\ell+L}^0}\leq\qalg\,\enorm{u_{\ell+L}^\star-u_{\ell+L-1}^{\kk}} \\
       & \quad\leqpad\qalg\big(\enorm{u_{\ell+L-1}^\star-u_{\ell+L-1}^{\kk}}+\enorm{u_{\ell+L}^\star-u_{\ell+L-1}^\star}\big)
      \\
       & \quad\eqreff*{lem:sumsolve:eq1.5}{\leqpad}\qalg\Big(\enorm{u_{\ell}^\star-u_\ell^{\kk}}+\sum_{\ell^\prime=\ell}^{\ell+L-1}\enorm{u_{\ell^\prime+1}^\star-u_{\ell^\prime}^\star}\Big)
      \eqreff{lem:sumsolve:eq0.6}{=}\qalg\,\big(\enorm{u_{\ell}^\star-u_\ell^{\kk}}+\sigma_\ell\big).
    \end{aligned}
  \end{equation}

  \step{4}
  For $j\in\N_0$ with $jL<\ellu$, define $C_\gamma\coloneq\gamma\Cstab$, $a_j\coloneq\enorm{u_{jL}^\star-u_{jL}^{\kk}}+\gamma\,\eta_{jL}(u_{jL}^{\kk})$, and $b_j\coloneq q\sigma_{jL}$, where $0<\gamma\leq 1$ is chosen so that $q\coloneq\max\{\qalg+C_\gamma(1+\qalg),\qest\}<1$.
  For $j\in\N_0$ with $(j+1)L<\ellu$,~\cref{lem:sumsolve:eq1,lem:sumsolve:eq2} with $\ell=jL$ give
  \begin{equation}\label{lem:sumsolve:eq2.1}
    \begin{aligned}
      a_{j+1}
       & \eqreff*{lem:sumsolve:eq1}{\;\leq\;}(1+C_\gamma)\enorm{u_{jL+L}^\star-u_{jL+L}^{\kk}}+\gamma\qest\,\eta_{jL}(u_{jL}^{\kk})
      +C_\gamma(\,\enorm{u_{jL}^\star-u_{jL}^{\kk}}+\sigma_{jL})
      \\
       & \eqreff*{lem:sumsolve:eq2}{\;\leq\;}\big[\qalg+C_\gamma(1+\qalg)\big]\big(\enorm{u_{jL}^\star-u_{jL}^{\kk}}+\sigma_{jL}\big)+\gamma\qest\,\eta_{jL}(u_{jL}^{\kk})\leq qa_{j}+b_j.
    \end{aligned}
  \end{equation}
  Moreover, from stability~\cref{A1} and reliability~\cref{A3}, it follows that
  \begin{equation}\label{lem:sumsolve:eq2.15}
    \enorm{u^\star-u_{jL}^\star}\eqreff{A3}{\lesssim}\eta_{jL}(u_{jL}^\star)\eqreff{A1}{\lesssim}\eta_{jL}(u_{jL}^{\kk})+\enorm{u_{jL}^\star-u_{jL}^{\kk}}\simeq a_j.
  \end{equation}
  Together with~\cref{A4}, the previous estimate shows, for all $N\in\N_0$ with $(j+N)L<\ellu$,
  \begin{equation}\label{lem:sumsolve:eq2.12}
    \begin{aligned}
       & \quad \sum_{j^\prime=j}^{j+N}\sigma_{j^\prime L}^2
      \leq L\sum_{j^\prime=j}^{j+N}\sum_{\ell^\prime=j^\prime L}^{j^\prime L+L-1}\enorm{u_{\ell^\prime+1}^\star-u_{\ell^\prime}^\star}^2
      = L\sum_{\ell^\prime=jL}^{jL+(N+1)L-1}\enorm{u_{\ell^\prime+1}^\star-u_{\ell^\prime}^\star}^2
      \\
       & \qquad\eqreff{A4}{\lesssim}L\big((N+1)L\big)^{1-\delta}\,\enorm{u^\star-u_{jL}^\star}^2\eqreff{lem:sumsolve:eq2.15}{\lesssim} L^{2-\delta}(N+1)^{1-\delta}\,a_j^2,
    \end{aligned}
  \end{equation}
  where the hidden constant depends only on $\theta, \qalg, \gamma$, and the constants in~\crefrange{A1}{A4}.
  By~\cref{lem:sumsolve:eq2.1}, \cref{lem:sumsolve:eq2.12}, and $b_j\leq\sigma_{jL}$, $(a_j)_{j\in\N_0}$ and $(b_j)_{j\in\N_0}$ (extended by zero if $\ellu<+\infty$) satisfy the assumptions~\cref{lem:tailsumCrit:eq1} of \cref{lem:tailsumCrit}.
  Hence, $(a_j)_{j\in\N_0}$ is tail summable, i.e.,
  \begin{equation}\label{lem:sumsolve:eq4}
    \sum_{j^\prime=j+1}^{\ellu/L-1}a_{j^\prime}\eqreff{lem:tailsumCrit:eq2}{\lesssim} a_j\quad\text{for all $j\in\N_0$ with $jL<\ellu$}.
  \end{equation}
  Finally, since, for all $j\in\N_0$ with $jL<\ellu$,
  \begin{equation*}
    \Hr_{jL}^{\kk}\eqreff{eq:defQuasiErr}{=}\enorm{u_{jL}^\star-u_{jL}^{\kk}}+\eta_{jL}(u_{jL}^\star)\eqreff{A1}{\simeq}\enorm{u_{jL}^\star-u_{jL}^{\kk}}+\eta_{jL}(u_{jL}^{\kk})\simeq a_j,
  \end{equation*}
  and every $\ell\in\SS$ can be written as $\ell=jL$ for some $j\in\N_0$ with $jL<\ellu$, the desired tail summability~\cref{lem:sumsolve:eq0} follows from~\cref{lem:sumsolve:eq4}.
  This concludes the proof.
\end{proof}

The central idea in proving \cref{th:fullrlinear} is that extending~\cref{lem:sumsolve:eq0} to tail summability of the full sequence requires only~\cref{eq:stablesolver} on intermediate and~\cref{eq:contractivesolver} on solve levels.

\begin{proof}[Proof of \cref{th:fullrlinear}]
  The proof is divided into five steps.

  \step[monotonicity in $\boldsymbol{k}$]{1}
  Let $(\ell, k), (\ell, k^\prime)\in\QQ$ with $k\leq k^\prime$.
  The solver contraction~\cref{eq:contractivesolver} and the smoother stability~\cref{eq:stablesolver} with $\const{C}{alg}\leq 1$ yield
  \begin{equation*}
    \Hr_\ell^{k'}\eqreff{eq:defQuasiErr}{=}\enorm{u_\ell^\star-u_\ell^{k'}}+\eta_\ell(u_\ell^\star)\leq\max\{\qalg, \const{C}{alg}\}^{k'-k}\,\enorm{u_\ell^\star-u_\ell^{k}}+\eta_\ell(u_\ell^\star)\leq\Hr_\ell^{k}.
  \end{equation*}
  Hence, it holds
  \begin{equation}\label{th:fullrlinear:eq1}
    \Hr_\ell^{k'}\leq\Hr_\ell^{k}\quad\text{for all $(\ell, k), (\ell, k^\prime)\in\QQ$ with $k\leq k^\prime$}.
  \end{equation}

  \step[tail summability in $\boldsymbol{k}$]{2} For intermediate levels $\ell\in\II$, tail summability in $k$ follows immediately from~\cref{th:fullrlinear:eq1} and ${\kk}[\ell]=K$ via
  \begin{equation}\label{th:fullrlinear:eq2}
    \sum_{k^\prime=k+1}^{{\kk}[\ell]}\Hr_\ell^{k^\prime}\eqreff{th:fullrlinear:eq1}{\leq}\sum_{k^\prime=k+1}^{{\kk}[\ell]}\Hr_\ell^k\leq K\Hr_\ell^k\quad\text{for all $\ell\in\II$ and all $k<{\kk}[\ell]=K$}.
  \end{equation}
  Next, consider an arbitrary solve level $\ell\in\underline{\SS}$.
  Let $0\leq k<k^\prime<{\kk}[\ell]$.
  The failure of the stopping criterion~\cref{algorithm:SAFEM:solverstop} and uniform contraction~\cref{eq:contractivesolver} of the iterative solver yield%
  \begin{equation}\label{th:fullrlinear:eq3.1}
    \begin{aligned}
      \lambda\eta_\ell(u_\ell^{k^\prime})\eqreff{algorithm:SAFEM:solverstop}{<}\enorm{u_\ell^{k^\prime}-u_\ell^{k^\prime-1}}
      \eqreff{eq:contractivesolver}{\leq}\qalg^{-1}(1+\qalg)\qalg^{k^\prime-k}\,\enorm{u_\ell^\star-u_\ell^k}.
    \end{aligned}
  \end{equation}
  Set $C_1\coloneq 1+\Cstab+\lambda^{-1}\qalg^{-1}(1+\qalg)$.
  Together with stability~\cref{A1} and $\enorm{u_\ell^\star-u_\ell^k}\leq\Hr_\ell^k$, this shows that, for all $0\leq k<k^\prime<{\kk}[\ell]$,
  \begin{equation}\label{th:fullrlinear:eq3}
    \begin{aligned}
      \Hr_\ell^{k^\prime}
       & \eqreff*{A1}{\leqpad}(1+\Cstab)\enorm{u_\ell^\star-u_\ell^{k^\prime}}+\eta_\ell(u_\ell^{k^\prime})
      \\
       & \eqreff*{eq:contractivesolver}{\leqpad}(1+\Cstab)\qalg^{k^\prime-k}\,\enorm{u_\ell^\star-u_\ell^k}+\eta_\ell(u_\ell^{k^\prime})
      \eqreff{th:fullrlinear:eq3.1}{\leq}
      C_1\,\qalg^{k^\prime-k}\,\Hr_\ell^k.
    \end{aligned}
  \end{equation}
  The estimate~\cref{th:fullrlinear:eq3} and the geometric series prove, for all $\ell\in\underline{\SS}$ and all $k<{\kk}[\ell]$,
  \begin{equation}\label{th:fullrlinear:eq4}
    \begin{aligned}
      \sum_{k^\prime=k+1}^{{\kk}[\ell]}\Hr_\ell^{k^\prime}
      \eqreff{th:fullrlinear:eq1}{\leq}\Hr_\ell^k+\sum_{k^\prime=k+1}^{{\kk}[\ell]-1}\Hr_\ell^{k^\prime}
       & \eqreff*{th:fullrlinear:eq3}{\:\leqpad\:}\Hr_\ell^k+C_1\,\Hr_\ell^k\sum_{k^\prime=k+1}^{{\kk}[\ell]-1}\qalg^{k^\prime-k} \\
       & \:\leqpad\: \big[1+C_1(1-\qalg)^{-1}\big]\,\Hr_\ell^k.
    \end{aligned}
  \end{equation}
  Set $C_2\coloneq\max\{K,1+C_1(1-\qalg)^{-1}\}$.
  Then, combining~\cref{th:fullrlinear:eq2} and~\cref{th:fullrlinear:eq4}, we conclude%
  \begin{equation}\label{th:fullrlinear:eq5}
    \sum_{k^\prime=k+1}^{{\kk}[\ell]}\Hr_\ell^{k^\prime}\leq C_2\,\Hr_\ell^k\quad\text{for all $(\ell,k)\in\QQ$ with $k<{\kk}[\ell]$.}
  \end{equation}

  \step[stability under mesh refinement]{3}
  By~\cref{QM} and $u_\ell^0=u_{\ell-1}^{\kk}$, it holds
  \begin{equation*}
    \begin{aligned}
       & \Hr_\ell^0
      \eqreff{QM}{\leq}\Cmon\big[ \enorm{u_\ell^\star-u_\ell^0}+\eta_{\ell-1}(u_{\ell-1}^\star)\big]
      \\
       & \leq\Cmon\big[\enorm{u_\ell^\star-u_{\ell-1}^\star}+\enorm{u_{\ell-1}^\star-u_\ell^0}+\eta_{\ell-1}(u_{\ell-1}^\star)\big]
      \eqreff{eq:defQuasiErr}{=}\Cmon\big[\enorm{u_\ell^\star-u_{\ell-1}^\star}+\Hr_{\ell-1}^{\kk}\big].
    \end{aligned}
  \end{equation*}
  To bound $\enorm{u_\ell^\star-u_{\ell-1}^\star}$, we use reliability~\cref{A3} and quasi-monotonicity~\cref{QM} to deduce
  \begin{equation*}
    \enorm{u_\ell^\star-u_{\ell-1}^\star}\eqreff{A3}{\leq}\Crel\big[\eta_\ell(u_\ell^\star)+\eta_{\ell-1}(u_{\ell-1}^\star)\big]
    \eqreff{QM}{\leq}\Crel(1+\Cmon)\eta_{\ell-1}(u_{\ell-1}^\star).
  \end{equation*}
  For $C_3\coloneq\Cmon[\Crel(1+\Cmon)+1]$, the previous two displayed formulas result in
  \begin{equation}\label{th:fullrlinear:eq6}
    \Hr_\ell^0\leq C_3\,\Hr_{\ell-1}^{\kk}.
  \end{equation}
  Repeated application of~\cref{th:fullrlinear:eq6} gives
  \begin{equation}\label{th:fullrlinear:eq7}
    \Hr_\ell^{\kk}\eqreff{th:fullrlinear:eq1}{\leq}\Hr_\ell^0\eqreff{th:fullrlinear:eq6}{\leq}
    C_3^j\,\Hr_{\ell-j}^{\kk}\quad\text{for all $0\leq j\leq\ell<\ellu$.
    }
  \end{equation}

  \step[tail summability of $\boldsymbol{\Hr_\ell^{\kk}}$ in $\boldsymbol{\ell}$]{4}
  Denote by $\modop{\ell}{L}$ the unique integer $0\leq r<L$ such that $\ell-r\in L\cdotS\N_0$.
  Let $\ell\in\N_0$ with $\ell<\ellu$ and let $\ell_0\in L\cdotS\N_0$ be the minimal index such that $\ell_0\geq\ell$.
  In particular, $\ell_0-\ell\leq L-1$.
  For the first case, suppose that $\ell_0<\ellu$, which ensures that $\ell_0\in\SS$.
  Then, the tail sum decomposes as
  \begin{equation}\label{th:fullrlinear:eq9}
    \sum_{\ell^\prime=\ell+1}^{\ellu-1}\Hr_{\ell^\prime}^{\kk}=\sum_{\ell^\prime=\ell+1}^{\ell_0}\Hr_{\ell^\prime}^{\kk}+\sum_{\ell^\prime\in\SS,\,\ell_0<\ell^\prime}\Hr_{\ell^\prime}^{\kk}+\sum_{\ell^\prime\in\II ,\,\ell_0<\ell^\prime}\Hr_{\ell^\prime}^{\kk}.
  \end{equation}
  To estimate the first sum in~\cref{th:fullrlinear:eq9}, we apply~\cref{th:fullrlinear:eq7} and $C_3\geq 1$ to deduce
  \begin{equation}\label{th:fullrlinear:eq9.1}
    \sum_{\ell^\prime=\ell+1}^{\ell_0}\Hr_{\ell^\prime}^{\kk}\eqreff{th:fullrlinear:eq7}{\leq}\sum_{\ell^\prime=\ell+1}^{\ell_0}
    C_3^{\ell^\prime-\ell}\,\Hr_{\ell}^{\kk}=\frac{C_3^{\ell_0-\ell+1}-C_3}{C_3-1}\,\Hr_{\ell}^{\kk}\leq\frac{C_3^L-C_3}{C_3-1}\,\Hr_{\ell}^{\kk}.
  \end{equation}
  For the second sum in~\cref{th:fullrlinear:eq9}, we apply \cref{lem:sumsolve},~\cref{th:fullrlinear:eq7}, and $C_3\geq 1$ to obtain
  \begin{equation}\label{th:fullrlinear:eq9.2}
    \sum_{\ell^\prime\in\SS,\,\ell_0<\ell^\prime}\Hr_{\ell^\prime}^{\kk}\eqreff{lem:sumsolve:eq0}{\leq}\const{C}{tail}\,\Hr_{\ell_0}^{\kk}\eqreff{th:fullrlinear:eq7}{\leq}\const{C}{tail}\,C_3^{\ell_0-\ell}\,\Hr_\ell^{\kk}\leq\const{C}{tail}\,C_3^{L-1}\,\Hr_\ell^{\kk}.
  \end{equation}
  To estimate the third sum in~\cref{th:fullrlinear:eq9}, note that there are exactly $L-1$ intermediate levels between consecutive solve levels.
  Hence, \cref{lem:sumsolve} and~\cref{th:fullrlinear:eq7} imply that
  \begin{equation}\label{th:fullrlinear:eq9.3}
    \begin{aligned}
       & \sum_{\ell^\prime\in\II ,\,\ell_0<\ell^\prime}\Hr_{\ell^\prime}^{\kk}
      \eqreff{th:fullrlinear:eq7}{\leq}\sum_{\ell^\prime\in\II ,\,\ell_0<\ell^\prime}
      C_3^{\modop{\ell^\prime}{L}}\,\Hr_{\ell^\prime-\modop{\ell^\prime}{L}}^{\kk} \leq\sum_{\ell^\prime\in\SS,\,\ell_0\leq\ell^\prime}\sum_{j=1}^{L-1}C_3^j\,\Hr_{\ell^\prime}^{\kk} \\ &\qquad\eqreff{lem:sumsolve:eq0}{\leq} \sum_{j=1}^{L-1}C_3^j(1+\const{C}{tail})\,\Hr_{\ell_0}^{\kk} \eqreff{th:fullrlinear:eq7}{\leq} \sum_{j=1}^{L-1}C_3^j(1+\const{C}{tail})\,C_3^{L-1}\,\Hr_{\ell}^{\kk}.
    \end{aligned}
  \end{equation}
  Using $\sum_{j=1}^{L-1}C_3^j=\frac{C_3^L-C_3}{C_3-1}$ and~\crefrange{th:fullrlinear:eq9.1}{th:fullrlinear:eq9.3}, we estimate~\cref{th:fullrlinear:eq9} via
  \begin{equation}\label{th:fullrlinear:eq9.4}
    \sum_{\ell^\prime=\ell+1}^{\ellu-1}\Hr_{\ell^\prime}^{\kk}\leq\Big(\frac{C_3^L-C_3}{C_3-1}+\const{C}{tail}\,C_3^{L-1}+\frac{C_3^L-C_3}{C_3-1}(1+\const{C}{tail})C_3^{L-1}\Big)\,\Hr_\ell^{\kk} \eqcolon C_4\,\Hr_\ell^{\kk}.
  \end{equation}
  The remaining case $\ell_0\geq\ellu$ can only occur if $\ell_0=\ellu\in L\cdotS\N_0$.
  Therefore, $\ellu-\ell\leq L-1$, and thus the stability estimate~\cref{th:fullrlinear:eq7} yields
  \begin{equation}\label{th:fullrlinear:eq9.5}
    \sum_{\ell^\prime=\ell+1}^{\ellu-1}\Hr_{\ell^\prime}^{\kk}\eqreff{th:fullrlinear:eq7}{\leq}\sum_{\ell^\prime=\ell+1}^{\ellu-1}
    C_3^{\ell^\prime-\ell}\,\Hr_{\ell}^{\kk}=\frac{C_3^{\ellu-\ell}-C_3}{C_3-1}\,\Hr_{\ell}^{\kk}\leq \frac{C_3^{L-1}-C_3}{C_3-1}\,\Hr_{\ell}^{\kk}\leq C_4\,\Hr_{\ell}^{\kk}.
  \end{equation}
  Overall, the combination of~\crefrange{th:fullrlinear:eq9.4}{th:fullrlinear:eq9.5} shows that
  \begin{equation}\label{th:fullrlinear:eq10}
    \sum_{\ell^\prime=\ell+1}^{\ellu-1}\Hr_{\ell^\prime}^{\kk}\leq C_4\,\Hr_\ell^{\kk}\quad\text{for all $0\leq\ell<\ellu$}.
  \end{equation}

  \step[tail summability of $\boldsymbol{\Hr_\ell^k}$ in $\boldsymbol{\ell}$ and $\boldsymbol{k}$]{5} For $C_5\coloneq C_3(1+C_4)$, it holds
  \begin{equation}\label{th:fullrlinear:eq10.1}
    \sum_{\ell^\prime=\ell+1}^\ellu\Hr_{\ell^\prime}^0
    \eqreff{th:fullrlinear:eq6}{\leq}
    C_3\sum_{\ell^\prime=\ell}^{\ellu-1}\Hr_{\ell^\prime}^{\kk} \eqreff{th:fullrlinear:eq10}{\leq}C_3(1+C_4)\Hr_\ell^{\kk} \eqreff{th:fullrlinear:eq1}{\leq}C_5\,\Hr_\ell^k.
  \end{equation}
  Set $C_6\coloneq C_2(1+2C_5)$.
  From $C_2\geq 1$,~\cref{th:fullrlinear:eq5} and~\cref{th:fullrlinear:eq10.1}, we obtain
  \begin{equation*}
    \sum_{\substack{(\ell^\prime,k^\prime)\in\QQ \\ \abs{\ell^\prime,k^\prime}>\abs{\ell,k}}}\Hr_{\ell^\prime}^{k^\prime}=\sum_{k^\prime=k+1}^{{\kk}[\ell]}\Hr_{\ell}^{k^\prime}+\sum_{\ell^\prime=\ell+1}^\ellu\sum_{k^\prime=0}^{{\kk}[\ell^\prime]}\Hr_{\ell^\prime}^{k'}\eqreff{th:fullrlinear:eq5}{\leq}
    C_2\Big(\Hr_\ell^k+2\sum_{\ell^\prime=\ell+1}^\ellu\Hr_{\ell^\prime}^0\Big) \eqreff{th:fullrlinear:eq10.1}{\leq} C_6\,\Hr_\ell^k.
  \end{equation*}
  By the preceding estimate, the sequence $\Hr_\ell^k$ is tail summable with constant
  \begin{equation*}
    \begin{aligned}
      C_6
       & =
      \max\bigg\{K,1+\frac{1+\Cstab+\lambda^{-1}\qalg^{-1}(1+\qalg)}{1-\qalg}\bigg\} \\
       & \qquad\quad\times\bigg[1+2C_3+
        \bigg(\frac{C_3^L-C_3}{C_3-1}(2C_3+2C_3^{L}+2\Ctail C_3^{L})+2\const{C}{tail}\,C_3^{L}\bigg)\bigg],
    \end{aligned}
  \end{equation*}
  where $C_3=\Cmon\big[\Crel(1+\Cmon)+1\big]> 1$ stems from~\cref{th:fullrlinear:eq6}.
  Thus,~\cite[Lemma 2]{BFMPS24} yields R-linear convergence~\cref{th:fullrlinear:eq0} with constants $\Clin\leq C_6+1$ and $\qlin\leq(1+C_6^{-1})^{-1}$.
\end{proof}

For sufficiently small $\lambda>0$, the stopping criterion~\cref{algorithm:SAFEM:solverstop} together with uniform contraction~\cref{eq:contractivesolver} guarantees estimator equivalence $\eta_\ell(u_\ell^{\kk})\simeq\eta_\ell(u_\ell^\star)$ on the solve levels $\ell\in\SS$.
The following lemma, essentially from~\cite[Lemma 4.9]{GHPS18}, states this precisely.

\begin{lemma}[estimator equivalence]\label{lem:estimeq}
  Suppose that $\Phi_\ell$ satisfies~\cref{eq:contractivesolver}.
  Define
  \begin{equation}\label{def:lambdaopt}
    \lambda_{\rm opt}\coloneq (1-\qalg)/(\Cstab\qalg).
  \end{equation}
  If $0<\lambda<\lambda_{\rm opt}$, then the final iterates $u_\ell^{\kk}\in\XX_\ell$ generated by \cref{algorithm:SAFEM} satisfy
  \begin{equation}\label{lem:estimeq:eq1}
    (1-\lambda \lambda_{\rm opt}^{-1})\eta_\ell(u_\ell^{\kk})\leq\eta_\ell(u_\ell^\star)\leq(1+\lambda \lambda_{\rm opt}^{-1})\eta_\ell(u_\ell^{\kk})\quad\text{for all $\ell\in\SS$}.
  \end{equation}\qed
\end{lemma}
By a standard perturbation argument, the estimator equivalence~\cref{lem:estimeq:eq1} implies tail summability along solve levels~\cref{lem:sumsolve:eq0} even for arbitrary $\const{C}{alg}>0$.
The extension to the full sequence proceeds analogously to the proof of \cref{th:fullrlinear} and leads to the following result.
A proof is provided in \cref{sec:appendixA}.

\begin{theorem}[full R-linear convergence for arbitrary $\const{C}{alg}>0$]\label{th:Condfullrlinear}
  Let the assumptions of \cref{th:fullrlinear} hold, except that the smoother $\Psi_\ell$ satisfies uniform stability~\cref{eq:stablesolver} with an arbitrary constant $\Calg>0$.
  Recall $\lambda_{\rm opt}$ from~\cref{def:lambdaopt}.
  If $0<\lambda<\theta^{1/2}\lambda_{\rm opt}$, then \cref{algorithm:SAFEM} guarantees full R-linear convergence of the quasi-error $\Hr_\ell^k$ from~\cref{eq:defQuasiErr}, i.e., there exist constants $\Clin>0$ and $0<\qlin<1$ such that
  \begin{equation}\label{th:Condfullrlinear:eq0}
    \Hr_\ell^k\leq\Clin\qlin^{\abs{\ell,k}-\abs{\ell^\prime,k^\prime}}\Hr_{\ell^\prime}^{k^\prime}\quad\text{for all }(\ell^\prime,k^\prime),(\ell,k)\in\QQ\text{ with }\abs{\ell^\prime,k^\prime}\leq\abs{\ell,k}.
  \end{equation}
  The constants $\Clin, \qlin$ depend only on $K, L, \lambda, \theta, \qalg, \Calg$, and the constants in~\crefrange{A1}{A4}, and~\cref{QM}..
\end{theorem}

\section{Quasi-optimal complexity}\label{sec:optrates}

First, we briefly comment on the computational cost of a practical implementation of \cref{algorithm:SAFEM}, where the polynomial degree $p\in\N$ of the finite element spaces~\cref{eq:femspace} is arbitrary, yet fixed.
On solve levels, optimal multigrid methods perform one solve step on $\TT_\ell$ in $\OO(\#\TT_\ell)$ operations, provided the grading of the mesh hierarchy is exploited appropriately~\cite{WZ17, IMPS24}.
For intermediate levels, standard smoothers have linear complexity $\OO(\#\TT_\ell)$, as does the computation of the refinement indicators.
Dörfler marking with quasi-minimal cardinality admits an $\mathcal{O}(\#\TT_\ell)$ implementation; see~\cite{Ste07} for $\const{C}{mark}=2$ and~\cite{PP20} for $\const{C}{mark}=1$.
The cardinality-control step has at most linear cost.
Finally, mesh refinement by NVB is well known to have linear complexity $\mathcal{O}(\#\TT_\ell)$; see, e.g.,~\cite{Ste08, DGS25}.

Since each adaptive step depends on the entire refinement history, the computational cost up to step $(\ell, k)\in\QQ$ is therefore, up to a multiplicative constant, given~by%
\begin{equation}\label{eq:costDef}
  \cost(\ell,k)\coloneq\sum_{\substack{(\ell^\prime, k^\prime)\in\QQ \\ \abs{\ell^\prime, k^\prime}\leq \abs{\ell, k}}}\#\TT_{\ell^\prime}.
\end{equation}

A key implication of full R-linear convergence~\cref{th:fullrlinear:eq0} is the following result, which goes back to~\cite{GHPS21,BFMPS24}.
It states that convergence rates of $\Hr_\ell^k$ with respect to the number of degrees of freedom $\dim\XX_\ell\simeq\#\TT_\ell$ and with respect to $\cost(\ell,k)$ coincide.

\begin{corollary}[rates = complexity]\label{cor:ratesEqComplex}
  Full R-linear convergence~\cref{th:fullrlinear:eq0} implies
  \begin{equation*}
    \sup_{(\ell,k)\in\QQ}(\#\TT_\ell)^s\,\Hr_\ell^k\leq\sup_{(\ell,k)\in\QQ}\cost(\ell,k)^s\,\Hr_\ell^k
    \leq \Ccost\sup_{(\ell,k)\in\QQ}(\#\TT_\ell)^s\,\Hr_\ell^k\quad\text{for all $s>0$,}
  \end{equation*}
  with the constant $\Ccost\coloneq\Clin(1-\qlin^{1/s})^{-s}$.
\end{corollary}

While \cref{th:fullrlinear} guarantees unconditional full R-linear convergence of \cref{algorithm:SAFEM} for $\Calg\leq 1$, sufficiently small adaptivity parameters $\lambda, \theta$ additionally ensure quasi-optimal complexity, i.e., optimal convergence rates with respect to the computational cost (resp.~the overall computational time).
This remains true even for general $\Calg$.
To formalize this, we employ approximation classes~\cite{BDD04, Ste07,CKNS08,CFPP14}.
For $N\in\N_0$, let $\T(N)$ denote the set of refinements $\TT\in\T$ with $\#\TT-\#\TT_0\leq N$.
For $s>0$, define
\begin{equation}\label{def:approxnorm:eq1}
  \| u^\star \|_{\A_s}\coloneq\sup_{N\in\N_0}\bigl[(N+1)^s\min_{\TT_{\rm opt}\in\T(N)}\eta_{\rm opt}(u_{\rm opt}^\star)\bigr]\in [0,+\infty].
\end{equation}
One has $\| u^\star \|_{\A_s}<+\infty$ if and only if the estimator for the exact discrete solutions converges at least with algebraic rate $s>0$ along optimal meshes.

\begin{theorem}[quasi-optimal complexity]\label{th:optrates}
  Assume that the solver $\Phi_\ell$ is uniformly contractive~\cref{eq:contractivesolver} and that the smoother $\Psi_\ell$ is uniformly stable~\cref{eq:stablesolver}.
  Let $K,L\in\N, \const{C}{mark}, \const{C}{card}\geq 1$, and $u_0^0\in\XX_0$ be arbitrary.
  Recall $\lambda_{\rm opt}$ from~\cref{def:lambdaopt}.
  Suppose that $0<\theta\leq 1$ and $\lambda>0$ are sufficiently small such that
  \begin{equation}\label{def:thetaprime}
    0<\lambda<\theta^{1/2}\lambda_{\rm opt}\;\;\text{and}\;\; 0<\theta^\prime\coloneq\frac{(\theta^{1/2}+\lambda/\lambda_{\rm opt})^2}{(1-\lambda/\lambda_{\rm opt})^2}<(1+\Cstab^2\Cdrel^2)^{-1}\eqcolon\theta_{\rm opt}.
  \end{equation}
  Then, for all $s>0$, \cref{algorithm:SAFEM} guarantees the existence of $\const{c}{opt},\const{C}{opt}>0$ satisfying%
  \begin{equation}\label{th:optrates:eq0}
    \const{c}{opt}\|u^\star\|_{\A_s}
    \leq\sup_{(\ell,k)\in\QQ}\cost(\ell,k)^s\,\Hr_\ell^k\leq\const{C}{opt}\max\{\|u^\star\|_{\A_s}, \Hr_0^0\}.
  \end{equation}
  The constant $\const{c}{opt}$ depends only on $\#\TT_0, s$, properties of NVB, and possibly in addition on the minimal index $\ell_0\in \underline{\SS}$ satisfying either $\ell_0=\ellu<+\infty$ or $\eta_{\ell_0}(u_{\ell_0}^{\kk})=0$.
  The constant $\const{C}{opt}$ depends only on $s, K, L, \const{C}{card}, \const{C}{mark}, \theta, \lambda, \qalg$, properties of NVB, and the constants in~\crefrange{A1}{A4},~\cref{A3p},~\cref{QM}; for the precise dependencies see~\cref{th:optrates:eq10}.
\end{theorem}

The significance of~\cref{th:optrates:eq0} is that it establishes that $\|u^\star\|_{\A_s}<+\infty$ holds if and only if $\Hr_\ell^k$ decays with algebraic rate $s>0$ with respect to the overall computational cost.
Since $s>0$ is arbitrary, \cref{algorithm:SAFEM} therefore achieves the best rate permitted by the approximation class (defined in terms of the number of elements) also with respect to the computational cost, and hence attains quasi-optimal complexity.

\begin{remark}
  \cref{th:fullrlinear,th:Condfullrlinear} remain valid in the degenerate case $\Ccard=+\infty$, i.e., when the cardinality-control step \cref{algorithm:SAFEM:3} is omitted.
  In contrast, the proof of \cref{th:optrates} \emph{crucially requires} this step, as without it bounding the number of marked elements on intermediate levels appears unclear.
\end{remark}

\begin{remark}
  Using a direct solver on the solve levels $\ell\in L\cdotS\N_0$ requires at least $\OO(\#\TT_\ell\log(\#\TT_\ell))$ operations.
  Since such solvers are uniformly contractive~\cref{eq:contractivesolver}, \cref{th:optrates} remains applicable and, in combination with \cref{cor:ratesEqComplex}, yields optimal rates with respect to the number of degrees of freedom $\dim\XX_\ell\simeq\#\TT_\ell$.
  However, the computational cost up to $(\ell,k)\in\QQ$ is no longer proportional to $\cost(\ell,k)$ from~\cref{eq:costDef}, so quasi-optimal complexity is no longer ensured.
\end{remark}

\begin{proof}[Proof of \cref{th:optrates}]
  The lower estimate in~\cref{th:optrates:eq0} follows as in~\cite[Lemma~13]{GHPS21}.
  The upper estimate is proved in three steps.
  It suffices to assume that $\|u^\star\|_{\A_s}<+\infty$.

  \step[perturbation argument on solve levels]{1}
  Let $\ell^\prime\in\SS$.
  For $\theta^\prime$ as defined in~\cref{def:thetaprime},~\cite[Lemma 4.14]{CFPP14} ensures the existence of a subset $\RR_{\ell^\prime}\subseteq\TT_{\ell^\prime}$ and constants $\const{C}{cmp}, \const{\widetilde{C}}{cmp}>0$, depending only on $\Cstab, \Cdrel, \Cmon$, such that
  \begin{equation}\label{lem:comparison:eq0}
    \theta^\prime\eta_{{\ell^\prime}}(u_{{\ell^\prime}}^\star)^2\leq\eta_{{\ell^\prime}}(\RR_{{\ell^\prime}}, u_{{\ell^\prime}}^\star)^2\quad\text{and}\quad
    \#\RR_{\ell^\prime}\leq\const{\widetilde{C}}{cmp}\const{C}{cmp}^{1/s}\,\|u^\star\|_{\A_s}^{1/s}\eta_{\ell^\prime}(u_{\ell^\prime}^\star)^{-1/s}.
  \end{equation}
  Using the same perturbation argument as in Step~4 of~\cite[Theorem 8]{GHPS21}, estimator equivalence~\cref{lem:estimeq:eq1} implies $\theta\,\eta_{\ell^\prime}(u_{\ell^\prime}^{\kk})^2\leq\eta_{\ell^\prime}(\RR_{\ell^\prime}, u_{\ell^\prime}^{\kk})^2$.
  Hence, $\RR_{\ell^\prime}$ satisfies the Dörfler criterion with parameter $\theta$.
  By \cref{algorithm:SAFEM}~\crefrange{algorithm:SAFEM:2}{algorithm:SAFEM:3}, $\MM_{\ell^\prime}=\widetilde{\MM}_{\ell^\prime}$ is quasi-minimal fulfilling this criterion.
  Together with~\cref{lem:comparison:eq0} and $C_1\coloneq \const{C}{mark}\const{\widetilde{C}}{cmp}\const{C}{cmp}^{1/s}$, this yields
  \begin{equation}\label{th:optrates:eq1}
    \#\MM_{\ell^\prime}\leq\const{C}{mark}\,\#\RR_{\ell^\prime}\eqreff{lem:comparison:eq0}{\leq}
    C_1 \|u^\star\|_{\A_s}^{1/s}\eta_{\ell^\prime}(u_{\ell^\prime}^\star)^{-1/s}.
  \end{equation}
  With $C_2\coloneq\lambda\qalg/(1-\qalg)$, the contraction~\cref{eq:contractivesolver} of $\Phi_{\ell^\prime}$ and the criterion~\cref{algorithm:SAFEM:solverstop} yield
  \begin{equation}\label{th:optrates:eq2}
    \enorm{u_{\ell^\prime}^\star-u_{\ell^\prime}^{\kk}}\eqreff{eq:contractivesolver}{\leq}\qalg(1-\qalg)^{-1}\,\enorm{u_{\ell^\prime}^{\kk}-u_{\ell^\prime}^{{\kk}-1}}\eqreff{algorithm:SAFEM:solverstop}{\leq}
    C_2\,\eta_{\ell^\prime}(u_{\ell^\prime}^{\kk}).
  \end{equation}
  Hence, for $C_3\coloneq [(1+\Cstab)C_2+1]$, stability~\cref{A1} implies that
  \begin{equation}\label{th:optrates:eq3}
    \Hr_{\ell^\prime}^{\kk}\eqreff{A1}{\leq}(1+\Cstab)\enorm{u_{\ell^\prime}^\star-u_{\ell^\prime}^{\kk}}+\eta_{\ell^\prime}(u_{\ell^\prime}^{\kk})\eqreff{th:optrates:eq2}{\leq}
    C_3\,\eta_{\ell^\prime}(u_{\ell^\prime}^{\kk}).
  \end{equation}
  This and the stability estimate~\cref{th:fullrlinear:eq6} give
  \begin{equation}\label{th:optrates:eq3.1}
    \Hr_{\ell^\prime+1}^0\eqreff{th:fullrlinear:eq6}{\leq}\Cmon\big[\Crel(1+\Cmon)+1\big]\Hr_{\ell^\prime}^{\kk}\eqreff{th:optrates:eq3}{\leq}
    C_3\Cmon\big[\Crel(1+\Cmon)+1\big]\eta_{\ell^\prime}(u_{\ell^\prime}^{\kk}).
  \end{equation}
  For $C_4\coloneq C_1C_3^{1/s}\Cmon^{1/s}[\Crel(1+\Cmon)+1]^{1/s}(1-\lambda \lambda_{\rm opt}^{-1})^{-1/s}$, the combination of~\cref{th:optrates:eq1} and~\cref{th:optrates:eq3.1} with estimator equivalence~\cref{lem:estimeq:eq1} from \cref{lem:estimeq} proves that
  \begin{equation}\label{th:optrates:eq4}
    \#\MM_{\ell^\prime}\leq C_4\,\|u^\star\|_{\A_s}^{1/s}\big(\Hr_{\ell^\prime+1}^0\big)^{-1/s}\quad\text{for all $\ell^\prime\in\SS$.}
  \end{equation}

  \step[extension to all levels]{2}
  Let $\modop{\ell}{L}$ denote the unique integer $0\leq r<L$ such that $\ell-r\in L\cdotS\N_0$.
  Cardinality control in \cref{algorithm:SAFEM}\cref{algorithm:SAFEM:3} ensures
  \begin{equation}\label{th:optrates:eq5}
    \#\MM_{\ell^\prime}\leq\Ccard^{\modop{\ell^\prime}{L}}\,\#\MM_{\ell^\prime-\modop{\ell^\prime}{L}}\quad\text{for all }\ell^\prime\in\II.
  \end{equation}
  Let $(\ell,k)\in\QQ$ with $\TT_\ell\neq\TT_0$.
  Using $\#\TT_\ell>\#\TT_0$ and a mesh-closure estimate for NVB meshes~\cite{Ste08, AFFKP13,DGS25} with constant $\Cmesh>0$ gives
  \begin{equation*}
    \#\TT_\ell-\#\TT_0+1\leq 2\,(\#\TT_\ell-\#\TT_0)
    \leq 2\const{C}{mesh}\bigg(\sum_{\ell^\prime\in \SS ,\,\ell^\prime<\ell}\#\MM_{\ell^\prime}+\sum_{\ell^\prime\in \II ,\,\ell^\prime<\ell}\#\MM_{\ell^\prime}\bigg).
  \end{equation*}
  Applying~\cref{th:optrates:eq5}, the second sum in the preceding formula can be bounded via
  \begin{equation*}
    \sum_{\ell^\prime\in\II ,\,\ell^\prime<\ell}\#\MM_{\ell^\prime}
    \eqreff{th:optrates:eq5}{\leq}\sum_{\ell^\prime\in\II ,\,\ell^\prime<\ell}\Ccard^{\modop{\ell^\prime}{L}}\,\#\MM_{\ell^\prime-\modop{\ell^\prime}{L}}
    \leq\sum_{\ell^\prime\in \SS ,\,\ell^\prime<\ell}\,\sum_{j=1}^{L-1}\Ccard^j\,\#\MM_{\ell^\prime}.
  \end{equation*}
  With $C_5\coloneq 2\const{C}{mesh}[1+\sum_{j=1}^{L-1}\Ccard^j]$, combining the previous two estimates yields
  \begin{equation*}
    \#\TT_\ell-\#\TT_0+1\leqpad C_5\sum_{\ell^\prime\in L\cdotS\N_0 ,\,\ell^\prime<\ell}\#\MM_{\ell^\prime}\eqreff{th:optrates:eq4}{\leq}
    C_4C_5\|u^\star\|_{\A_s}^{1/s}\sum_{\ell^\prime\in L\cdotS\N_0 ,\,\ell^\prime\leq\ell}\big(\Hr_{\ell^\prime}^0\big)^{-1/s}.
  \end{equation*}
  By full R-linear convergence~\cref{th:Condfullrlinear:eq0}, the geometric series gives
  \begin{equation*}
    \sum_{\ell^\prime\in L\cdotS\N_0 ,\,\ell^\prime\leq\ell}\big(\Hr_{\ell^\prime}^0\big)^{-1/s}\leq\sum_{\substack{(\ell^\prime,k^\prime)\in\QQ \\ \abs{\ell^\prime,k^\prime}\leq\abs{\ell,k}}}\big(\Hr_{\ell^\prime}^{k^\prime}\big)^{-1/s}\eqreff{th:Condfullrlinear:eq0}{\leq}\Clin^{1/s}\big(1-\qlin^{1/s}\big)^{-1}\big(\Hr_{\ell}^{k}\big)^{-1/s}.
  \end{equation*}
  With $C_6\coloneq C_4C_5\Clin^{1/s}(1-\qlin^{1/s})^{-1}$, the previous two displayed formulas imply
  \begin{equation*}
    \#\TT_\ell-\#\TT_0+1\leq C_6\|u^\star\|_{\A_s}^{1/s}\big(\Hr_{\ell}^{k}\big)^{-1/s}.
  \end{equation*}
  Exponentiation by $s$ and rearranging the terms thus proves that
  \begin{equation}\label{th:optrates:eq6}
    (\#\TT_\ell-\#\TT_0+1)^s\,\Hr_\ell^k\leq C_6^s\|u^\star\|_{\A_s}\quad\text{for all $(\ell,k)\in\QQ$ with $\TT_\ell\neq\TT_0$.}
  \end{equation}
  Let $(\ell, k)\in\QQ$ with $\TT_\ell=\TT_0$ and $\ell>0$.
  In this case, no refinement has occurred and hence, $\widetilde{\MM}_0=\MM_0=\emptyset$.
  The quasi-minimality condition in \cref{algorithm:SAFEM}\cref{algorithm:SAFEM:2} then implies $\eta_0(u_0^{\kk})=0$.
  By the stopping criterion~\cref{algorithm:SAFEM:solverstop}, it holds $u_0^{\kk}=u_0^{{\kk}-1}$.
  The contraction~\cref{eq:contractivesolver} of the solver then implies $u_0^{\kk}=u_0^\star$, and thus $\eta_0(u_0^\star)=0$.
  Consequently, $\Hr_0^{\kk}=\eta_0(u_0^\star)=0$ and estimate~\cref{th:optrates:eq6} generalizes to
  \begin{equation}\label{th:optrates:eq7}
    (\#\TT_\ell-\#\TT_0+1)^s\,\Hr_\ell^k\leq C_6^s\|u^\star\|_{\A_s}\quad\text{for all $(\ell,k)\in\QQ$ with $\ell>0$.}
  \end{equation}
  Finally, since~\cref{eq:contractivesolver} implies $\Hr_0^k\leq\Hr_0^0$ for all $0\leq k\leq{\kk}[0]$, estimate~\cref{th:optrates:eq7} extends to
  \begin{equation}\label{th:optrates:eq8}
    (\#\TT_\ell-\#\TT_0+1)^s\,\Hr_\ell^k\leq C_6^s\max\{\|u^\star\|_{\A_s}, \Hr_0^0\}\quad\text{for all $(\ell,k)\in\QQ$}.
  \end{equation}

  \step[proof of the upper estimate in~\cref{th:optrates:eq0}]{3} By~\cite[Lemma 22]{BHP17}, it holds
  \begin{equation}\label{th:optrates:eq9}
    \#\TT_{\ell}\leq(\#\TT_0)(\#\TT_{\ell}-\#\TT_0+1)\quad\text{for all $\ell\leq\ellu$}.
  \end{equation}
  For all $(\ell^\prime,k^\prime)\in\QQ$, \cref{cor:ratesEqComplex} and the combination of~\crefrange{th:optrates:eq8}{th:optrates:eq9} show that
  \begin{equation*}
    \cost(\ell^\prime,k^\prime)^s\,\Hr_{\ell^\prime}^{k^\prime}
    \leq\Ccost\sup_{(\ell,k)\in\QQ}(\#\TT_\ell)^s\,\Hr_\ell^k\stackrel{\mathclap{\cref{th:optrates:eq8},\cref{th:optrates:eq9}}}{\quad\;\;\leq\quad\;\;} \Ccost(\#\TT_0)^sC_6^s\max\{\|u^\star\|_{\A_s}, \Hr_0^0\}.
  \end{equation*}
  This concludes the proof of the upper estimate in~\cref{th:optrates:eq0} with constant
  \begin{equation}\label{th:optrates:eq10}
    \const{C}{opt}\leq\frac{\big[1+\sum_{j=1}^{L-1}\Ccard^j\big]^s\big[(1+\Cstab)\frac{\qalg\lambda}{1-\qalg}+1\big]\const{C}{mark}^s\const{\widetilde{C}}{cmp}^{s}\const{C}{cmp}\Clin^2(\#\TT_0)^s}
    {2^{-s}\const{C}{mesh}^{-s}\Cmon^{-1}\big[\Crel(1+\Cmon)+1\big]^{-1}\big(1-\lambda \lambda_{\rm opt}^{-1}\big)\big(1-\qlin^{1/s}\big)^{2s}}.\qedhere
  \end{equation}
\end{proof}

\begin{remark}\label{rem:CC-rates}
  The proof of \cref{th:optrates} verifies the upper bound in~\cref{th:optrates:eq0} by bounding the number of elements $\#\TT_\ell\lesssim(\Hr_\ell^k)^{-1/s}$ on each level $\ell$.
  Step~1 of the proof employed a perturbation argument to bound the number of marked elements $\#\MM_\ell\lesssim(\Hr_\ell^k)^{-1/s}$ on the solve levels $\ell\in\SS$ of \cref{algorithm:SAFEM}.
  This argument applies only on solve levels, where the stopping criterion~\cref{algorithm:SAFEM:solverstop} and uniform contraction~\cref{eq:contractivesolver} ensure the required estimator equivalence~\cref{lem:estimeq:eq1} of \cref{lem:estimeq}.
  The cardinality-control step \cref{algorithm:SAFEM:3} was used in Step~2 to extend this bound to the intermediate levels via estimate~\cref{th:optrates:eq5}.
  In this sense, \cref{th:optrates} \emph{crucially requires} this step.

\end{remark}

\section{Numerical experiments}\label{sec:numExp}

In the following subsections, we investigate three numerical benchmark problems in 2D (one of which is a non-symmetric convection-diffusion problem) and one Poisson problem in 3D.
In view of the extensive numerical experiments in~\cite{HeltMul}, we focus on complementing aspects such as optimal convergence rates and the performance comparison with standard AFEM with inexact solvers.
For reproducibility, the source code and all parameter settings used to generate the numerical results are publicly available in the Code Ocean capsule~\cite{BLP26codeocean}.

\subsection{Implementation}

The two-dimensional experiments were conducted using the object-oriented \textsc{Matlab} software package MooAFEM from~\cite{MooAFEM}, with the $hp$-robust geometric multigrid method from~\cite{IMPS24} employed on the solve levels of \cref{algorithm:SAFEM}.
For the experiments in 3D, we extended the octAFEM3D software package~\cite{octAFEM3D} which uses the mesh-refinement procedure from~\cite{Traxler1997}.

Inhomogeneous Dirichlet boundary data $g\colon\partial\Omega\to\R$ is approximated using nodal interpolation for $d=2$ and the Scott--Zhang quasi-interpolation~\cite{sz1990, AFKPP13} for $d=3$.
This introduces an additional data oscillation term and leads to the extended estimator
\begin{equation*}
  \begin{aligned}
     & \eta_H(T,v_H)^2\coloneq\abs{T}^{2/d}\norm{-\div(\vec{A}\nabla v_H-\vec{f})
    +\vec{b}\cdot\nabla v_H+cv_H-f}^2_{L^2(T)}                                                                  \\
     & \qquad +\abs{T}^{1/d}\norm{\jump{(\vec{A}\nabla v_H-\vec{f})\cdot\vec{n}}}^2_{L^2(\partial T\cap\Omega)}
    +\abs{T}^{1/d}\sum_{E\subseteq \partial T\cap\partial\Omega}\norm{(1-\Pi_E^{p-1})\nabla_\gamma g}^2_{L^2(E)}.
  \end{aligned}
\end{equation*}
Therein, $\nabla_\gamma$ denotes the arc-length derivative for $d=2$ and the surface gradient for $d=3$, and $\Pi_E^{p-1}$ is the $L^2(E)$-orthogonal projection onto the space of polynomials of degree~$p-1$ on the boundary facet $E\subseteq\partial\Omega$.
For further details, we refer to~\cite{AFKPP13, FPP14, BC17}.

All experiments employ \cref{algorithm:SAFEM} with cardinality control constant $\Ccard=10$, initial guess $u_0^0=0$, and, unless stated otherwise, polynomial degree $p=2$ and bulk parameter $\theta = 0.5$.
Dörfler marking in \cref{algorithm:SAFEM}\cref{algorithm:SAFEM:2} is implemented using the binning-based algorithm from~\cite{Ste07} with $\Cmark=2$.
In the cardinality-control step \cref{algorithm:SAFEM}\cref{algorithm:SAFEM:3}, we choose $\MM_\ell=\widetilde{\MM}_\ell$ whenever this choice satisfies~\cref{algorithm:SAFEM:cardcontrol}.
Otherwise, we construct a subset $\MM_\ell\subset\widetilde{\MM}_\ell$ of maximal cardinality satisfying~\cref{algorithm:SAFEM:cardcontrol} by successively selecting elements from the bins in descending order.
The stopping parameter $\lambda>0$ reflects the balance between discretization and algebraic error.
The reader is referred to~\cite{GHPS21,BFMPS24,BMP24} for a numerical investigation and discussion on the choice of $\lambda$.
We compare the \emph{Richardson} iteration (RI), \emph{Gauss--Seidel} iteration (GS), and \emph{PCG with incomplete Cholesky preconditioning} (PCG-iChol) for the intermediate-level smoothing.
Moreover, we employ the \emph{identity} operator (ID) on intermediate mesh levels, corresponding to omitting the smoothing iteration.
All these choices are uniformly stable~\cref{eq:stablesolver} with $\Calg=1$; cf.~\cref{subsec:smoothers:sym}.
For the non-symmetric problem, we employ the uniformly stable Zarantonello symmetrization according to \cref{prop:ZarStable}.

As a reference for the runtime improvements, we employ the standard AFEM~\cite{GHPS21,BFMPS24} with an inexact algebraic solver realized by the S-AFEM \cref{algorithm:SAFEM} with $L = 1$.
We measure the algebraic computation time required to reach a given accuracy in terms of the energy error $\enorm{u^\star-u_\ell^{\kk}}$.
Let $\tts(\ell)$ denote the time required by \cref{algorithm:SAFEM} on level $\ell$ to execute step~\cref{algorithm:SAFEM:1:solve} if $\ell\in L\cdotS\N_0$ or step \cref{algorithm:SAFEM:1:smooth} if $\ell\notin L\cdotS\N_0$.
The remaining steps~\crefrange{algorithm:SAFEM:2}{algorithm:SAFEM:3} are excluded from the measurement as they are identical for AFEM and S-AFEM and, in practice, the algebraic solution dominates the overall runtime.
The cumulative solution time up to level $\ell$ reads $\TTS(\ell)\coloneq\sum_{\ell^\prime=0}^{\ell}\tts(\ell^\prime)$.
We determine a function $e \mapsto \TTS_{\rm ref}(e)$ by piecewise affine interpolation in log-log scale of points $(e_\ell, t_\ell)$ with the energy error $e_\ell \coloneqq \enorm{u^\star-u_\ell^{\kk}}$ and $t_\ell \coloneq\TTS(\ell)$ on the level $\ell$ of the AFEM reference computation ($L=1$).
Evaluating this function at the energy error $\enorm{u^\star-u_\ell^{\kk}}$ attained by S-AFEM allows the definition of the \emph{algebraic speed-up factor} of S-AFEM in relation to AFEM by
\begin{equation}\label{eq:speedupfactor}
  S_\ell\coloneq\frac{\TTS_{\rm ref}\big(\enorm{u^\star-u_\ell^{\kk}}\big)}{\TTS(\ell)}\quad\text{for $\ell\in\N_0$}.
\end{equation}
For all timing measurements, three independent runs were performed and the reported times correspond to the run attaining the median total runtime.

\subsection{2D Poisson problem with interface load}
\label{subsec:interface_load}

We consider the Z-shaped domain $\Omega= (-1, 1)^2\setminus\conv\{(0, 0), (-1, 0), (-1, -1)\}$ and seek $u^\star \in H^1_0(\Omega)$ satisfying
\begin{equation}\label{eq:Hminus1}
  -\Delta u^\star
  = \div\big(\chi_{\omega}(1, 1)^\intercal\big)
  \;\text{in }\Omega,
\end{equation}
where $\chi_{\omega}\colon\Omega\to\{0, 1\}$ denotes the indicator function of $\omega\coloneq\conv\{(1, 0), (1, 1), (0, 1)\}$.
A moderate choice of $\lambda=0.1$ is sufficient for the diffusion problem at hand.

\cref{fig:distribRHS:meshes}~(center) shows that \cref{algorithm:SAFEM} with Gauss--Seidel smoothing resolves both types of singularities, the one at the re-entrant corner and the one induced by the right-hand side in~\cref{eq:Hminus1}.
The resulting mesh is essentially comparable to that obtained for the standard AFEM ($L=1$), see \cref{fig:distribRHS:meshes}~(right).
In contrast, the identity operator in S-AFEM yields noticeably inferior meshes; see \cref{fig:distribRHS:meshes}~(left).
The choices of Gauss--Seidel smoothing or the identity operator in \cref{algorithm:SAFEM} result in optimal convergence rates with respect to the number of degrees of freedom and the cumulative runtime as shown in \cref{fig:distribRHS:conv}.
This confirms the theoretical result of \cref{th:optrates}.
Moreover, \cref{table:distribRHS} compares estimator-weighted cumulative runtimes
\begin{equation}
  \label{eq:weighted_cumulative_time}
  \eta_\ell(u_\ell^{\kk})\TTC(\ell)^{p/2},
\end{equation}
where $\TTC(\ell)$ denotes the cumulative runtime of \cref{algorithm:SAFEM} up to level $\ell$,
for different smoothers, polynomial degrees $p$, and parameters $L,K$.
Since estimator equivalence~\cref{lem:estimeq:eq1} from \cref{lem:estimeq} holds exclusively on solve levels, the S-AFEM algorithm is stopped for $\ell\in L\cdotS\N_0$ as soon as the prescribed tolerance $\eta_\ell(u_\ell^{\kk})<2\cdot 10^{-4}$ is reached.
We observe that intermediate-level smoothing consistently yields significant runtime reductions, emphasizing the computational efficiency of S-AFEM compared to standard AFEM ($L=1$).
For all parameter and smoother choices in this benchmark, cardinality control was rarely activated, i.e., $\MM_\ell=\widetilde{\MM}_\ell$ was permitted in \cref{algorithm:SAFEM}\cref{algorithm:SAFEM:3} on almost all computed mesh levels.
The few exceptions mainly occurred for the Richardson smoother and the identity operator.

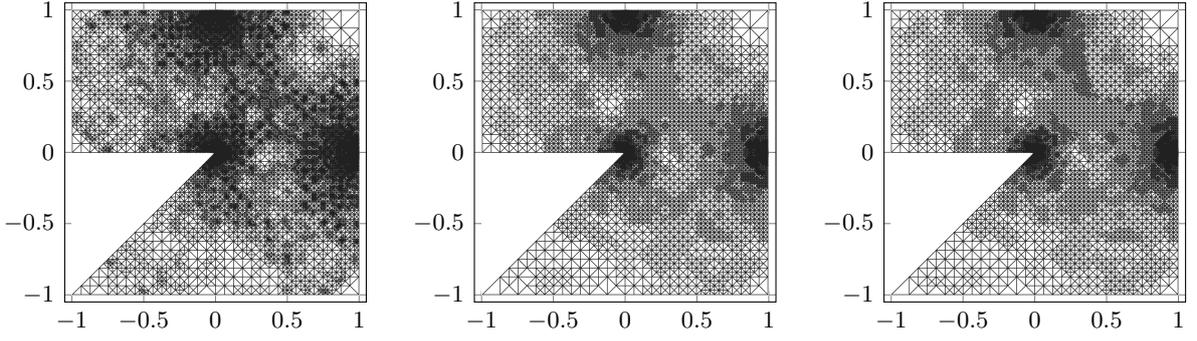
\begin{figure}
  \centering
  \input{plots/distribRHSmeshes/distribRHS_mesh_Id.tex}
  \hfill
  \input{plots/distribRHSmeshes/distribRHS_mesh_GS.tex}
  \hfill
  \input{plots/distribRHSmeshes/distribRHS_mesh_AFEM.tex}
  \caption{Final meshes $\TT_{15}$ on level $\ell=15$ generated by \cref{algorithm:SAFEM} for the Poisson problem~\cref{eq:Hminus1}
    with $L=5$ and ID (left), $L=5=K$ and GS (center), and $L=1$ (AFEM, right).}
  \label{fig:distribRHS:meshes}
\end{figure}

\tikzexternaldisable
\begin{figure}
  \centering
  \input{plots/distribRHS/disitribRHSnDofs.tex}
  \hfill
  \input{plots/distribRHS/disitribRHStime.tex}

  \smallskip
  \centering
  \input{plots/distribRHS/distribRHSLegend.tex}
  \caption{Convergence history of the error estimator $\eta_\ell(u_\ell^{\kk})$ for various polynomial degrees~$p$ in
    \cref{algorithm:SAFEM} applied to the Poisson problem~\cref{eq:Hminus1}.}
  \label{fig:distribRHS:conv}
\end{figure}
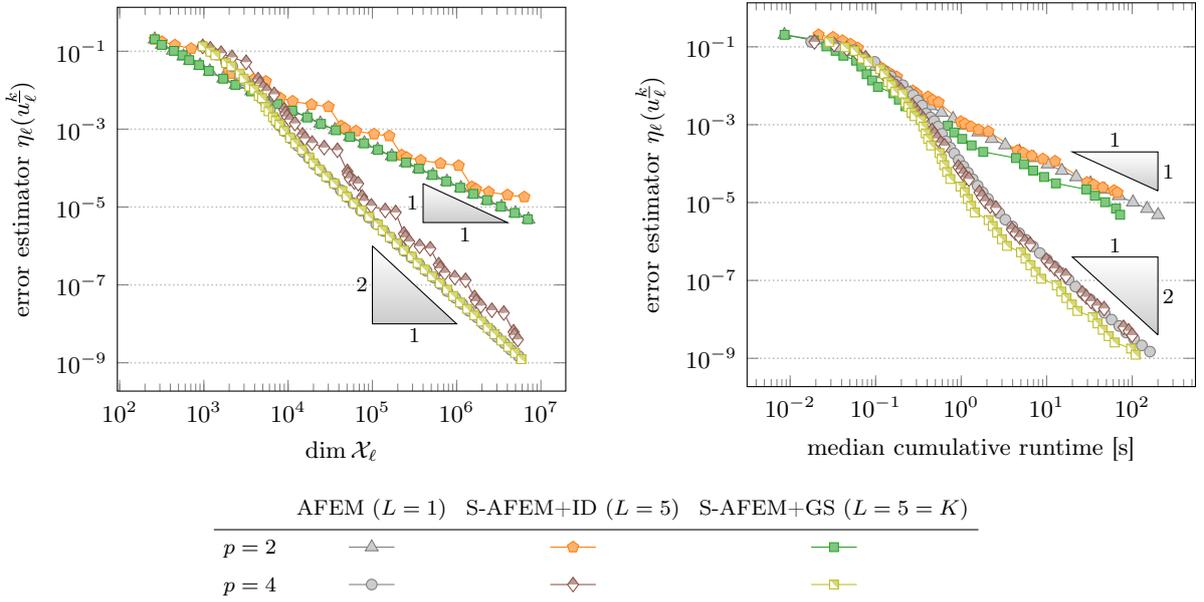
\tikzexternalenable

\begin{table}
  \caption{
    Estimator-weighted cumulative runtime from~\cref{eq:weighted_cumulative_time} (values in $10^{-4}$) for the Poisson problem~\cref{eq:Hminus1} with various smoothers and parameters~$p$, $L$, $K$.
    The fastest runtime in each row is highlighted in blue.
  }
  \label{table:distribRHS}
  \centering
  \input{tables/distribRHScompTime.tex}
\end{table}

\subsection{2D Kellogg problem}
\label{subsec:kellogg}

We consider the classical Kellogg problem on the square domain $\Omega = (-1, 1)^2$ seeking $u^\star \in H^1(\Omega)$ satisfying
\begin{equation}\label{eq:kellogg}
  -\div(\alpha\nabla u^\star) = 0\;\text{in }\Omega
  \quad\text{subject to}\quad
  u^\star= g\;\text{on }\partial\Omega,
\end{equation}
where the diffusion coefficient $\alpha:\Omega\to\R$ is given by $\alpha(x_1, x_2)=161.4476387975881$ if $x_1x_2>0$ and $\alpha(x_1, x_2)=1$ otherwise.
The known exact solution from~\cite{Kel74} satisfies $u^\star\in H^{1.1-\varepsilon}(\Omega)$ for all $\varepsilon>0$ and prescribes the Dirichlet boundary data $g\colon\partial\Omega\to\R$; see, e.g.,~\cite[Section~3]{BMP24} for the explicit expression.
The challenging singularity at the origin motivates a smaller stopping parameter $\lambda=10^{-3}$ to ensure a sufficiently accurate approximation on the solve levels.

\cref{table:Kellogg} reports the resulting algebraic speed-up factors measured on the final mesh.
Inter\-mediate-level smoothing with PCG-iChol yields notable speed-ups, especially for $p\geq 2$, reaching values of up to $S_\ell\approx 8$ for $p=3$.
For $p=1$, the effect is markedly weaker, indicating that the performance difference between PCG-iChol smoothing and the multigrid solver is less pronounced.
Intermediate-level Gauss--Seidel smoothing produces only moderate improvements for the strongly singular Kellogg problem and increasing the number of intermediate smoothing steps~$K$ does not lead to further gains.
Undisplayed numerical experiments suggest that this inferior performance stems from a high number of solver iterations on the solve levels needed to compensate the poor algebraic approximation on the intermediate levels.

\Cref{fig:Kellogg:meshes} displays the adaptive meshes generated by S-AFEM for $L=10$ and $K=30$, together with the AFEM reference ($L=1$).
All three meshes resolve the singularity at the origin comparably well.
For PCG-iChol, the S-AFEM mesh (center) is essentially identical to the AFEM reference mesh (right).
The larger algebraic error on intermediate levels for Gauss--Seidel (left) leads to mild over-refinement.
These observations are consistent with the speed-up factors reported in \cref{table:Kellogg}.

For the Kellogg problem, cardinality control in \cref{algorithm:SAFEM}\cref{algorithm:SAFEM:3} was active on a moderate number of intermediate mesh levels, in particular for Gauss--Seidel smoothing.
\Cref{fig:Kellogg:Ccard} therefore compares the convergence histories for $L=10=K$ and Gauss--Seidel smoothing with $\Ccard=10$ and without cardinality control ($\Ccard=+\infty$).
For $\Ccard=10$, the cardinality bound~\cref{algorithm:SAFEM:cardcontrol} rules out the choice $\MM_\ell=\widetilde{\MM}_\ell$ on 13 of 99 intermediate mesh levels for $p=2$ and on 10 of 171 intermediate mesh levels for $p=4$.
Nevertheless, optimal convergence rates are observed in all displayed runs, independently of cardinality control.

\begin{table}
  \caption{Algebraic speed-up factors $S_\ell$ defined in~\cref{eq:speedupfactor} in the Kellogg problem~\cref{eq:kellogg}.
    The factors are evaluated on the final mesh with respect to the overall stopping criterion $\dim\XX_\ell\geq 10^{5}$.
    The largest value of $S_\ell$ in each row is highlighted in blue.
    Within each $p$-block, the largest value of $S_\ell$ is highlighted in yellow, and in green when both coincide.
  }
  \label{table:Kellogg}
  \centering
  \input{tables/KelloggcompTime.tex}
\end{table}

\begin{figure}
  \centering
  \input{plots/Kelloggmeshes/Kellogg_mesh_GS.tex}
  \hfill
  \input{plots/Kelloggmeshes/Kellogg_mesh_PCG.tex}
  \hfill
  \input{plots/Kelloggmeshes/Kellogg_mesh_AFEM.tex}
  \caption{Final meshes $\TT_{50}$ on level $\ell=50$ generated by \cref{algorithm:SAFEM} for the Kellogg problem~\cref{eq:kellogg} with $L=10$, $K=30$, and GS (left), $L=10$, $K=30$, and PCG-iChol (center), and $L=1$ (AFEM, right).}
  \label{fig:Kellogg:meshes}
\end{figure}

\begin{figure}
  \centering
  \input{plots/KelloggCcard/KelloggCcardnDofs.tex}
  \hfill
  \input{plots/KelloggCcard/KelloggCcardnTime.tex}
  \caption{Convergence history of the error $\protect\enorm{u^\star-u_\ell^\kk}$ for various polynomial degrees~$p$ in \cref{algorithm:SAFEM} with and without cardinality control applied to the Kellogg problem~\cref{eq:kellogg} with $L=10=K$ and GS.}
  \label{fig:Kellogg:Ccard}
\end{figure}

\subsection{2D convection-diffusion on L-shaped domain}

On the L-shaped domain $\Omega = (-1,1)^2 \setminus \big([0,1) \times [-1, 0)\big)$, we consider
\begin{equation}\label{eq:convecDiff}
  -\Delta u^\star +
  (5, 5)^\intercal
  \cdot \nabla u^\star
  = 1\;\text{in }\Omega
  \quad\text{subject to}\quad
  u^\star = 0\;\text{on }\partial\Omega.
\end{equation}
Throughout this subsection, we choose $\lambda=0.05$, $\theta=0.5$ and the Zarantonello damping parameter $\delta=0.5$.
We construct uniformly stable iterative methods via Zarantonello symmetrization according to \cref{prop:ZarStable}.
For $v_H\in\XX_H$, let $\Theta_H^{\rm MG}(v_H)$ and $\Theta_H^{\rm PCG}(v_H)$ denote the iteration operators of the multigrid method from~\cite{IMPS24} and of PCG-iChol, respectively, applied to the symmetrized problem~\cref{eq:zarantonello}.
For the non-symmetric system~\cref{eq:galeq}, define the iteration operators
\begin{equation}\label{eq:convecDiff:solvers}
  \Phi_H(v_H)\coloneq(\Theta_H^{\rm MG}(v_H))^{2}(v_H)\quad\text{and}\quad\Psi_H(v_H)\coloneq(\Theta_H^{\rm PCG}(v_H))^{4}(v_H),
\end{equation}
i.e., applying, e.g., $\Phi_H$ corresponds to two multigrid iterations on the symmetrized problem~\cref{eq:zarantonello}.
The adaptive meshes computed by S-AFEM are displayed in \cref{fig:convDiff:meshes} and exhibit slight under-refinement at the boundary layer for $L > 1$ which, however, does not impact on the optimal convergence rates with respect to the cumulative runtime in \cref{fig:convDiff}~(left).
The number of solver iterations on the solve levels remains moderate; see \cref{fig:convDiff}~(right).
Finally, \cref{table:convDiff} indicates that, for the non-symmetric problem~\cref{eq:convecDiff}, S-AFEM yields noticeable computational advantages over standard AFEM ($L=1$), although less significant than for the previous symmetric problems.
For all parameter choices in this benchmark, $\MM_\ell=\widetilde{\MM}_\ell$ was admissible in \cref{algorithm:SAFEM}\cref{algorithm:SAFEM:3} on all computed mesh levels, i.e., cardinality control was never activated.

\begin{figure}
  \centering
  \input{plots/convDiffmeshes/convDiff_mesh_L5.tex}
  \hfill
  \input{plots/convDiffmeshes/convDiff_mesh_L10.tex}
  \hfill
  \input{plots/convDiffmeshes/convDiff_mesh_AFEM.tex}
  \caption{Final meshes $\TT_{20}$ on level $\ell=20$ generated by \cref{algorithm:SAFEM} for the convection-diffusion problem~\cref{eq:convecDiff} with $p=3$ for $K=3, L=5$ (left), $K=3, L=10$ (center), and $L=1$ (AFEM, right).}
  \label{fig:convDiff:meshes}
\end{figure}
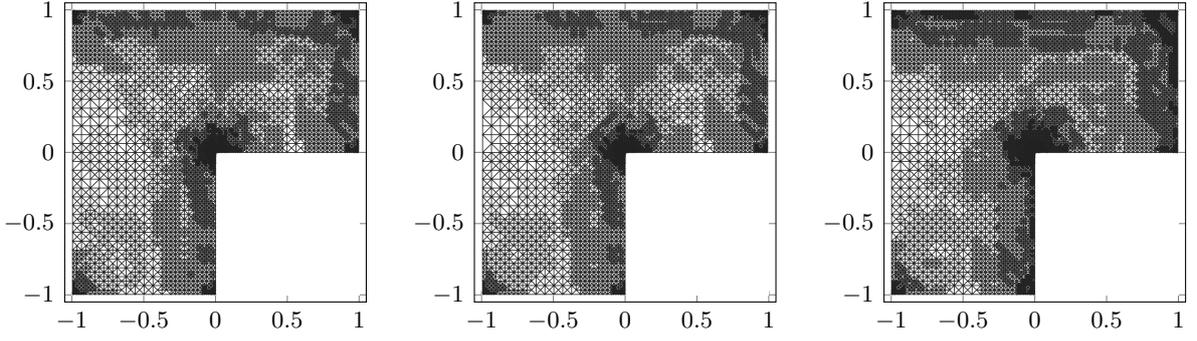

\begin{figure}
  \centering
  \input{plots/convDiff/convDiff_runtime.tex}
  \hfill
  \input{plots/convDiff/conv_diff_solverIterations.tex}
  \caption{Convergence of the error estimator $\eta_\ell(u_\ell^{\kk})$ (left) and the number of solver iterations (right) on solve levels for the convection-diffusion problem~\cref{eq:convecDiff} with $p=3$ for $K=3$.}
  \label{fig:convDiff}
\end{figure}
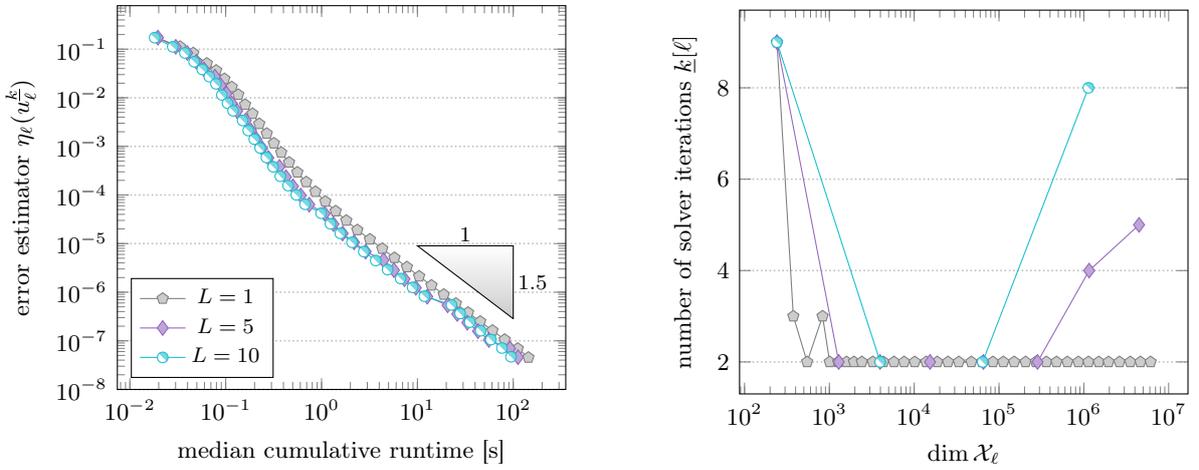

\begin{table}
  \caption{Estimator-weighted cumulative runtime from~\cref{eq:weighted_cumulative_time} (values in $10^{-3}$) of \cref{algorithm:SAFEM} for the convection-diffusion problem~\cref{eq:convecDiff}.
    The fastest runtime in each row is highlighted in blue.
  }
  \label{table:convDiff}
  \centering
  \input{tables/convDiffcompTime.tex}
\end{table}

\subsection{3D Poisson problem on Fichera cube}

In three spatial dimensions, we consider the Poisson problem on the Fichera cube $\Omega = (-1,1)^3 \setminus [0,1)^3$ given by
\begin{equation}\label{eq:Fichera}
  -\Delta u^\star = f\;\text{in }\Omega
  \quad\text{subject to}\quad
  u^\star= g\;\text{on }\partial\Omega,
\end{equation}
where the forcing term $f\colon\Omega\to\R$ and the Dirichlet data $g\colon\partial\Omega\to\R$ are chosen as in~\cite{HeltMul} so that the exact solution in spherical coordinates is $u^\star(r,\varphi,\theta) = r^{1/2}$.
We fix $\theta=0.3, p=1$ and use \textsc{Matlab}'s built-in direct solver \texttt{mldivide} on the solve levels.

\cref{fig:ficheraSolMesh} illustrates the S-AFEM approximation and the adapted final mesh, showing that few non-preconditioned intermediate smoothing steps yield effective adaptivity near the origin.
We investigate the algebraic speed-up factors $S_\ell$ from~\cref{eq:speedupfactor} in \cref{fig:Fichera:speedUp} and observe that S-AFEM also yields substantial speed-ups for the 3D problem, achieving runtime improvements of up to $S_\ell\approx 6$ relative to standard AFEM.
For the 3D Fichera problem, the choice $\MM_\ell=\widetilde{\MM}_\ell$ was permitted in \cref{algorithm:SAFEM}\cref{algorithm:SAFEM:3} on all computed mesh levels, i.e., cardinality control was never activated.

Finally, we propose a practice-oriented strategy for the use of S-AFEM in realistic computations.
Given a target tolerance $\tau>0$ for the error estimator, perform a moderate number of initial solve-estimate-mark-refine loops with very inexpensive direct solves on coarse meshes.
Based on the resulting log-log relation between the estimator and the number of degrees of freedom, a linear regression is used to predict the number of degrees of freedom required to reach $\tau$.
The adaptive loop is then continued using only a fixed number of smoothing steps per mesh, as in \cref{algorithm:SAFEM} on intermediate mesh levels, until this threshold is reached.
We terminate with a final exact (or sufficiently accurate) solve.
This allows an acceptable prediction of the number of intermediate smoothing levels in the S-AFEM algorithm as used in~\cite{HeltMul}.
\cref{fig:Fichera:thresh} illustrates the efficiency of the strategy for the model problem at hand.

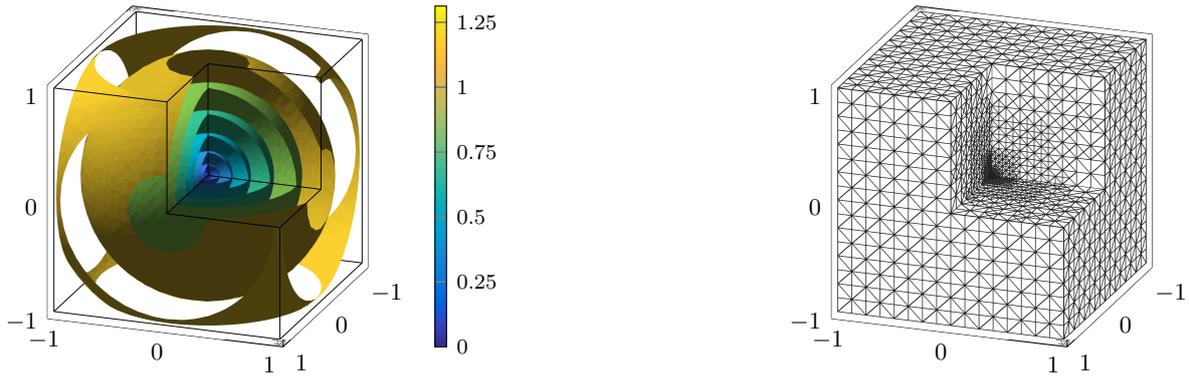
\begin{figure}
  \centering
  \input{plots/ficheraSolutionMesh/ficheraSolution.tex}
  \hfill
  \input{plots/ficheraSolutionMesh/ficheraMesh.tex}
  \caption{Isosurfaces of the S-AFEM approximation $u_{30}^\star$ (left) and final mesh with 61092 simplices (right) on level $\ell=30$ generated by \cref{algorithm:SAFEM}
    with $K=5$ Gauss--Seidel smoothing iterations and $L=10$ for the 3D Fichera problem~\cref{eq:Fichera}.}
  \label{fig:ficheraSolMesh}
\end{figure}

\tikzexternaldisable
\begin{figure}
  \centering
  \input{plots/ficheraSpeedUp/Fichera_speedUp.tex}
  \hfil
  \raisebox{2cm}{\input{plots/ficheraSpeedUp/Fichera_speedUp_legend.tex}}
  \caption{Algebraic speed-up factors $S_\ell$ from~\cref{eq:speedupfactor} for $\ell\in L\cdotS\N$ of \cref{algorithm:SAFEM} relative to standard AFEM ($L=1$) for the 3D Fichera problem~\cref{eq:Fichera}.}
  \label{fig:Fichera:speedUp}
\end{figure}
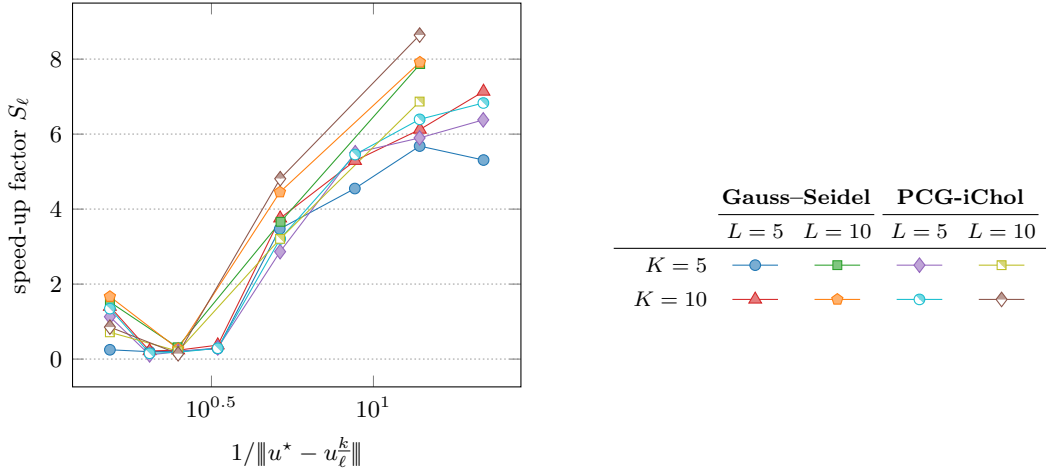
\tikzexternalenable

\begin{figure}
  \centering
  \input{plots/ficheraThresh/ficheraThreshNdofs.tex}
  \hfill
  \raisebox{3cm}{\input{tables/FicheraThreshTime.tex}}
  \caption{Convergence history (left) and solution time (right) for the 3D Fichera problem~\cref{eq:Fichera}.
    After direct solves on the first 20 meshes, reaching estimator threshold $\tau=0.15$ is predicted to require $\dim\XX_\ell=165444$.
    Thereafter, S-AFEM applies $K=5$ Gauss--Seidel iterations per mesh until a final direct solve is performed.
    The AFEM reference is computed with $L=1$ up to the predicted threshold.
    Hollow markers indicate smoothing.
  }
  \label{fig:Fichera:thresh}
\end{figure}
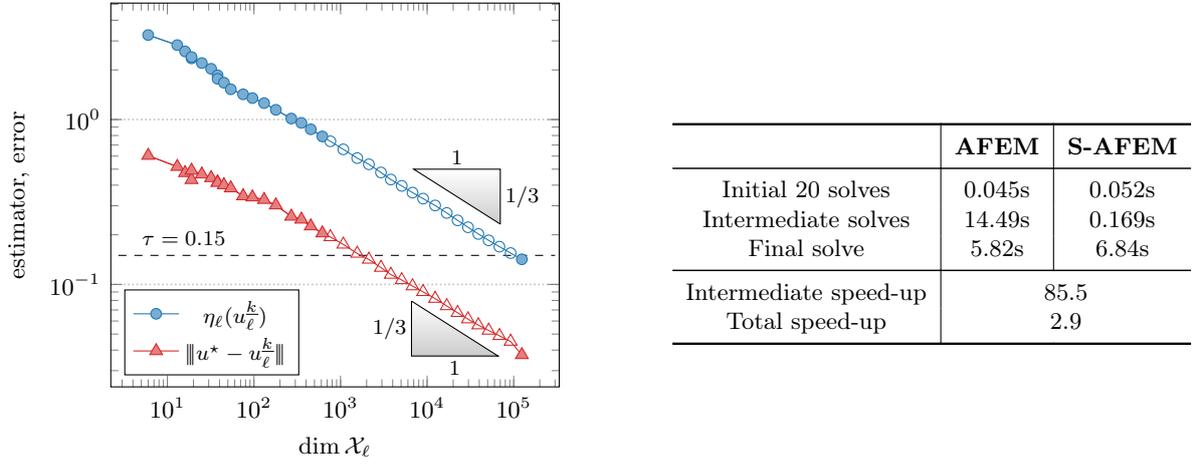

\subsection{Conclusion}

Adaptive mesh-refinement algorithms with inexact algebraic solvers allow the efficient computation of finite element solutions to PDEs in quasi-optimal runtime.
The numerical experiments revealed that intermediate smoothing steps suffice to generate comparably good meshes and further reduce the overall computational effort considerably.
The cardinality-control step in \cref{algorithm:SAFEM} was essentially only activated for the strongly singular Kellogg problem.
However, omitting it did not reduce the observed optimal convergence rates.
This suggests that cardinality control is not very restrictive in practice.
Even very inexpensive standard algebraic iterations without or with light preconditioning yield sound results.
This is particularly important in practical computations, where one is interested in a good approximation only on a final mesh rather than on all intermediate meshes.
We demonstrated that a small number of accurate solution steps in the beginning can be used to predict the required number of smoothing steps to compute this final mesh.

%% file: plots/indicators/indicatorsInitial.tex
\begin{tikzpicture}
  \input{plots/pgfplotStyleLietz.tex}
  \begin{axis}[%
      name = plot,
      axis equal image,%
      width = \indicatorWidth,%
      xmin=-1.05, xmax=1.05,%
      ymin=-1.05, ymax=1.05,%
      font=\scriptsize,%
      xtick={-1,0,1},
      ytick={-1,0,1},
      point meta min=-7.22460693e+00,
      point meta max=-3.76829305e+00,%
      colorbar,%
      colorbar style={%
          name = cb,
          at={(1.05,0.5)},
          anchor=west,
          title={$\eta(T)$},%
          font=\scriptsize,%
          width=1.5mm,%
          title style={yshift=-2mm},%
          yticklabel={$10^{\pgfmathprintnumber{\tick}}$},%
          ticklabel shift=-2pt,
        },%
    ]

    \addplot graphics [xmin=-1, xmax=1, ymin=-1, ymax=1]
      {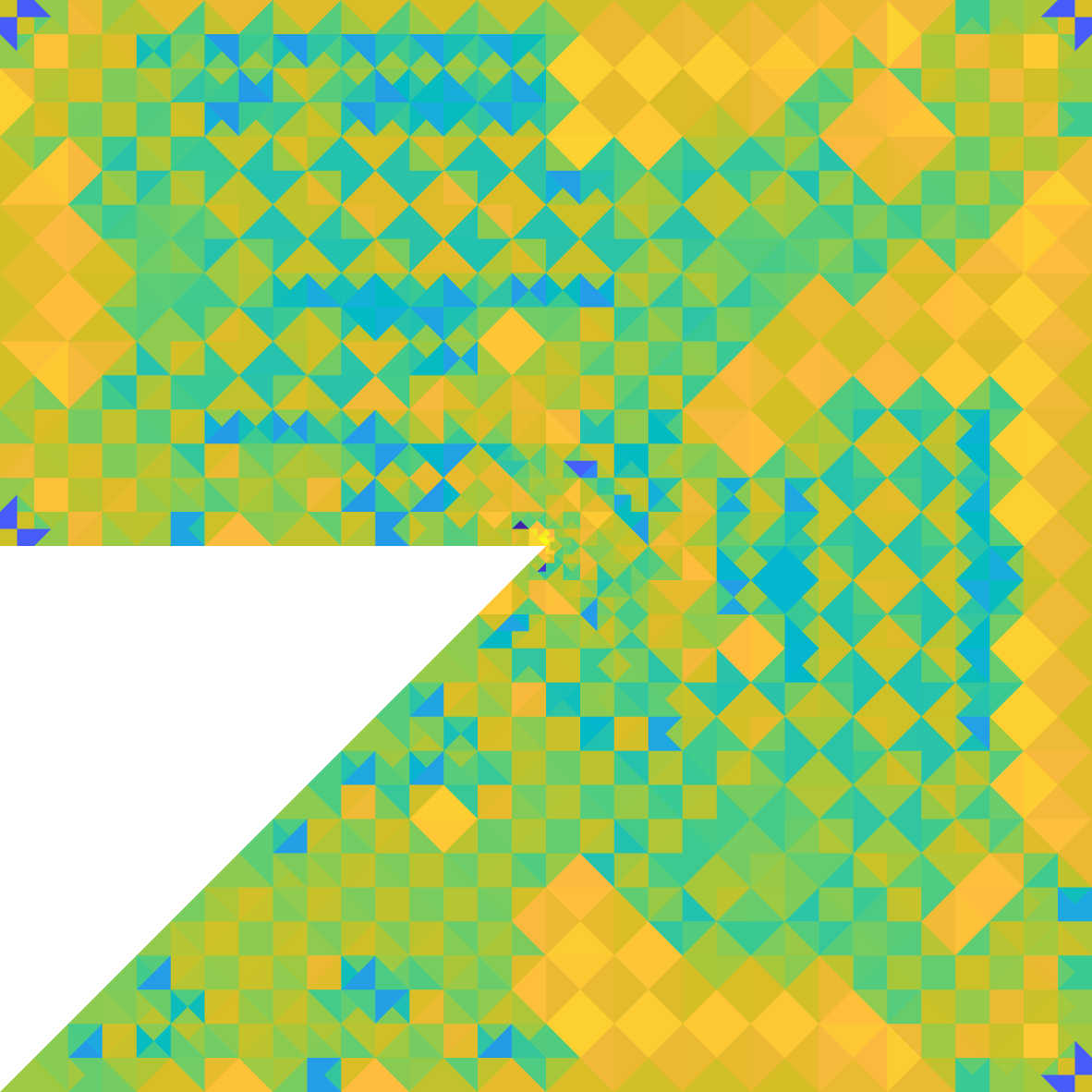};

  \end{axis}
\end{tikzpicture}

%% file: plots/indicators/indicatorsSmoothed.tex
\begin{tikzpicture}
  \input{plots/pgfplotStyleLietz.tex}
  \begin{axis}[%
      name = plot,
      axis equal image,%
      width = \indicatorWidth,%
      xmin=-1.05, xmax=1.05,%
      ymin=-1.05, ymax=1.05,%
      font=\scriptsize,%
      xtick={-1,0,1},
      ytick={-1,0,1},
      point meta min=-5.78593034e+00,
      point meta max=-3.39415996e+00,%
      colorbar,%
      colorbar style={%
          name = cb,
          at={(1.05,0.5)},
          anchor=west,
          title={$\eta(T)$},%
          font=\scriptsize,%
          width=1.5mm,%
          title style={yshift=-2mm},%
          yticklabel={$10^{\pgfmathprintnumber{\tick}}$},%
          ticklabel shift=-2pt,
        },%
    ]

    \addplot graphics [xmin=-1, xmax=1, ymin=-1, ymax=1]
      {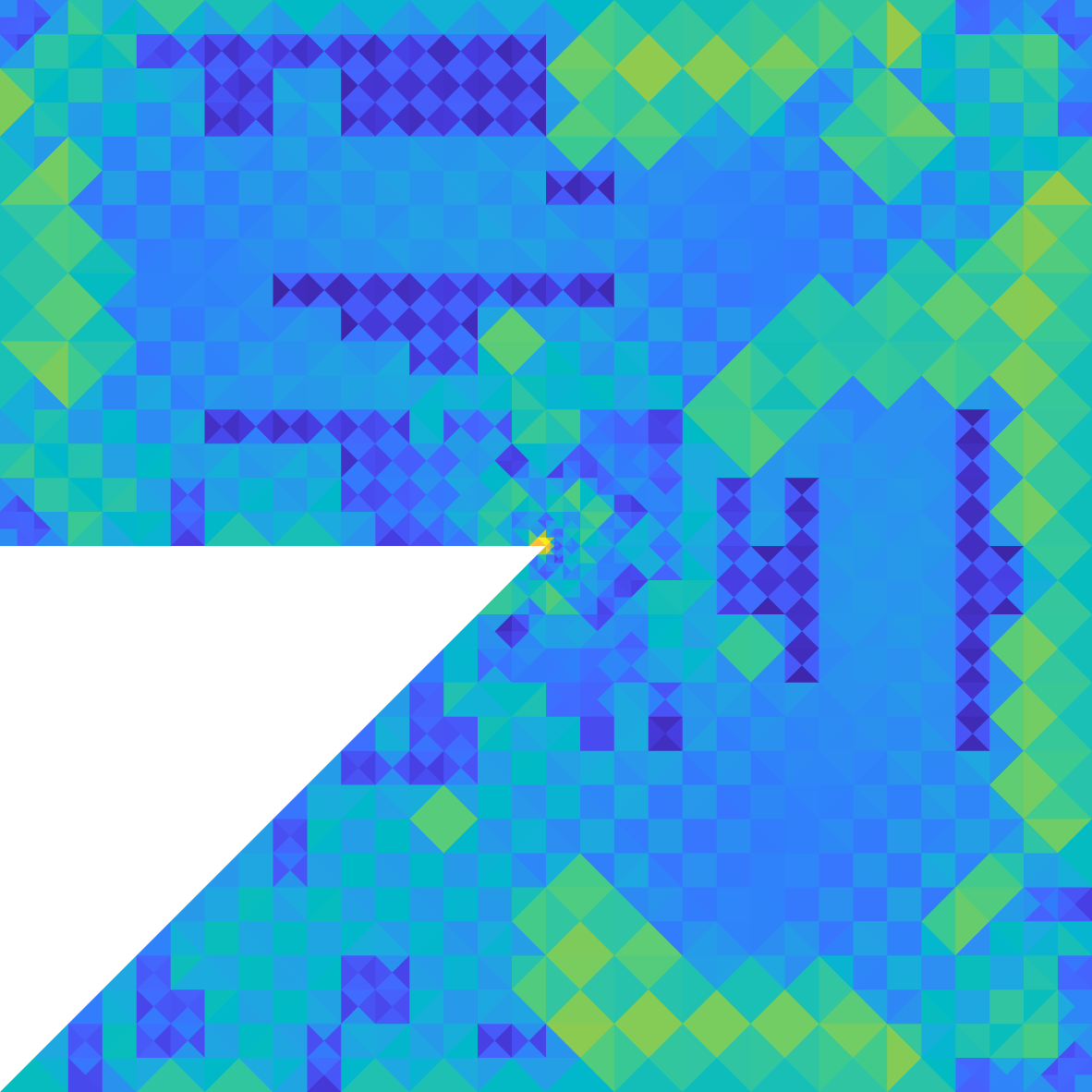};

  \end{axis}
\end{tikzpicture}

%% file: plots/indicators/indicatorsExact.tex
\begin{tikzpicture}
  \input{plots/pgfplotStyleLietz.tex}
  \begin{axis}[%
      name = plot,
      axis equal image,%
      width = \indicatorWidth,%
      xmin=-1.05, xmax=1.05,%
      ymin=-1.05, ymax=1.05,%
      font=\scriptsize,%
      xtick={-1,0,1},
      ytick={-1,0,1},
      point meta min=-5.78593034e+00,
      point meta max=-3.39415996e+00,%
      colorbar,%
      colorbar style={%
          name = cb,
          at={(1.05,0.5)},
          anchor=west,
          title={$\eta(T)$},%
          font=\scriptsize,%
          width=1.5mm,%
          title style={yshift=-2mm},%
          yticklabel={$10^{\pgfmathprintnumber{\tick}}$},%
          ticklabel shift=-2pt,
        },%
    ]

    \addplot graphics [xmin=-1, xmax=1, ymin=-1, ymax=1]
      {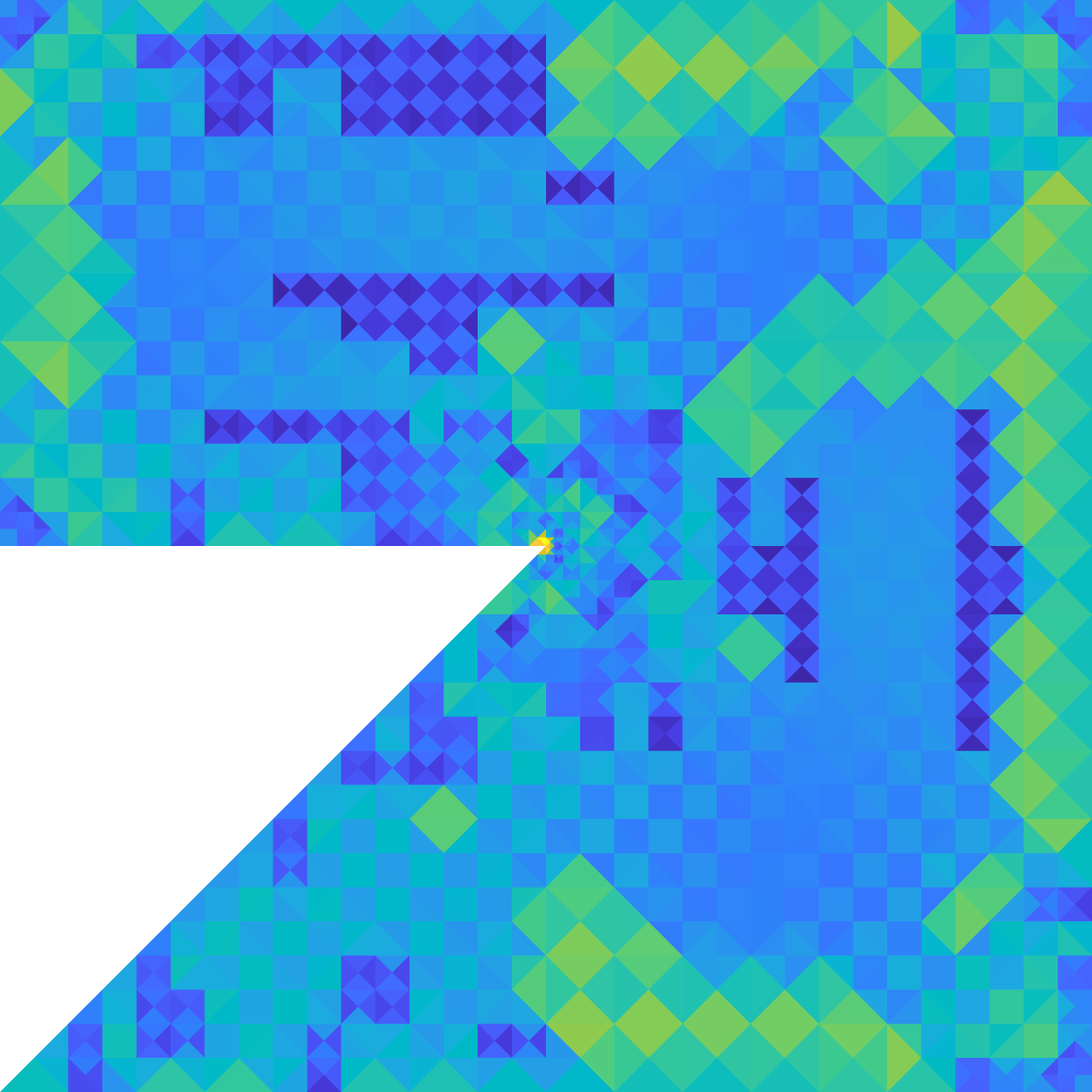};

  \end{axis}
\end{tikzpicture}

%% file: figures/loops.tex
\begin{tikzpicture}

  \tikzset{myline/.style={draw=black, line width=1.0pt, rounded corners}}
  \tikzset{afemnode/.style={myline, fill=white, minimum width={width("ESTIMATE")}}}
  \tikzset{node distance = 7mm}
  \node[afemnode] (solve) at (0,0) {\texttt{SOLVE}};
  \node[afemnode,right=of solve] (estimate) {\texttt{ESTIMATE}};
  \node[afemnode,right=of estimate] (mark) {\texttt{MARK}};
  \node[afemnode,right=of mark] (refine) {\texttt{REFINE}};
  \node[afemnode,right=of refine] (solve2) {\texttt{SOLVE}};
  \node[afemnode,draw=TUMagenta!80!black, fill=white, below=of mark] (smooth) {\color{TUMagenta!80!black}{\texttt{SMOOTH}}};
  \node (afembottom) at ([yshift=-7pt]solve.south) {}; 
  \node (safembottom) at ([yshift=3pt]smooth.center) {}; 

  \tikzset{connection/.style={myline, line width=1.2pt, black, -stealth, shorten >=4pt, shorten <=4pt}}
  \tikzset{connectionDashed/.style={myline, line width=1.2pt, black, -stealth, dashed, shorten >=4pt, shorten <=4pt}}
  \draw[connection] (solve) -- (estimate);
  \draw[connection] (estimate) -- (mark);
  \draw[connection] (mark) -- (refine);
  \draw[connection] (refine) -- (smooth);
  \draw[connection] (smooth) -- (estimate);
  \draw[connectionDashed] (refine) -- (solve2);

  \begin{pgfonlayer}{background}

    \node[
      myline,
      draw=TUMagenta!80!black,
      line width=1.5pt,
      fill=TUMagenta!18,
      fit=(afembottom)(solve)(estimate)(mark)(refine)(safembottom)(solve2),
      inner xsep=9pt,
      inner ysep=10pt,
    ] (safembox){};

    \node[
      myline,
      draw=black,
      line width=1.5pt,
      fill=TULightGrey,
      fit=(afembottom)(solve)(estimate)(mark)(refine)(solve2),
      inner xsep=3pt,
      inner ysep=4pt,
    ] (afembox){};

    \node[
      font = \ttfamily,
      text = black,
      anchor = south west
    ] at (afembox.south west) {Classical AFEM loop};

    \node[
      font = \ttfamily,
      text = TUMagenta!80!black,
      anchor = south west
    ] at (safembox.south west) {S-AFEM loop};

  \end{pgfonlayer}
\end{tikzpicture}

%% file: plots/distribRHSmeshes/distribRHS_mesh_Id.tex
\begin{tikzpicture}
  \begin{axis}[%
      axis equal image,%
      width=\meshWidth,%
      xmin=-1.05, xmax=1.05,%
      ymin=-1.05, ymax=1.05,%
      font=\footnotesize%
    ]
    \addplot graphics [xmin=-1, xmax=1, ymin=-1, ymax=1]
      {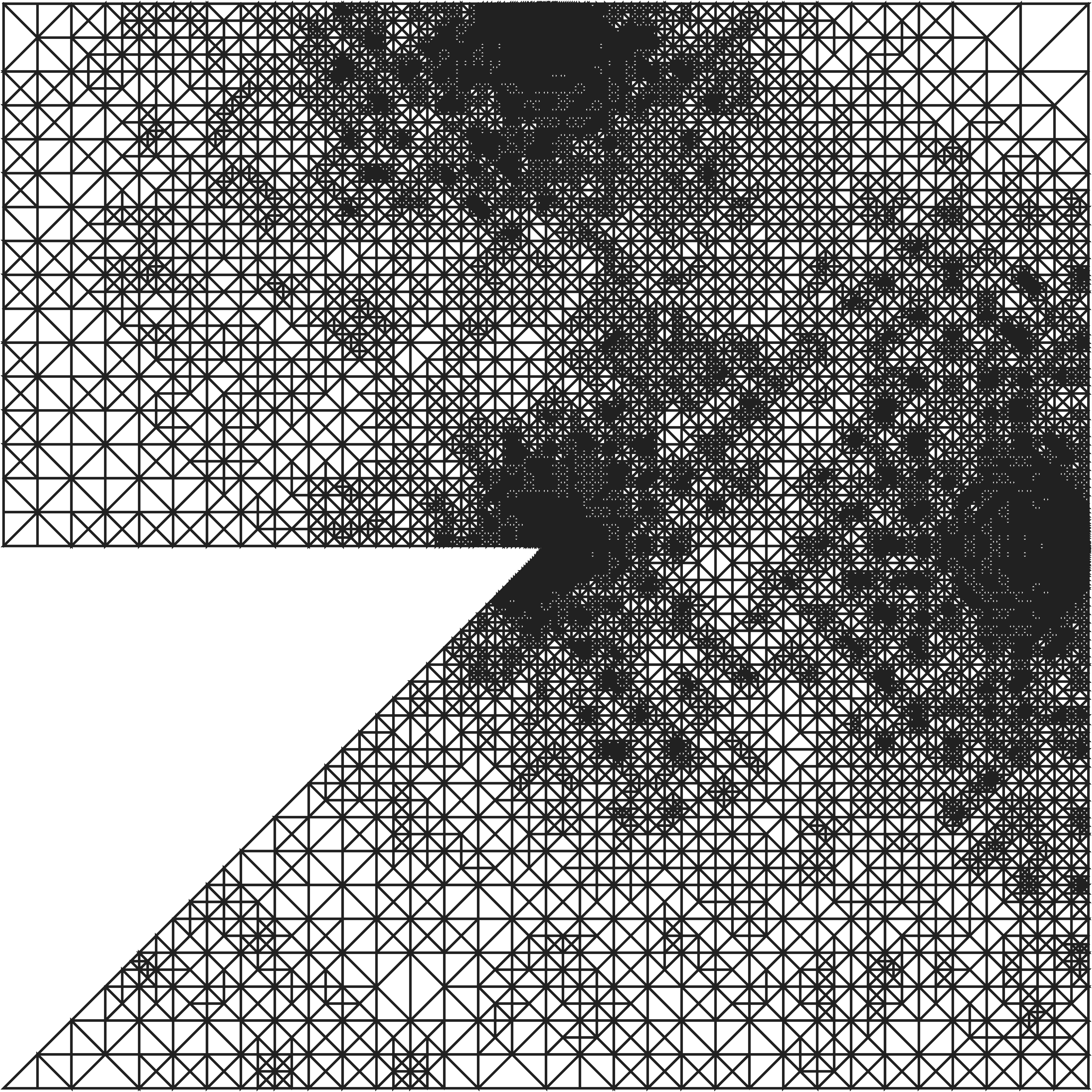};
  \end{axis}

\end{tikzpicture}

%% file: plots/distribRHSmeshes/distribRHS_mesh_GS.tex
\begin{tikzpicture}
  \begin{axis}[%
      axis equal image,%
      width=\meshWidth,%
      xmin=-1.05, xmax=1.05,%
      ymin=-1.05, ymax=1.05,%
      font=\footnotesize%
    ]
    \addplot graphics [xmin=-1, xmax=1, ymin=-1, ymax=1]
      {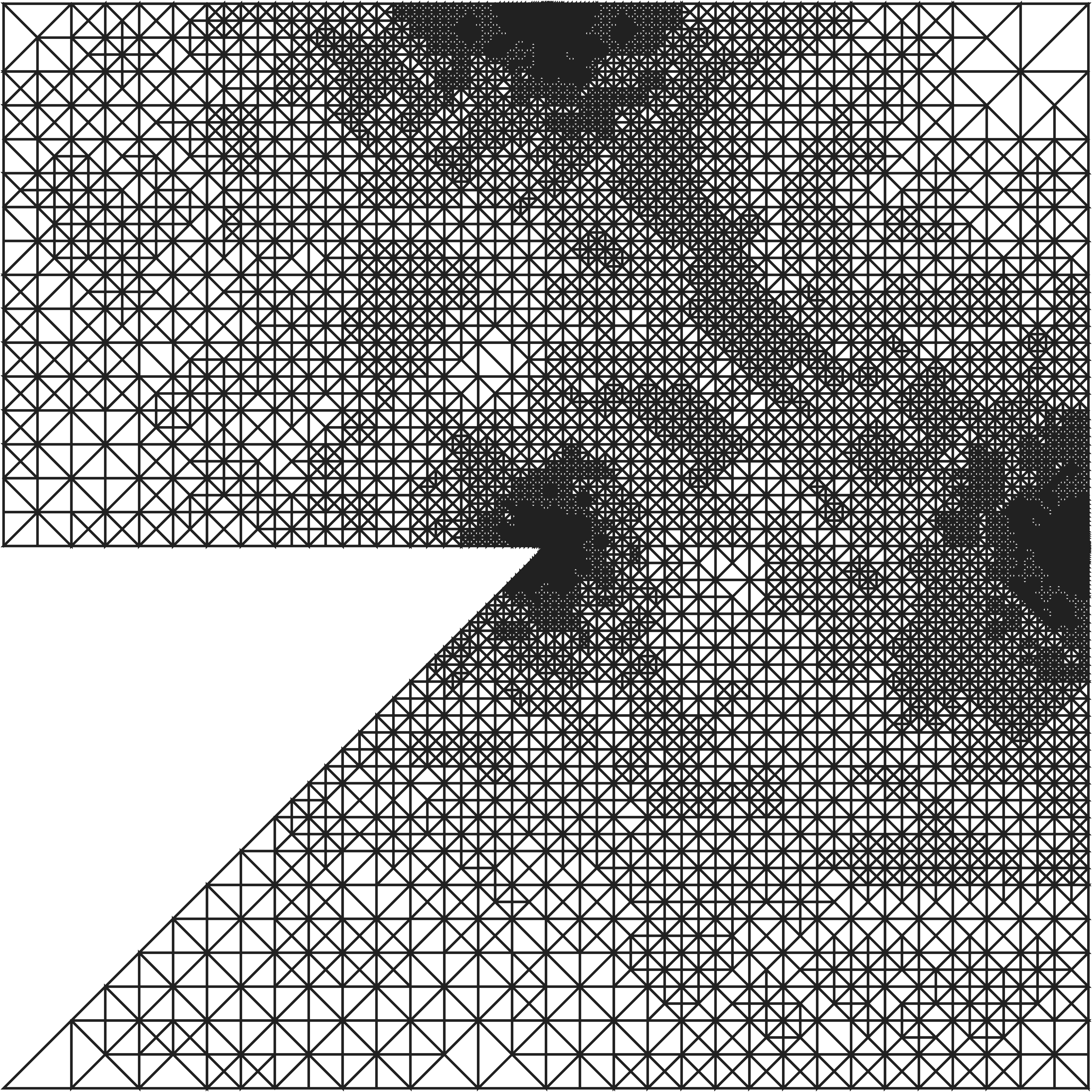};
  \end{axis}
\end{tikzpicture}

%% file: plots/distribRHSmeshes/distribRHS_mesh_AFEM.tex
\begin{tikzpicture}
  \begin{axis}[%
      axis equal image,%
      width=\meshWidth,%
      xmin=-1.05, xmax=1.05,%
      ymin=-1.05, ymax=1.05,%
      font=\footnotesize%
    ]
    \addplot graphics [xmin=-1, xmax=1, ymin=-1, ymax=1]
      {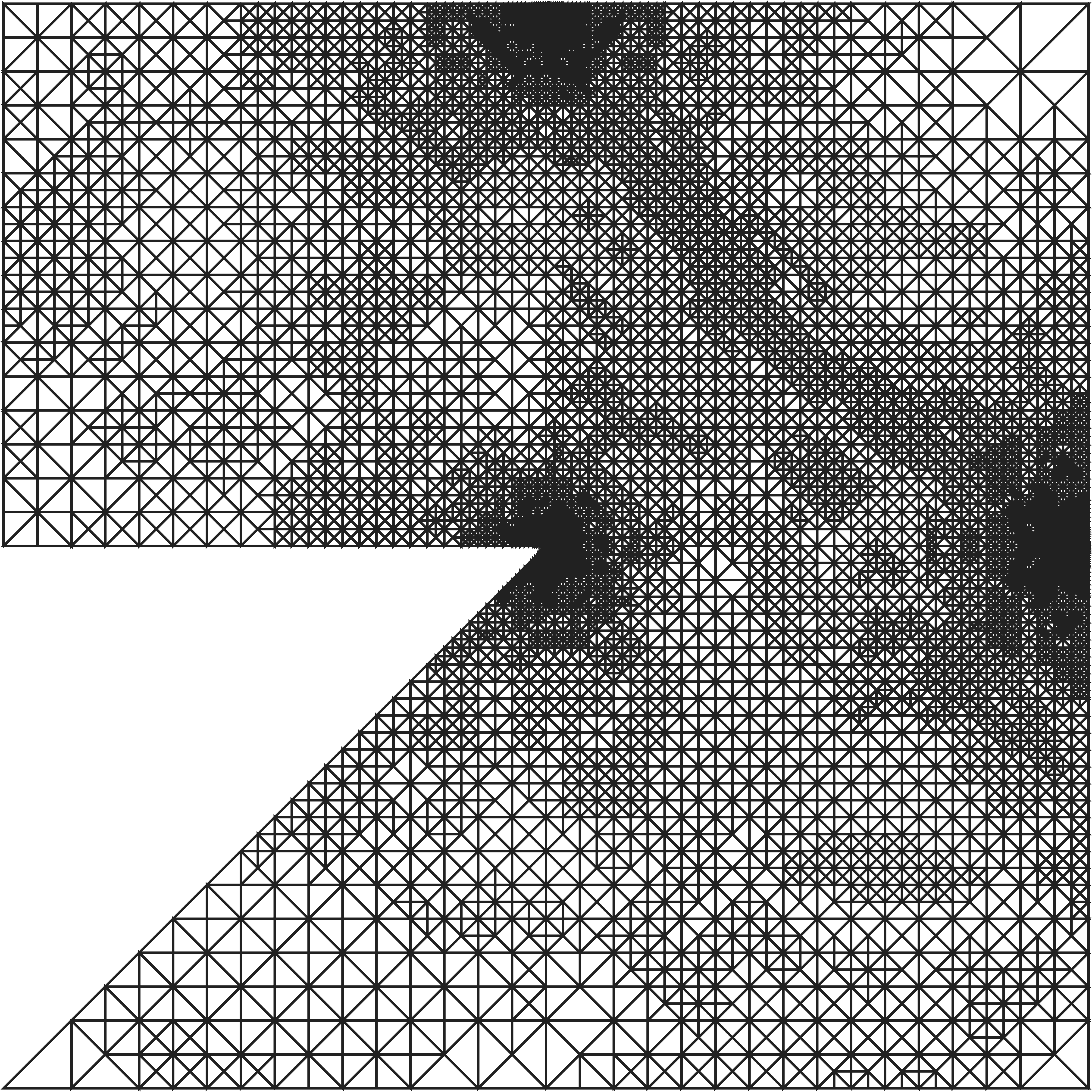};
  \end{axis}
\end{tikzpicture}

%% file: plots/distribRHS/disitribRHSnDofs.tex
\begin{tikzpicture}[>=stealth]


  \pgfplotstableread[col sep=comma]{plots/distribRHS/p2_L1_K0_CcardInf_theta50_solver-multigrid-lowOrderVcycle_lambda0.1_smoother-_nStepsInf.csv}\StwoAFEM

  \pgfplotstableread[col sep=comma]{plots/distribRHS/p2_L5_K5_Ccard10_theta50_solver-multigrid-lowOrderVcycle_lambda0.1_smoother-gaussSeidel_nStepsInf.csv}\StwoGS

  \pgfplotstableread[col sep=comma]{plots/distribRHS/p2_L5_K5_Ccard10_theta50_solver-multigrid-lowOrderVcycle_lambda0.1_smoother-identity_nStepsInf.csv}\StwoID

  \pgfplotstableread[col sep=comma]{plots/distribRHS/p4_L1_K0_CcardInf_theta50_solver-multigrid-lowOrderVcycle_lambda0.1_smoother-_nStepsInf.csv}\SfourAFEM

  \pgfplotstableread[col sep=comma]{plots/distribRHS/p4_L5_K5_Ccard10_theta50_solver-multigrid-lowOrderVcycle_lambda0.1_smoother-gaussSeidel_nStepsInf.csv}\SfourGS

  \pgfplotstableread[col sep=comma]{plots/distribRHS/p4_L5_K5_Ccard10_theta50_solver-multigrid-lowOrderVcycle_lambda0.1_smoother-identity_nStepsInf.csv}\SfourID

  \input{plots/pgfplotStyleLietz.tex}

  \begin{loglogaxis}[%
      width            = \convergenceWidth,%
      xlabel           = {$\dim\XX_\ell$},%
      ylabel           = {error estimator $\eta_\ell(u_\ell^\kk)$},%
      ymajorgrids      = true,%
      font             = \footnotesize,%
      grid style       = {%
          densely dotted,%
          semithick%
        },%
      legend style     = {%
          legend pos  = south west,%
          font = \scriptsize%
        },%
    ]


    \addplot+ [p2L1, iterative, forget plot]
    table [x={nDofs}, y={eta}] {\StwoAFEM};
    \label{leg:distribRHS:S2AFEM}

    \addplot+ [p2L3, iterative, forget plot]
    table [x={nDofs}, y={eta}] {\StwoID};
    \label{leg:distribRHS:S2ID}

    \addplot+ [p2L2, iterative, forget plot]
    table [x={nDofs}, y={eta}] {\StwoGS};
    \label{leg:distribRHS:S2GS}

    \addplot+ [p4L1, iterative, forget plot]
    table [x={nDofs}, y={eta}] {\SfourAFEM};
    \label{leg:distribRHS:S4AFEM}

    \addplot+ [p4L3, iterative, forget plot]
    table [x={nDofs}, y={eta}] {\SfourID};
    \label{leg:distribRHS:S4ID}

    \addplot+ [p4L2, iterative, forget plot]
    table [x={nDofs}, y={eta}] {\SfourGS};
    \label{leg:distribRHS:S4GS}

    \drawslopetriangle{1}{7e5}{1.7e-6}
    \drawslopetriangle{2}{2e5}{2e-9}

  \end{loglogaxis}
\end{tikzpicture}

%% file: plots/distribRHS/disitribRHStime.tex
\begin{tikzpicture}[>=stealth]


  \pgfplotstableread[col sep=comma]{plots/distribRHS/p2_L1_K0_CcardInf_theta50_solver-multigrid-lowOrderVcycle_lambda0.1_smoother-_nStepsInf.csv}\StwoAFEM

  \pgfplotstableread[col sep=comma]{plots/distribRHS/p2_L5_K5_Ccard10_theta50_solver-multigrid-lowOrderVcycle_lambda0.1_smoother-gaussSeidel_nStepsInf.csv}\StwoGS

  \pgfplotstableread[col sep=comma]{plots/distribRHS/p2_L5_K5_Ccard10_theta50_solver-multigrid-lowOrderVcycle_lambda0.1_smoother-identity_nStepsInf.csv}\StwoID

  \pgfplotstableread[col sep=comma]{plots/distribRHS/p4_L1_K0_CcardInf_theta50_solver-multigrid-lowOrderVcycle_lambda0.1_smoother-_nStepsInf.csv}\SfourAFEM

  \pgfplotstableread[col sep=comma]{plots/distribRHS/p4_L5_K5_Ccard10_theta50_solver-multigrid-lowOrderVcycle_lambda0.1_smoother-gaussSeidel_nStepsInf.csv}\SfourGS

  \pgfplotstableread[col sep=comma]{plots/distribRHS/p4_L5_K5_Ccard10_theta50_solver-multigrid-lowOrderVcycle_lambda0.1_smoother-identity_nStepsInf.csv}\SfourID

  \input{plots/pgfplotStyleLietz.tex}

  \begin{loglogaxis}[%
      width            = \convergenceWidth,%
      xlabel           = {median cumulative runtime [s]},%
      ylabel           = {error estimator $\eta_\ell(u_\ell^\kk)$},%
      ymajorgrids      = true,%
      font             = \footnotesize,%
      grid style       = {%
          densely dotted,%
          semithick%
        },%
      legend style     = {%
          legend pos  = south west,%
          font = \scriptsize%
        },%
    ]


    \addplot+ [p2L1, iterative, forget plot]
    table [x={runtime}, y={eta}] {\StwoAFEM};

    \addplot+ [p2L3, iterative, forget plot]
    table [x={runtime}, y={eta}] {\StwoID};

    \addplot+ [p2L2, iterative, forget plot]
    table [x={runtime}, y={eta}] {\StwoGS};

    \addplot+ [p4L1, iterative, forget plot]
    table [x={runtime}, y={eta}] {\SfourAFEM};

    \addplot+ [p4L3, iterative, forget plot]
    table [x={runtime}, y={eta}] {\SfourID};

    \addplot+ [p4L2, iterative, forget plot]
    table [x={runtime}, y={eta}] {\SfourGS};

    \drawswappedslopetriangle{1}{2e3}{1.4e-4}
    \drawswappedslopetriangle{2}{2e3}{4e-7}

  \end{loglogaxis}
\end{tikzpicture}

%% file: plots/distribRHS/distribRHSLegend.tex
\begin{tikzpicture}[>=stealth]

  \input{plots/pgfplotStyleLietz.tex}

  \matrix [
    matrix of nodes,
    anchor = south,
    font = \scriptsize,
    column 1/.style={anchor=base east},
  ] at (0,0) {
     & AFEM ($L=1$)
     & S-AFEM+ID ($L=5$)
     & S-AFEM+GS ($L=5=K$)
    \\
    \hline                          \\
    \(p=2\)
     & \ref*{leg:distribRHS:S2AFEM}
     & \ref*{leg:distribRHS:S2ID}
     & \ref*{leg:distribRHS:S2GS}
    \\
    \(p=4\)
     & \ref*{leg:distribRHS:S4AFEM}
     & \ref*{leg:distribRHS:S4ID}
     & \ref*{leg:distribRHS:S4GS}
    \\
  };
\end{tikzpicture}

%% file: tables/distribRHScompTime.tex
{\footnotesize
\begin{booktabs}{
    colspec={cccccccc|c},
    row{1}={font=\bfseries},
    column{1}={font=\bfseries},
  }
  \toprule
  S-AFEM  & & \SetCell[c=2]{c} Richardson && \SetCell[c=2]{c} Gauss--Seidel && \SetCell[c=2]{c} PCG-iChol && \SetCell{c} AFEM \\
  \cmidrule[lr]{3-4} \cmidrule[lr]{5-6} \cmidrule[lr]{7-8} \cmidrule[lr]{9-9}
  & &  $L=5$     & $L=10$    & $L=5$     & $L=10$  & $L=5$     & $L=10$   & $L=1$  \\
  \midrule

  $p=2$ & $K=5$ & 53.13 & 61.57 & 51.51 & 44.47 & 46.76 & \hb{41.15} & 72.01 \\
  $p=2$ & $K=10$ & 64.11 & 80.37 & 44.61 & \hb{39.11} & 40.44 & 45.47 & 67.20 \\
  \midrule

  $p=3$ & $K=5$ & 21.06 & 55.81 & 22.13 & 19.42 & 21.63 & \hb{18.23} & 44.25 \\
  $p=3$ & $K=10$ & \hb{15.33} & 23.50 & 16.87 & 22.01 & 19.43 & 16.73 & 36.77 \\
  \midrule

  $p=4$ & $K=5$ & 58.27 & 106.28 & 23.67 & 41.31 & \hb{21.26} & 31.42 & 59.03 \\
  $p=4$ & $K=10$ & 21.04 & 53.42 & 3.18 & \hb{2.86} & 4.16 & 4.62 & 10.23 \\
  \bottomrule

\end{booktabs}
}

%% file: tables/KelloggcompTime.tex
{\footnotesize
\begin{booktabs}{
    colspec={c|cc|cc|cc|cc},
    row{1}={font=\bfseries},
    column{1}={font=\bfseries},
  }
  \toprule
  & \SetCell[c=2]{c} $p=1$ && \SetCell[c=2]{c} $p=2$ && \SetCell[c=2]{c} $p=3$ && \SetCell[c=2]{c} $p=4$ \\
  \cmidrule[lr]{2-3} \cmidrule[lr]{4-5} \cmidrule[lr]{6-7} \cmidrule[lr]{8-9}
  &       $L=5$ & $L=10$
  & $L=5$ & $L=10$
  & $L=5$ & $L=10$
  & $L=5$ & $L=10$ \\
  \midrule

  GS ($K=10$) & 1.24 & 0.40 & 0.99 & 0.80 & \hb{2.27} & 1.83 & 1.12 & 0.93 \\
  GS ($K=20$) & 1.18 & 0.51 & 0.98 & 0.90 & \hb{2.27} & 1.92 & 1.12 & 0.97 \\
  GS ($K=30$) & 1.63 & 0.59 & 1.02 & 0.87 & \hb{2.30} & 1.96 & 1.13 & 0.99 \\
  PCG-iChol ($K=10$) & 1.74 & 1.18 & 1.52 & 1.48 & 3.18 & \hb{3.43} & 1.63 & 1.51 \\
  PCG-iChol ($K=20$) & 2.02 & \hy{2.22} & 3.79 & 4.70 & 6.88 & \hg{8.04} & 3.22 & \hy{3.66} \\
  PCG-iChol ($K=30$) & 2.18 & 1.77 & 4.31 & \hy{5.18} & 6.01 & \hb{6.94} & 2.65 & 2.96 \\
  \bottomrule
\end{booktabs}
}

%% file: plots/Kelloggmeshes/Kellogg_mesh_GS.tex
\begin{tikzpicture}
  \begin{axis}[%
      axis equal image,%
      width=\meshWidth,%
      xmin=-1.05, xmax=1.05,%
      ymin=-1.05, ymax=1.05,%
      font=\footnotesize%
    ]
    \addplot graphics [xmin=-1, xmax=1, ymin=-1, ymax=1]
      {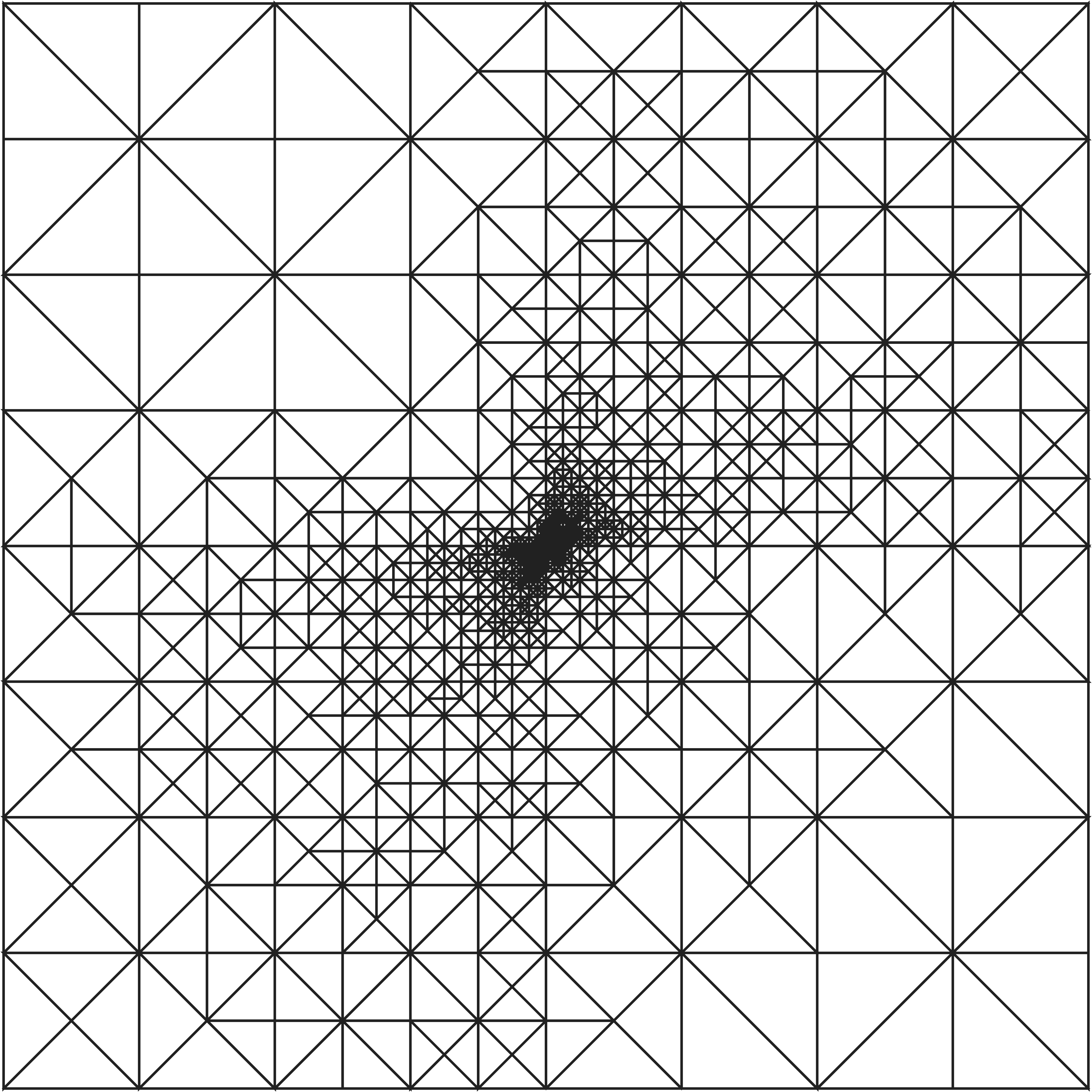};
  \end{axis}
\end{tikzpicture}

%% file: plots/Kelloggmeshes/Kellogg_mesh_PCG.tex
\begin{tikzpicture}
  \begin{axis}[%
      axis equal image,%
      width=\meshWidth,%
      xmin=-1.05, xmax=1.05,%
      ymin=-1.05, ymax=1.05,%
      font=\footnotesize%
    ]
    \addplot graphics [xmin=-1, xmax=1, ymin=-1, ymax=1]
      {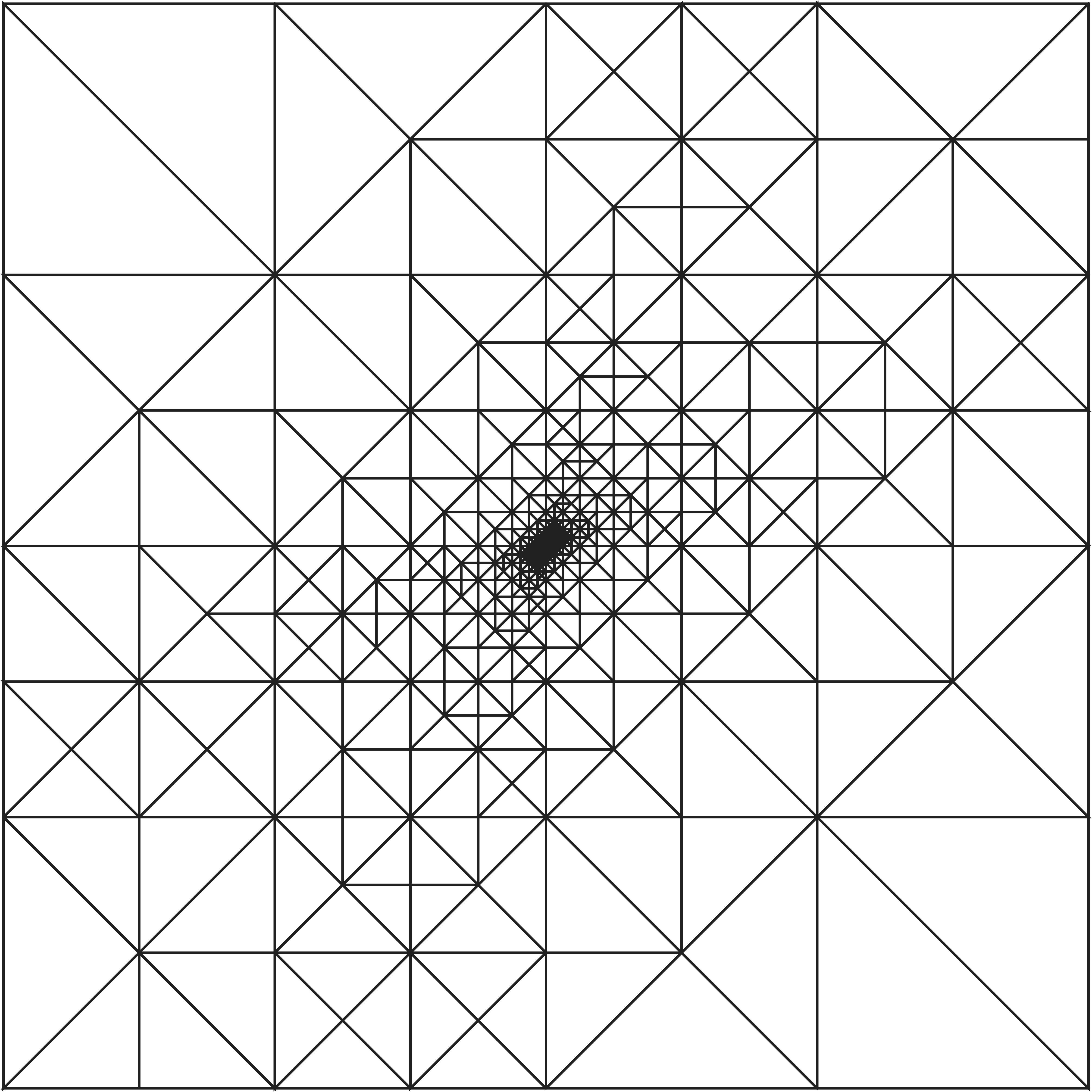};
  \end{axis}
\end{tikzpicture}

%% file: plots/Kelloggmeshes/Kellogg_mesh_AFEM.tex
\begin{tikzpicture}
  \begin{axis}[%
      axis equal image,%
      width=\meshWidth,%
      xmin=-1.05, xmax=1.05,%
      ymin=-1.05, ymax=1.05,%
      font=\footnotesize%
    ]
    \addplot graphics [xmin=-1, xmax=1, ymin=-1, ymax=1]
      {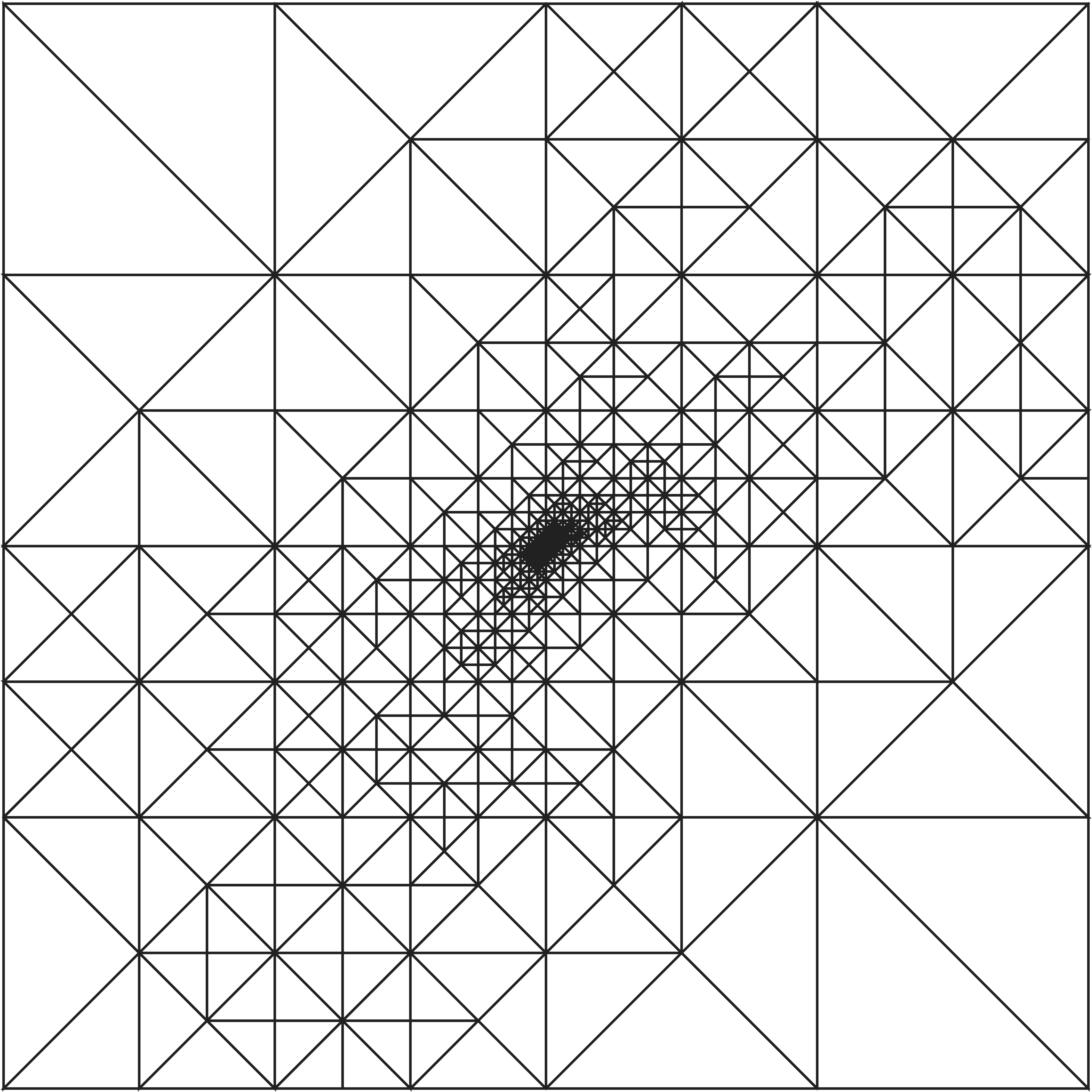};
  \end{axis}
\end{tikzpicture}

%% file: plots/KelloggCcard/KelloggCcardnDofs.tex
\begin{tikzpicture}[>=stealth]


  \pgfplotstableread[col sep=comma]{plots/KelloggCcard/p2_L10_K10_Ccard10_theta50_solver-multigrid-lowOrderVcycle_lambda0.001_smoother-gaussSeidel_nStepsInf.csv}\StwoLtenOn

  \pgfplotstableread[col sep=comma]{plots/KelloggCcard/p2_L10_K10_CcardInf_theta50_solver-multigrid-lowOrderVcycle_lambda0.001_smoother-gaussSeidel_nStepsInf.csv}\StwoLtenOff

  \pgfplotstableread[col sep=comma]{plots/KelloggCcard/p4_L10_K10_Ccard10_theta50_solver-multigrid-lowOrderVcycle_lambda0.001_smoother-gaussSeidel_nStepsInf.csv}\SfourLtenOn

  \pgfplotstableread[col sep=comma]{plots/KelloggCcard/p4_L10_K10_CcardInf_theta50_solver-multigrid-lowOrderVcycle_lambda0.001_smoother-gaussSeidel_nStepsInf.csv}\SfourLtenOff

  \input{plots/pgfplotStyleLietz.tex}

  \begin{loglogaxis}[%
      width            = \convergenceWidth,%
      xlabel           = {$\dim\XX_\ell$},%
      ylabel           = {error $\enorm{u^\star-u_\ell^\kk}$},%
      ymajorgrids      = true,%
      font             = \footnotesize,%
      xmin             = 10.1,
      grid style       = {%
          densely dotted,%
          semithick%
        },%
      legend style     = {%
          legend pos  = south west,%
          font = \scriptsize%
        },%
    ]

    \coordinate (legend) at (axis description cs:0.31,-0.01);


    \addplot+ [p1L2, iterative, forget plot]
    table [x={nDofs}, y={error}] {\StwoLtenOn};
    \label{leg:distribRHS:S2L10on}

    \addplot+ [p1L3, iterative, forget plot]
    table [x={nDofs}, y={error}] {\StwoLtenOff};
    \label{leg:distribRHS:S2L10off}

    \addplot+ [p2L2, iterative, forget plot]
    table [x={nDofs}, y={error}] {\SfourLtenOn};
    \label{leg:distribRHS:S4L10on}

    \addplot+ [p2L3, iterative, forget plot]
    table [x={nDofs}, y={error}] {\SfourLtenOff};
    \label{leg:distribRHS:S4L10off}

    \drawswappedslopetriangle{1}{1.2e6}{1.5e-2}
    \drawslopetriangle{2}{4.5e4}{4e-6}

  \end{loglogaxis}
  \matrix [
    matrix of nodes,
    anchor=south,
    font=\scriptsize,
    line join = round,
    line cap=round
  ] at (legend) {
    $\Ccard\!=\!\!\!$ & $10$         & $+\infty$           \\
    $p\!=\!2$     & \ref{leg:distribRHS:S2L10on} & \ref{leg:distribRHS:S2L10off} \\
    $p\!=\!4$     & \ref{leg:distribRHS:S4L10on} & \ref{leg:distribRHS:S4L10off} \\
  };
\end{tikzpicture}

%% file: plots/KelloggCcard/KelloggCcardnTime.tex
\begin{tikzpicture}[>=stealth]


  \pgfplotstableread[col sep=comma]{plots/KelloggCcard/p2_L10_K10_Ccard10_theta50_solver-multigrid-lowOrderVcycle_lambda0.001_smoother-gaussSeidel_nStepsInf.csv}\StwoLtenOn

  \pgfplotstableread[col sep=comma]{plots/KelloggCcard/p2_L10_K10_CcardInf_theta50_solver-multigrid-lowOrderVcycle_lambda0.001_smoother-gaussSeidel_nStepsInf.csv}\StwoLtenOff

  \pgfplotstableread[col sep=comma]{plots/KelloggCcard/p4_L10_K10_Ccard10_theta50_solver-multigrid-lowOrderVcycle_lambda0.001_smoother-gaussSeidel_nStepsInf.csv}\SfourLtenOn

  \pgfplotstableread[col sep=comma]{plots/KelloggCcard/p4_L10_K10_CcardInf_theta50_solver-multigrid-lowOrderVcycle_lambda0.001_smoother-gaussSeidel_nStepsInf.csv}\SfourLtenOff

  \input{plots/pgfplotStyleLietz.tex}

  \begin{loglogaxis}[%
      width            = \convergenceWidth,%
      xlabel           = {median cumulative runtime [s]},%
      ylabel           = {error $\enorm{u^\star-u_\ell^\kk}$},%
      ymajorgrids      = true,%
      font             = \footnotesize,%
      xmin             = 9e-3,
      grid style       = {%
          densely dotted,%
          semithick%
        },%
      legend style     = {%
          legend pos  = south west,%
          font = \scriptsize%
        },%
    ]

    \coordinate (legend) at (axis description cs:0.31,-0.01);


    \addplot+ [p1L2, iterative, forget plot]
    table [x={runtime}, y={error}] {\StwoLtenOn};

    \addplot+ [p1L3, iterative, forget plot]
    table [x={runtime}, y={error}] {\StwoLtenOff};

    \addplot+ [p2L2, iterative, forget plot]
    table [x={runtime}, y={error}] {\SfourLtenOn};

    \addplot+ [p2L3, iterative, forget plot]
    table [x={runtime}, y={error}] {\SfourLtenOff};

    \drawswappedslopetriangle{1}{9e2}{1.5e-2}
    \drawslopetriangle{2}{4.5e1}{5e-6}

  \end{loglogaxis}
  \matrix [
    matrix of nodes,
    anchor=south,
    font=\scriptsize,
    line join = round,
    line cap=round
  ] at (legend) {
    $\Ccard\!=\!\!\!$ & $10$                         & $+\infty$                     \\
    $p\!=\!2$         & \ref{leg:distribRHS:S2L10on} & \ref{leg:distribRHS:S2L10off} \\
    $p\!=\!4$         & \ref{leg:distribRHS:S4L10on} & \ref{leg:distribRHS:S4L10off} \\
  };
\end{tikzpicture}

%% file: plots/convDiffmeshes/convDiff_mesh_L5.tex
\begin{tikzpicture}
  \begin{axis}[%
      axis equal image,%
      width=\meshWidth,%
      xmin=-1.05, xmax=1.05,%
      ymin=-1.05, ymax=1.05,%
      font=\footnotesize%
    ]
    \addplot graphics [xmin=-1, xmax=1, ymin=-1, ymax=1]
      {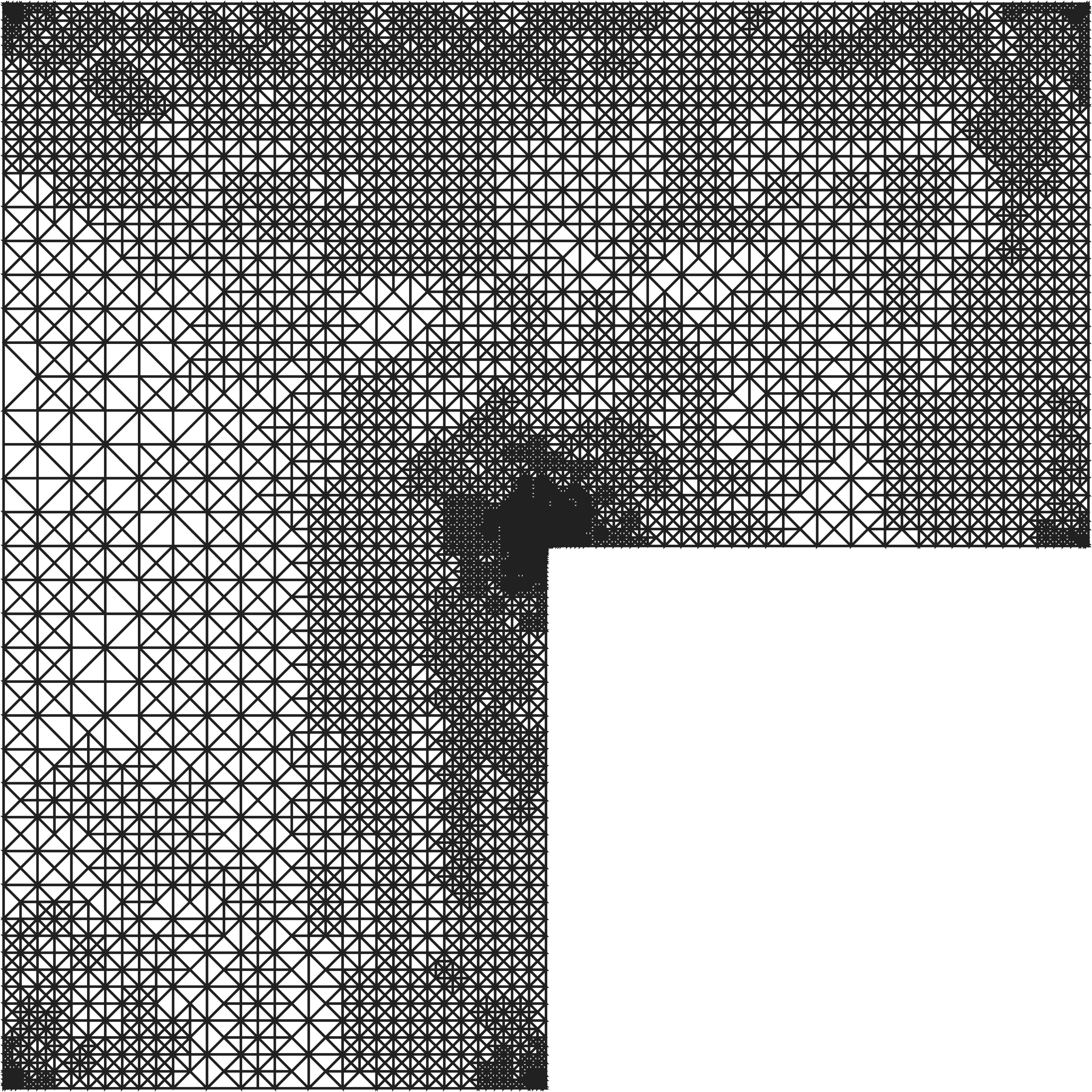};
  \end{axis}
\end{tikzpicture}

%% file: plots/convDiffmeshes/convDiff_mesh_L10.tex
\begin{tikzpicture}
  \begin{axis}[%
      axis equal image,%
      width=\meshWidth,%
      xmin=-1.05, xmax=1.05,%
      ymin=-1.05, ymax=1.05,%
      font=\footnotesize%
    ]
    \addplot graphics [xmin=-1, xmax=1, ymin=-1, ymax=1]
      {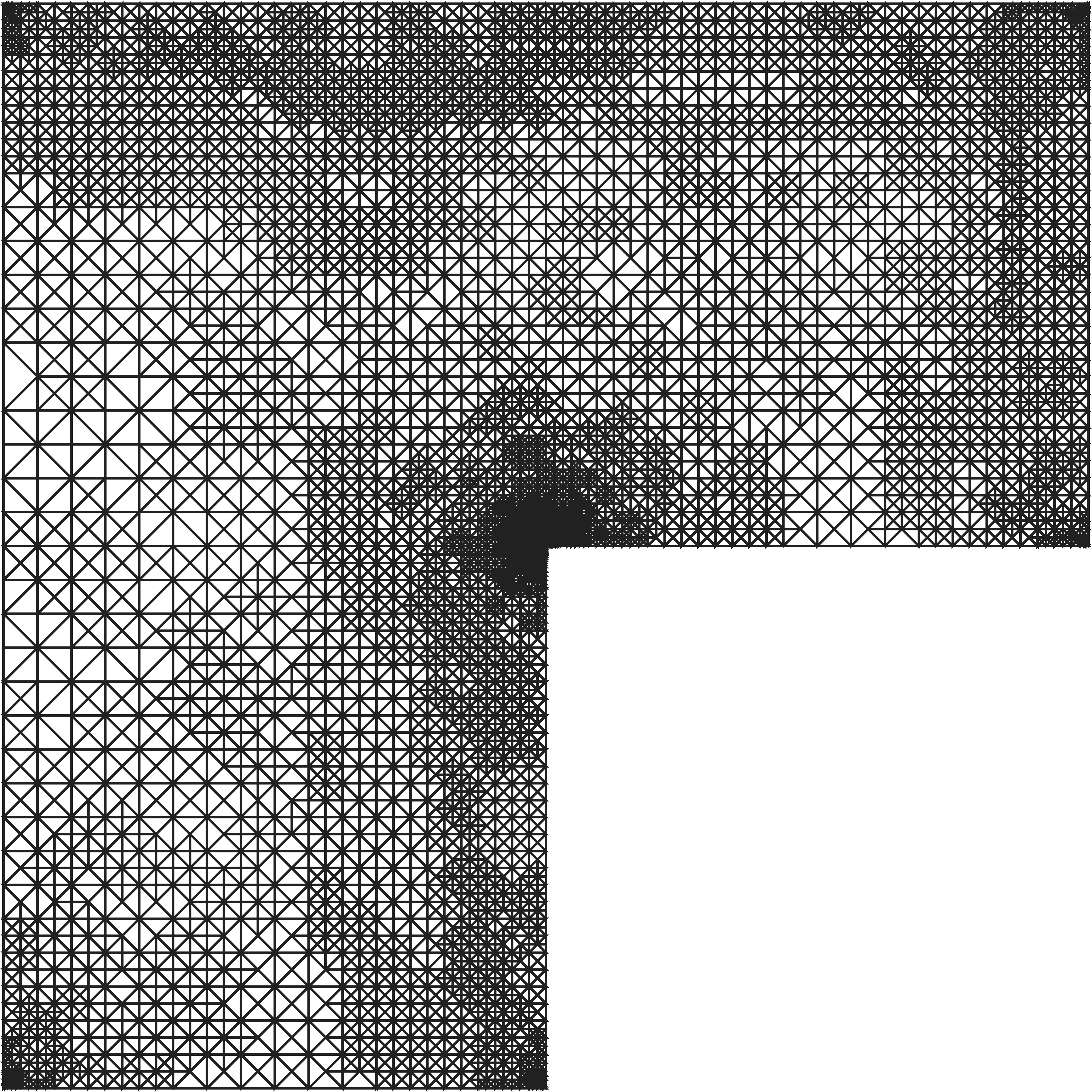};
  \end{axis}
\end{tikzpicture}

%% file: plots/convDiffmeshes/convDiff_mesh_AFEM.tex
\begin{tikzpicture}
  \begin{axis}[%
      axis equal image,%
      width=\meshWidth,%
      xmin=-1.05, xmax=1.05,%
      ymin=-1.05, ymax=1.05,%
      font=\footnotesize%
    ]
    \addplot graphics [xmin=-1, xmax=1, ymin=-1, ymax=1]
      {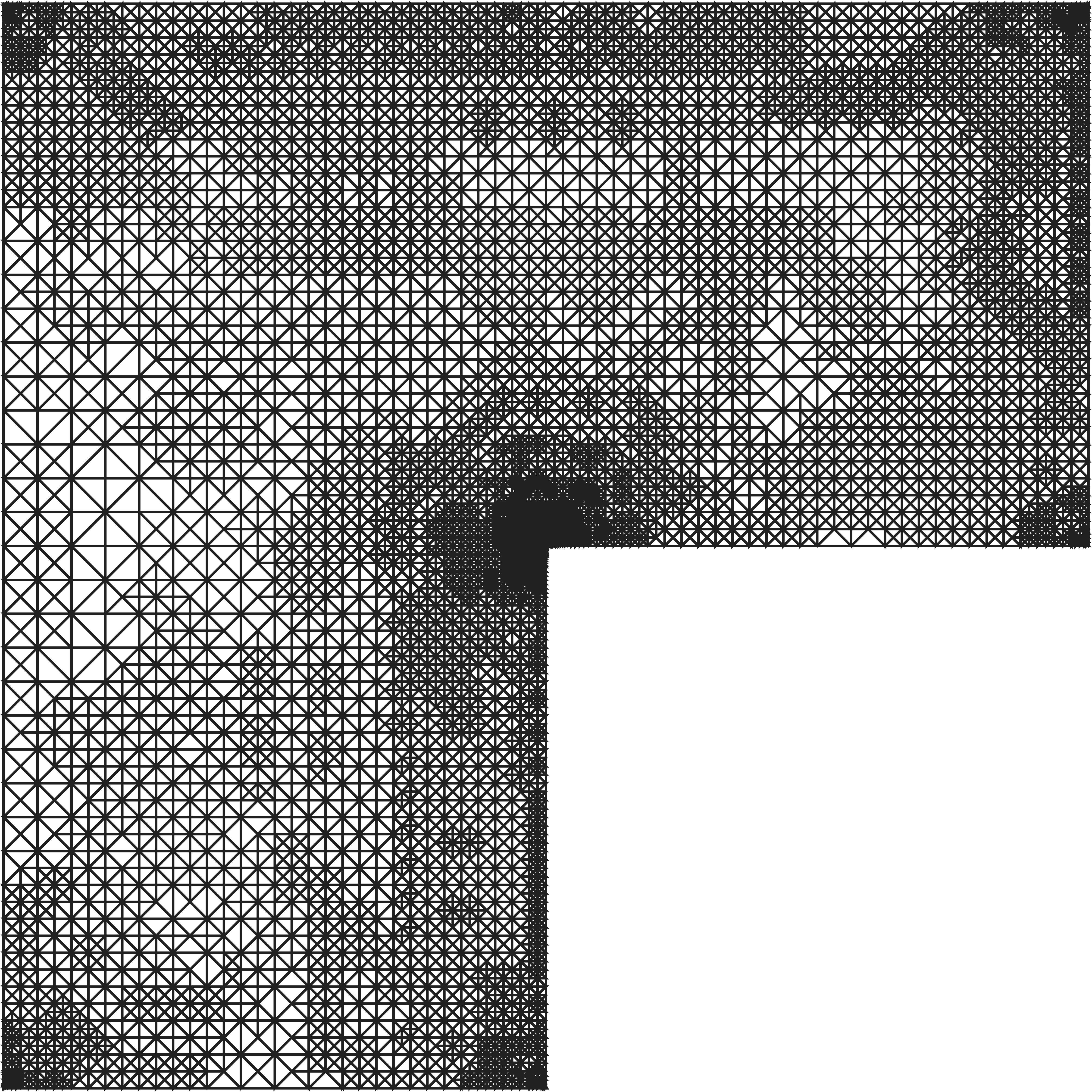};
  \end{axis}
\end{tikzpicture}

%% file: plots/convDiff/convDiff_runtime.tex
\begin{tikzpicture}[>=stealth]


  \pgfplotstableread[col sep=comma]{plots/convDiff/p3_L1_K0_CcardInf_theta50_solver-multigrid-lowOrderVcycle_lambda0.05_smoother-_deltaEx_0.5_deltaInt_0_JEx_2_JInt_0_nStepsInf.csv}\SthreeAFEM

  \pgfplotstableread[col sep=comma]{plots/convDiff/p3_L5_K3_Ccard10_theta50_solver-multigrid-lowOrderVcycle_lambda0.05_smoother-pcg-iChol_deltaEx_0.5_deltaInt_0.5_JEx_2_JInt_4_nStepsInf.csv}\SthreeLfive

  \pgfplotstableread[col sep=comma]{plots/convDiff/p3_L10_K3_Ccard10_theta50_solver-multigrid-lowOrderVcycle_lambda0.05_smoother-pcg-iChol_deltaEx_0.5_deltaInt_0.5_JEx_2_JInt_4_nStepsInf.csv}\SthreeLten

  \input{plots/pgfplotStyleLietz.tex}

  \begin{loglogaxis}[%
      width            = \convergenceWidth,%
      xlabel           = {median cumulative runtime [s]},%
      ylabel           = {error estimator $\eta_\ell(u_\ell^\kk)$},%
      ymajorgrids      = true,%
      font             = \footnotesize,%
      grid style       = {%
          densely dotted,%
          semithick%
        },%
      legend style     = {%
          legend pos  = south west,%
          font = \scriptsize%
        },%
    ]


    \addplot+ [p3L1, iterative, forget plot]
    table [x={runtime}, y={eta}] {\SthreeAFEM};

    \addplot+ [p3L2, iterative, forget plot]
    table [x={runtime}, y={eta}] {\SthreeLfive};

    \addplot+ [p3L3, iterative, forget plot]
    table [x={runtime}, y={eta}] {\SthreeLten};

    \addlegendimage{iterative, p3L1}
    \addlegendentry{$L=1$}
    \addlegendimage{iterative, p3L2}
    \addlegendentry{$L=5$}
    \addlegendimage{iterative, p3L3}
    \addlegendentry{$L=10$}

    \drawswappedslopetriangle{1.5}{4e2}{9e-6}

  \end{loglogaxis}
\end{tikzpicture}

%% file: plots/convDiff/conv_diff_solverIterations.tex
\begin{tikzpicture}[>=stealth]


  \pgfplotstableread[col sep=comma]{plots/convDiff/p3_L1_K0_CcardInf_theta50_solver-multigrid-lowOrderVcycle_lambda0.05_smoother-_deltaEx_0.5_deltaInt_0_JEx_2_JInt_0_nStepsInf.csv}\SthreeAFEM

  \pgfplotstableread[col sep=comma]{plots/convDiff/p3_L5_K3_Ccard10_theta50_solver-multigrid-lowOrderVcycle_lambda0.05_smoother-pcg-iChol_deltaEx_0.5_deltaInt_0.5_JEx_2_JInt_4_nStepsInf.csv}\SthreeLfive

  \pgfplotstableread[col sep=comma]{plots/convDiff/p3_L10_K3_Ccard10_theta50_solver-multigrid-lowOrderVcycle_lambda0.05_smoother-pcg-iChol_deltaEx_0.5_deltaInt_0.5_JEx_2_JInt_4_nStepsInf.csv}\SthreeLten

  \input{plots/pgfplotStyleLietz.tex}

  \begin{semilogxaxis}[%
      width            = \convergenceWidth,%
      xlabel           = {$\dim\XX_\ell$},%
      ylabel           = {number of solver iterations $\kk[\ell]$},%
      ymajorgrids      = true,%
      font             = \footnotesize,%
      grid style       = {%
          densely dotted,%
          semithick%
        },%
      legend style     = {%
          legend pos  = north west,%
          font = \scriptsize%
        },%
    ]


    \addplot+ [p3L1, iterative, forget plot]
    table [x={nDofs}, y={solverIterations}] {\SthreeAFEM};

    \addplot+ [p3L2, iterative, forget plot, each nth point=5]
    table [x={nDofs}, y={solverIterations}] {\SthreeLfive};

    \addplot+ [p3L3, iterative, forget plot, each nth point=10]
    table [x={nDofs}, y={solverIterations}] {\SthreeLten};

  \end{semilogxaxis}
\end{tikzpicture}

%% file: tables/convDiffcompTime.tex
{\footnotesize
\begin{booktabs}{
    colspec={ccccccc|c},
    row{1}={font=\bfseries},
    column{1}={font=\bfseries},
  }
  \toprule
  S-AFEM  & \SetCell[c=3]{c} $\boldsymbol{L=5}$ &&& \SetCell[c=3]{c} $\boldsymbol{L=10}$ &&& \SetCell{c} AFEM  \\
  \cmidrule[lr]{2-4} \cmidrule[lr]{5-7} 
  &  $K=1$     & $K=3$    & $K=5$     &  $K=1$     & $K=3$    & $K=5$    &  $\boldsymbol{L=1}$ \\
  \midrule

  $p=2$ & 6.86 & 7.33 & 5.28 & 10.26 & \hb{5.17} & 5.67 & 11.50 \\
  $p=3$ & 1.44 & 1.56 & 0.90 & 1.11 & \hb{0.77} & 2.05 & 3.32 \\
  $p=4$ & 1.41 & 1.12 & 1.59 & 2.70 & \hb{0.39} & 0.64 & 6.48 \\
  \bottomrule

\end{booktabs}
}

%% file: plots/ficheraSolutionMesh/ficheraSolution.tex
\begin{tikzpicture}
  \input{plots/pgfplotStyleLietz.tex}
  \pgfplotsset{/pgf/number format/fixed}
  \begin{axis}[%
      width=\convergenceWidth,%
      xmin= -1.05, xmax=1.05,%
      ymin= -1.05, ymax=1.05,%
      zmin= -1.05 , zmax=1.05
      ,%
      font=\footnotesize,
      point meta min=0,
      point meta max=1.31562719e+00,%
      colorbar,%
      colorbar style={%
          name = cb,
          font=\scriptsize,%
          ytick distance = 0.25,
          width=1.5mm,
        },%
    ]
    \addplot3 graphics [%
        points={%
            (-1, 1, -1) => (6.25, 14.43)
            (1, 1, -1) => (122.41, 0.37)
            (1, -1, -1) => (164.11, 39.75)
            (-1, -1, 1) => (48.11, 169.64)
          }%
      ]
      {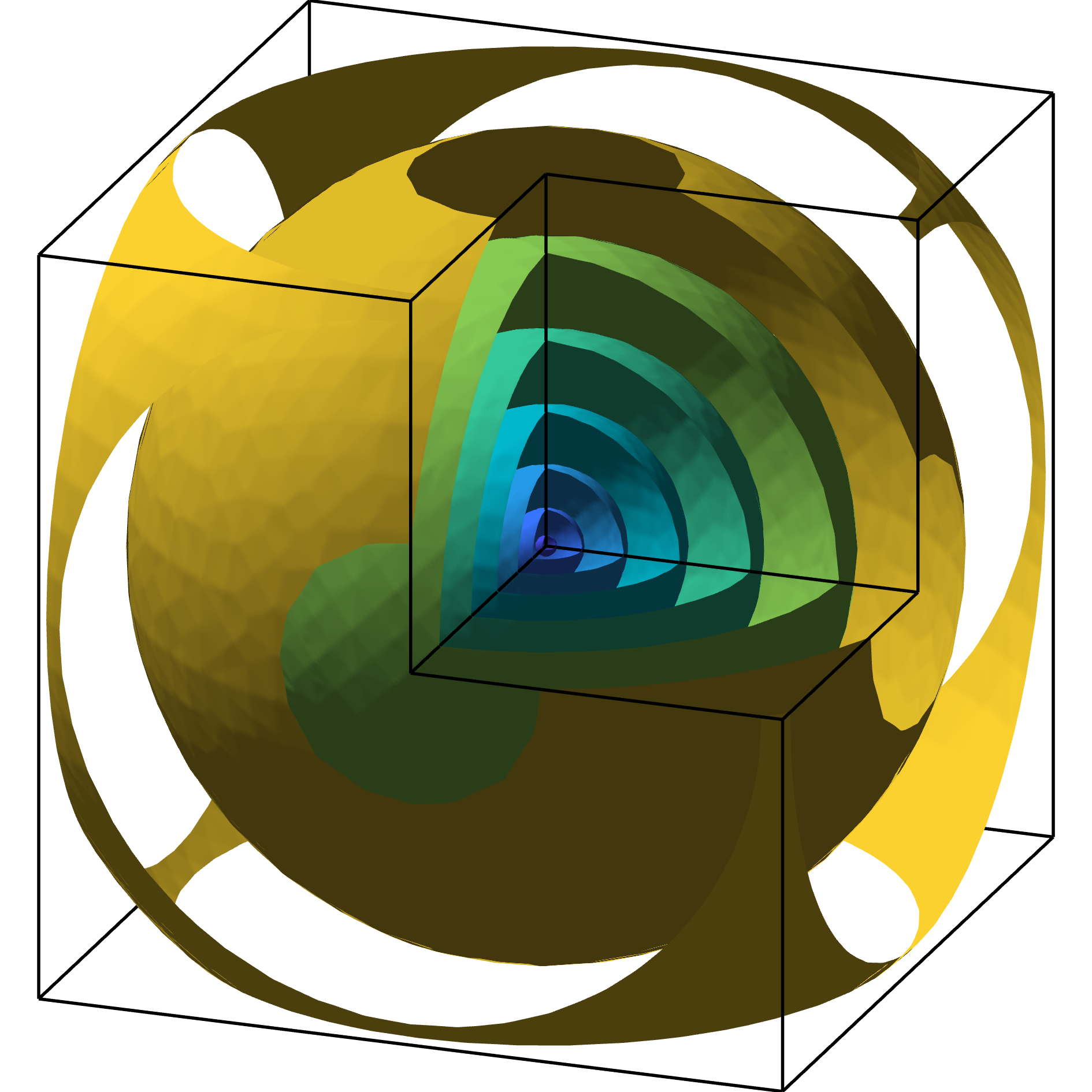};
  \end{axis}
\end{tikzpicture}

%% file: plots/ficheraSolutionMesh/ficheraMesh.tex
\begin{tikzpicture}
  \input{plots/pgfplotStyleLietz.tex}
  \pgfplotsset{/pgf/number format/fixed}
  \begin{axis}[%
      width=\convergenceWidth,%
      xmin= -1.05, xmax=1.05,%
      ymin= -1.05, ymax=1.05,%
      zmin= -1.05 , zmax=1.05
      ,%
      font=\footnotesize,
    ]
    \addplot3 graphics [%
        points={%
            (-1, 1, -1) => (6.25, 14.43)
            (1, 1, -1) => (122.41, 0.37)
            (1, -1, -1) => (164.11, 39.75)
            (-1, -1, 1) => (48.11, 169.64)
          }%
      ]
      {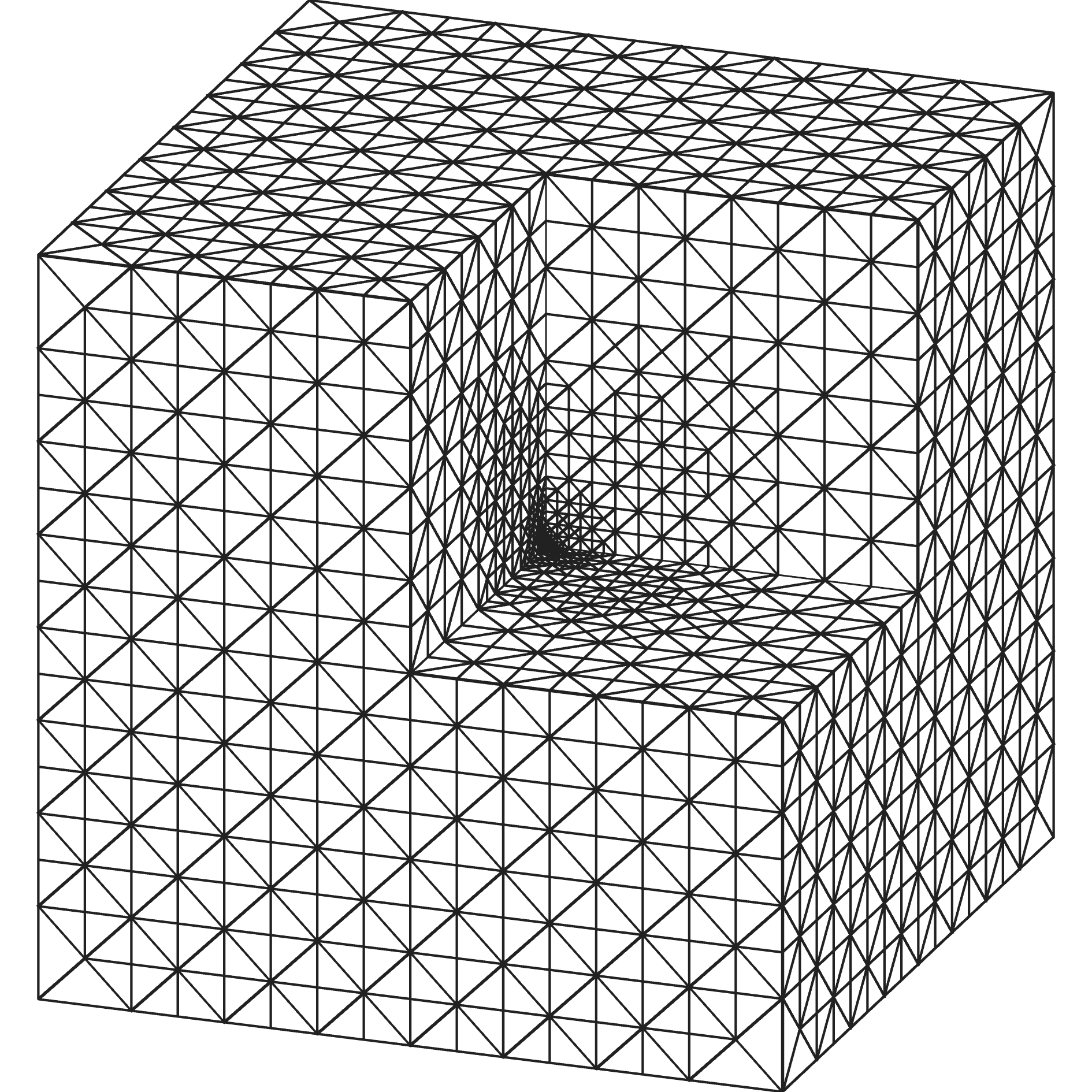};
  \end{axis}
\end{tikzpicture}

%% file: plots/ficheraSpeedUp/Fichera_speedUp.tex
\begin{tikzpicture}[>=stealth]


  \pgfplotstableread[col sep=comma]{plots/ficheraSpeedUp/Fig_Fichera_S1_speedup_safem_L5_K5_GaussSeidel_median_speedUpFactors.csv}\GSLfiveKfive

  \pgfplotstableread[col sep=comma]{plots/ficheraSpeedUp/Fig_Fichera_S1_speedup_safem_L5_K10_GaussSeidel_median_speedUpFactors.csv}\GSLfiveKten

  \pgfplotstableread[col sep=comma]{plots/ficheraSpeedUp/Fig_Fichera_S1_speedup_safem_L10_K5_GaussSeidel_median_speedUpFactors.csv}\GSLtenKfive

  \pgfplotstableread[col sep=comma]{plots/ficheraSpeedUp/Fig_Fichera_S1_speedup_safem_L10_K10_GaussSeidel_median_speedUpFactors.csv}\GSLtenKten

  \pgfplotstableread[col sep=comma]{plots/ficheraSpeedUp/Fig_Fichera_S1_speedup_safem_L5_K5_PCG_median_speedUpFactors.csv}\PCGLfiveKfive

  \pgfplotstableread[col sep=comma]{plots/ficheraSpeedUp/Fig_Fichera_S1_speedup_safem_L5_K10_PCG_median_speedUpFactors.csv}\PCGLfiveKten

  \pgfplotstableread[col sep=comma]{plots/ficheraSpeedUp/Fig_Fichera_S1_speedup_safem_L10_K5_PCG_median_speedUpFactors.csv}\PCGLtenKfive

  \pgfplotstableread[col sep=comma]{plots/ficheraSpeedUp/Fig_Fichera_S1_speedup_safem_L10_K10_PCG_median_speedUpFactors.csv}\PCGLtenKten

  \input{plots/pgfplotStyleLietz.tex}

  \begin{semilogxaxis}[%
      width            = \convergenceWidth,%
      xlabel           = {$1/\enorm{u^\star-u_\ell^\kk}$},%
      ylabel           = {speed-up factor $S_\ell$},%
      ymajorgrids      = true,%
      font             = \footnotesize,%
      grid style       = {%
          densely dotted,%
          semithick%
        },%
      legend style     = {%
          legend pos  = south west,%
          font = \scriptsize%
        },%
    ]


    \addplot+ [p1L2, iterative, forget plot, x filter/.code={
          \ifnum\coordindex=0
            \def\pgfmathresult{}
          \fi
        }, each nth point=5]
    table [x={errorInv}, y={speedupFactor}] {\GSLfiveKfive};
    \label{leg:ficheraSpeed:GSL5K5}

    \addplot+ [p1L3, iterative, forget plot, x filter/.code={
          \ifnum\coordindex=0
            \def\pgfmathresult{}
          \fi
        }, each nth point=5]
    table [x={errorInv}, y={speedupFactor}] {\GSLfiveKten};
    \label{leg:ficheraSpeed:GSL5K10}

    \addplot+ [p2L2, iterative, forget plot, x filter/.code={
          \ifnum\coordindex=0
            \def\pgfmathresult{}
          \fi
        }, each nth point=10]
    table [x={errorInv}, y={speedupFactor}] {\GSLtenKfive};
    \label{leg:ficheraSpeed:GSL10K5}

    \addplot+ [p2L3, iterative, forget plot, x filter/.code={
          \ifnum\coordindex=0
            \def\pgfmathresult{}
          \fi
        }, each nth point=10]
    table [x={errorInv}, y={speedupFactor}] {\GSLtenKten};
    \label{leg:ficheraSpeed:GSL10K10}

    \addplot+ [p3L2, iterative, forget plot, x filter/.code={
          \ifnum\coordindex=0
            \def\pgfmathresult{}
          \fi
        }, each nth point=5]
    table [x={errorInv}, y={speedupFactor}] {\PCGLfiveKfive};
    \label{leg:ficheraSpeed:PCGL5K5}

    \addplot+ [p3L3, iterative, forget plot, x filter/.code={
          \ifnum\coordindex=0
            \def\pgfmathresult{}
          \fi
        }, each nth point=5]
    table [x={errorInv}, y={speedupFactor}] {\PCGLfiveKten};
    \label{leg:ficheraSpeed:PCGL5K10}

    \addplot+ [p4L2, iterative, forget plot, x filter/.code={
          \ifnum\coordindex=0
            \def\pgfmathresult{}
          \fi
        }, each nth point=10]
    table [x={errorInv}, y={speedupFactor}] {\PCGLtenKfive};
    \label{leg:ficheraSpeed:PCGL10K5}

    \addplot+ [p4L3, iterative, forget plot, x filter/.code={
          \ifnum\coordindex=0
            \def\pgfmathresult{}
          \fi
        }, each nth point=10]
    table [x={errorInv}, y={speedupFactor}] {\PCGLtenKten};
    \label{leg:ficheraSpeed:PCGL10K10}

  \end{semilogxaxis}
\end{tikzpicture}

%% file: plots/ficheraSpeedUp/Fichera_speedUP_legend.tex
\begin{tikzpicture}[>=stealth]

  \input{plots/pgfplotStyleLietz.tex}

  \matrix(m) [
    matrix of nodes,
    nodes in empty cells,
    anchor = center,
    font = \scriptsize,
    column 1/.style={anchor=base east},
  ] at (0,0) {
     &                                   &  &  &
    \\
     & $L=5$
     & $L=10$
     & $L=5$
     & $L=10$
    \\
    \(K=5\)
     & \ref*{leg:ficheraSpeed:GSL5K5}
     & \ref*{leg:ficheraSpeed:GSL10K5}
     & \ref*{leg:ficheraSpeed:PCGL5K5}
     & \ref*{leg:ficheraSpeed:PCGL10K5}
    \\
    \(K=10\)
     & \ref*{leg:ficheraSpeed:GSL5K10}
     & \ref*{leg:ficheraSpeed:GSL10K10}
     & \ref*{leg:ficheraSpeed:PCGL5K10}
     & \ref*{leg:ficheraSpeed:PCGL10K10}
    \\
  };

  \node[font = \scriptsize\bfseries, anchor=base] at ($(m-1-2)!0.5!(m-1-3)$) {Gauss--Seidel};
  \node[font = \scriptsize\bfseries, anchor=base] at ($(m-1-4)!0.5!(m-1-5)$) {PCG-iChol};
  \draw[line width=\arrayrulewidth, color=black] (m.south west |- m-2-1.south west) -- (m-2-5.south east);
  \draw[line width=\arrayrulewidth]
  ([xshift=-10pt]m-1-2.south west) -- ([xshift=10pt]m-1-3.south east);
  \draw[line width=\arrayrulewidth]
  ([xshift=-10pt]m-1-4.south west) -- ([xshift=10pt]m-1-5.south east);

\end{tikzpicture}

%% file: plots/ficheraThresh/ficheraThreshNdofs.tex
\begin{tikzpicture}[>=stealth]

  \pgfplotstableread[col sep=comma]{plots/ficheraThresh/Fig_Fichera_S1_threshold_safem_median.csv}\SAFEM

  \input{plots/pgfplotStyleLietz.tex}

  \begin{loglogaxis}[%
      width            = \convergenceWidth,%
      xlabel           = {$\dim\XX_\ell$},%
      ylabel           = {estimator, error},%
      ymajorgrids      = true,%
      font             = \footnotesize,%
      grid style       = {%
          densely dotted,%
          semithick%
        },%
      legend style     = {%
          legend pos  = south west,%
          font = \scriptsize%
        },%
      log basis x=10,
      log basis y=10,
    ]

    \draw[dashed] (axis cs:2,0.15) -- (axis cs:4.5e5,0.15)
    node[pos=0.15, above] {\scriptsize $\tau=0.15$};

    \addplot+[p1L2, iterative, mark=none, forget plot]
    table[x=ndof, y=eta] {\SAFEM};

    \addplot+[p1L2, iterative,  mark options={fill opacity=0}, only marks, forget plot]
    table[x=ndof, y=eta, restrict expr to domain={\thisrow{isExactLevel}}{0:0}] {\SAFEM};

    \addplot+[p1L2, iterative]
    table[x=ndof, y=eta, restrict expr to domain={\thisrow{isExactLevel}}{1:1}] {\SAFEM};

    \addplot+[p1L3, iterative, mark=none, forget plot]
    table[x=ndof, y=errGrad] {\SAFEM};

    \addplot+[p1L3, iterative, mark options={fill opacity=0}, only marks, forget plot]
    table[x=ndof, y=errGrad, restrict expr to domain={\thisrow{isExactLevel}}{0:0}] {\SAFEM};

    \addplot+[p1L3, iterative]
    table[x=ndof, y=errGrad, restrict expr to domain={\thisrow{isExactLevel}}{1:1}] {\SAFEM};

    \drawswappedslopetriangle{1/3}{7e4}{5e-1}
    \drawslopetriangle{1/3}{8.8e3}{3.1e-2}

    \addlegendimage{iterative, p1L2}
    \addlegendentry{$\eta_\ell(u_\ell^\kk)$}
    \addlegendimage{iterative, p1L3}
    \addlegendentry{$\enorm{u^\star-u_\ell^\kk}$}

  \end{loglogaxis}
\end{tikzpicture}

%% file: tables/FicheraThreshTime.tex
{\footnotesize
\begin{booktabs}{
    colspec={c|c|c},
    row{1}={font=\bfseries},
  }
  \toprule
  &       AFEM & S-AFEM \\
  \midrule

  Initial 20 solves &
  0.00709s &
  0.00800s \\
  Intermediate solves &
  19.2s &
  0.203s \\
  Final solve &
  8.19s &
  8.19s\\
  \midrule
  Intermediate speed-up &
  \SetCell[c=2]{c}94.5  \\
  Total speed-up &
  \SetCell[c=2]{c}3.26  \\
  \bottomrule
\end{booktabs}
}

%% file: safemAppendix.tex

\section{Proofs of \texorpdfstring{\cref{prop:ZarStable,th:Condfullrlinear}}{Proposition 3 and Theorem 12}}\label{sec:appendixA}

\begin{proof}[Proof of \cref{prop:ZarStable}]
  We adapt the reasoning from~\cite[Section~2]{BIM24} to the present setting. By the triangle inequality and the definition of $\Psi_H$, it follows that
  \begin{equation}\label{prop:ZarStable:eq3}
    \enorm{u^\star_H-\Psi_H(v_H)}\leq\enorm{u^\star_H-\ZZ_H^\delta v_H}+\enorm{\ZZ_H^\delta v_H-(\Theta_H(v_H))^{J}(v_H)}.
  \end{equation}
  Since $\ZZ_H^\delta u_H^\star=u_H^\star$, the Lipschitz continuity~\cref{eq:ZarLipschitz} of $\ZZ_H^\delta$ yields
  \begin{equation}\label{prop:ZarStable:eq4}
    \enorm{u^\star_H-\ZZ_H^\delta v_H}\eqreff{eq:ZarLipschitz}{\leq}
    C[\delta]\,\enorm{u^\star_H-v_H}.
  \end{equation}
  For the second term in~\cref{prop:ZarStable:eq3}, uniform stability~\cref{prop:ZarStable:eq1} of $\Theta_H(v_H)$ gives
  \begin{equation}\label{prop:ZarStable:eq5}\begin{aligned}
       & \enorm{\ZZ_H^\delta v_H-(\Theta_H(v_H))^{J}(v_H)}\eqreff{prop:ZarStable:eq1}{\leq}\const{C}{alg}^{J}\,\enorm{\ZZ_H^\delta v_H-v_H} \\
       & \qquad\leq\const{C}{alg}^{J}\,\big(\enorm{u^\star_H-\ZZ_H^\delta v_H}+\enorm{u^\star_H-v_H}\big)
      \eqreff{prop:ZarStable:eq4}{\leq}\const{C}{alg}^{J}(C[\delta]+1)\,\enorm{u^\star_H-v_H}.
    \end{aligned}\end{equation}
  Combining~\cref{prop:ZarStable:eq3,prop:ZarStable:eq5} establishes~\cref{prop:ZarStable:eq2}.
  Finally, an elementary calculation shows that $\const{C}{alg}<1$ and $0<\delta<2\Cell^{2}/\Cbnd^2$ imply $0\leq C[\delta]<1$.
  Thus, for $J\in\N_0$ sufficiently large, $\Psi_H$ is uniformly contractive~\cref{eq:contractivesolver}.
  This concludes the proof.
\end{proof}

\begin{proof}[Proof of \cref{th:Condfullrlinear}]
  Let $\ell\in\SS$ and write $\ell=jL$ with $j\in\N_0$. Contraction~\cref{eq:contractivesolver} of $\Phi_\ell$ yields
  \begin{equation}\label{th:Condfullrlinear:eq1.1}
    \enorm{u_{jL}^\star-u_{jL}^{\kk-1}}\leq\enorm{u_{jL}^\star-u_{jL}^\kk}+\enorm{u_{jL}^\kk-u_{jL}^{\kk-1}}\eqreff{eq:contractivesolver}{\leq}\qalg\,\enorm{u_{jL}^\star-u_{jL}^{\kk-1}}+\enorm{u_{jL}^\kk-u_{jL}^{\kk-1}}.
  \end{equation}
  A rearrangement of the previous estimate together with the stopping criterion~\cref{algorithm:SAFEM:solverstop} implies
  \begin{equation}\label{lem:estimeq:eq2}
    \enorm{u_{jL}^\star-u_{jL}^\kk}\eqreff{eq:contractivesolver}{\leq}\qalg\,\enorm{u_{jL}^\star-u_{jL}^{\kk-1}}\eqreff{th:Condfullrlinear:eq1.1}{\leq}\frac{\qalg}{1-\qalg}\enorm{u_{jL}^\kk-u_{jL}^{\kk-1}}\eqreff{algorithm:SAFEM:solverstop}{\leq}\lambda\frac{\qalg}{1-\qalg}\eta_{jL}(u_{jL}^\kk).
  \end{equation}
  Dörfler marking in \cref{algorithm:SAFEM}\crefrange{algorithm:SAFEM:2}{algorithm:SAFEM:3}, stability~\cref{A1}, and~\cref{lem:estimeq:eq2} give
  \begin{equation*}\begin{aligned}
      \theta^{1/2}\eta_{jL}(u_{jL}^\kk)\leq\eta_{jL}(\MM_{jL}, u_{jL}^\kk) & \eqreff*{A1}{\leqpad}\eta_{jL}(\MM_{jL}, u_{jL}^\star)+\Cstab\,\enorm{u_{jL}^\star-u_{jL}^\kk}                            \\
                                                                           & \eqreff*{lem:estimeq:eq2}{\leqpad}\eta_{jL}(\MM_{jL}, u_{jL}^\star)+\lambda\lambda_{\rm opt}^{-1}\,\eta_{jL}(u_{jL}^\kk).
    \end{aligned}\end{equation*}
  Rearranging the terms, this yields
  \begin{equation}\label{th:Condfullrlinear:eq1}
    (\theta^{1/2} - \lambda\lambda_{\rm opt}^{-1})\,\eta_{jL}(u_{jL}^\kk)\leq\eta_{jL}(\MM_{jL}, u_{jL}^\star).
  \end{equation}
  The combination of~\cref{lem:estimeq:eq1} and~\cref{th:Condfullrlinear:eq1} implies the following Dörfler criterion for the exact discrete solution $u_{jL}^\star$ with parameter $0 < \widetilde{\theta} \coloneq (\theta^{1/2} - \lambda\lambda_{\rm opt}^{-1})(1+\lambda\lambda_{\rm opt}^{-1})^{-1}<1$
  \begin{equation}\label{th:Condfullrlinear:eq2}
    \widetilde{\theta}\,\eta_{jL}(u_{jL}^\star)\eqreff{lem:estimeq:eq1}{\leq}(\theta^{1/2} - \lambda\lambda_{\rm opt}^{-1})\,\eta_{jL}(u_{jL}^\kk)\eqreff{th:Condfullrlinear:eq1}{\leq}\eta_{jL}(\MM_{jL}, u_{jL}^\star)\leq\eta_{jL}(\TT_{jL}\setminus\TT_{(j+1)L}, u_{jL}^\star).
  \end{equation}
  Set $a_j\coloneq\eta_{jL}(u_{jL}^\star)$ and $b_j\coloneq\Cstab\enorm{u_{(j+1)L}^\star-u_{jL}^\star}$.
  The Dörfler criterion~\cref{th:Condfullrlinear:eq2} and the argument used for~\cref{lem:sumsolve:eq0.5} in the proof of \cref{lem:sumsolve} yield
  \begin{equation}\label{th:Condfullrlinear:eq3}
    a_{j+1}\leq q_{\widetilde{\theta}}\,a_j+b_j\quad\text{with } 0<q_{\widetilde{\theta}}\coloneq\big[1-(1-\qred^2)\widetilde{\theta}\big]^{1/2}<1.
  \end{equation}
  Quasi-orthogonality~\cref{A4} and reliability~\cref{A3} show that
  \begin{equation}\label{th:Condfullrlinear:eq4}
    \sum_{j^\prime=j}^{j+N}b_{j^\prime}^2\simeq\sum_{j^\prime=j}^{j+N}\enorm{u_{(j^\prime+1)L}^\star-u_{j^\prime L}^\star}^2\eqreff{A4}{\lesssim}(N+1)^{1-\delta}\enorm{u^\star-u_{jL}^\star}^2\eqreff{A3}{\lesssim}(N+1)^{1-\delta}\,a_j^2.
  \end{equation}
  Thus~\crefrange{th:Condfullrlinear:eq3}{th:Condfullrlinear:eq4} show that the sequences $(a_j)_{j\in\N_0}, (b_j)_{j\in\N_0}$ (extended by zero if $\ellu<+\infty$) satisfy assumptions~\cref{lem:tailsumCrit:eq1} of \cref{lem:tailsumCrit} and thus yield tail summability of $(a_j)_{j\in\N_0}$.
  Using stability~\cref{A1},~\cref{lem:estimeq:eq2}, and equivalence~\cref{lem:estimeq:eq1} from \cref{lem:estimeq}, it follows that
  \begin{equation*}
    \eta_{{jL}}(u_{jL}^\star)\leq\Hr_{jL}^\kk\eqreff{A1}{\lesssim}\enorm{u_{jL}^\star-u_{jL}^\kk}+\eta_{jL}(u_{jL}^\kk)\eqreff{lem:estimeq:eq2}{\lesssim}\eta_{{jL}}(u_{jL}^\kk)\eqreff{lem:estimeq:eq1}{\lesssim}\eta_{jL}(u_{jL}^\star)
  \end{equation*}
  and, hence, $\Hr_{jL}^\kk\simeq\eta_{jL}(u_{jL}^\star)=a_j$ is also tail summable, i.e., there exists a constant $\const{C}{tail}>0$ such that, under the change of indices $\ell=jL\in\SS$,
  \begin{equation}\label{th:Condfullrlinear:eq5}
    \sum_{\ell^\prime\in\SS,\,\ell^\prime>\ell}\Hr_{\ell^\prime}^\kk\leq \const{C}{tail}\,\Hr_\ell^\kk\quad\text{for all }\ell\in\SS.
  \end{equation}
  Finally, the solver contraction~\cref{eq:contractivesolver} and the smoother stability~\cref{eq:stablesolver} ensure that
  \begin{equation}\label{th:Condfullrlinear:eq6}
    \Hr_\ell^{k'}\leq
    \begin{cases}
      \Hr_\ell^{k},                              & \text{for all $(\ell, k), (\ell, k^\prime)\in\QQ$ with $\ell\in\underline{\SS}$ and $k\leq k^\prime$,} \\
      \max\{1, \const{C}{alg}^K\}\,\Hr_\ell^{k}, & \text{for all $(\ell, k), (\ell, k^\prime)\in\QQ$ with $\ell\in\II$ and $k\leq k^\prime$.}
    \end{cases}
  \end{equation}
  With tail summability~\cref{th:Condfullrlinear:eq5} of $\Hr_\ell^\kk$ along the solve levels and quasi-monotonicity \cref{th:Condfullrlinear:eq6} in $k$, Steps~2--5 of the proof of \cref{th:fullrlinear} apply to the situation at hand, except that monotonicity~\cref{th:fullrlinear:eq1} is replaced by quasi-monotonicity~\cref{th:Condfullrlinear:eq6}.
  This concludes the proof.
\end{proof}